\documentclass[11pt]{article}
\usepackage{amsmath,amssymb,amsthm,amscd}

\newtheorem{theorem}{Theorem}[section]
\newtheorem{lemma}[theorem]{Lemma}
\newtheorem{proposition}[theorem]{Proposition}
\newtheorem{corollary}[theorem]{Corollary}

\theoremstyle{definition}
\newtheorem{definition}[theorem]{Definition}
\newtheorem{remark}[theorem]{Remark}

\numberwithin{equation}{section}

\renewcommand{\phi}{\varphi}

\newcommand{\ep}{\varepsilon}

\newcommand{\Ad}{\operatorname{Ad}}
\newcommand{\Aff}{\operatorname{Aff}}
\newcommand{\aInn}{\operatorname{\overline{Inn}}}
\newcommand{\Aut}{\operatorname{Aut}}
\newcommand{\Bott}{\operatorname{Bott}}
\newcommand{\Ext}{\operatorname{Ext}}

\newcommand{\id}{\operatorname{id}}
\newcommand{\Hom}{\operatorname{Hom}}
\newcommand{\Ima}{\operatorname{Im}}
\newcommand{\Ker}{\operatorname{Ker}}
\newcommand{\Lip}{\operatorname{Lip}}
\newcommand{\OrderExt}{\operatorname{OrderExt}}

\newcommand{\N}{\mathbb{N}}
\newcommand{\Z}{\mathbb{Z}}
\newcommand{\Q}{\mathbb{Q}}
\newcommand{\R}{\mathbb{R}}
\newcommand{\C}{\mathbb{C}}
\newcommand{\T}{\mathbb{T}}

\title{$\mathcal{Z}$-stability of crossed products 
by strongly outer actions II}
\author{Hiroki Matui \\
Graduate School of Science \\
Chiba University \\
Inage-ku, Chiba 263-8522, Japan 
\and
Yasuhiko Sato \\
Graduate School of Science \\
Kyoto University \\
Sakyo-ku, Kyoto 606-8502, Japan}
\date{}

\begin{document}
\maketitle

\begin{abstract}
We consider a crossed product of 
a unital simple separable nuclear stably finite 
$\mathcal{Z}$-stable $C^*$-algebra $A$ 
by a strongly outer cocycle action of 
a discrete countable amenable group $\Gamma$. 
Under the assumption that $A$ has finitely many extremal tracial states 
and $\Gamma$ is elementary amenable, we show that 
the twisted crossed product $C^*$-algebra is $\mathcal{Z}$-stable. 
As an application, we also prove that 
all strongly outer cocycle actions of the Klein bottle group on $\mathcal{Z}$ 
are cocycle conjugate to each other. 
This is the first classification result 
for actions of non-abelian infinite groups on stably finite $C^*$-algebras. 
\end{abstract}

%%%%%%%%%%%%%%%%%%%%%%%%%%%%%%%%%%%%%%%%%%%%%%%%%%%%%%%%%%%%
\section{Introduction}

This is a continuation of our previous paper \cite{MS} 
studying $\mathcal{Z}$-stability of crossed product $C^*$-algebras 
and classification of group actions on $\mathcal{Z}$. 
Here $\mathcal{Z}$ denotes the Jiang-Su algebra 
introduced by X. Jiang and H. Su in \cite{JS}, 
and a unital $C^*$-algebra $A$ is said to be $\mathcal{Z}$-stable 
if $A$ is isomorphic to $A\otimes\mathcal{Z}$. 
In Elliott's program to classify nuclear $C^*$-algebras 
via $K$-theoretic invariants 
(see \cite{Rtext} for an introduction to this subject), 
the Jiang-Su algebra $\mathcal{Z}$ has recently become to play a central role. 
In fact, all unital simple separable nuclear $C^*$-algebras classified so far 
by their Elliott invariants are $\mathcal{Z}$-stable. 
Also, $\mathcal{Z}$-stability is known to be closely related to 
other regularity properties, 
such as finite nuclear dimension, strict comparison of positive elements 
and the property (SI) (see \cite{R04IJM,W12Invent,MS2}). 
We refer the reader to \cite{RW10crelle,W11JNG} 
for several characterizations of $\mathcal{Z}$. 

In the first half of the paper, 
we study $\mathcal{Z}$-stability of crossed products 
arising from strongly outer cocycle actions of discrete countable groups. 
Let $A$ be a unital, simple, separable, nuclear, stably finite, 
$\mathcal{Z}$-stable $C^*$-algebra with finitely many extremal tracial states. 
Let $(\alpha,u):\Gamma\curvearrowright A$ be 
a cocycle action of a discrete group $\Gamma$ (see Definition \ref{Defofca}). 
We say that $(\alpha,u)$ is strongly outer 
when its extension to the weak closure in any tracial representation is outer 
(see Definition \ref{so&wR} for the precise definition). 
In this setting, we prove that 
if $\Gamma$ is elementary amenable and $(\alpha,u)$ is strongly outer, 
then the twisted crossed product $A\rtimes_{(\alpha,u)}\Gamma$ 
is $\mathcal{Z}$-stable (Corollary \ref{so>Z}). 
In our previous paper \cite{MS}, we had to make extra hypotheses 
on the $C^*$-algebra $A$ and on the group $\Gamma$ 
in order to deduce the weak Rohlin property (see Definition \ref{so&wR}) 
from strong outerness. 
The hypothesis on $A$ was necessary 
to replace central sequences in the sense of von Neumann algebras 
with central sequences in the sense of $C^*$-algebras. 
In this paper, 
we have removed it by using amenability of nuclear $C^*$-algebras 
proved by U. Haagerup \cite{H83Invent} (see Proposition \ref{Haagerup2}). 
As for the group $\Gamma$, 
we dealt with only $\Z^N$ and finite groups in the previous paper. 
Now we extend our arguments to all elementary amenable groups 
by showing that they admit outer actions on the hyperfinite II$_1$-factor 
with a `nice' Rohlin property 
(see Definition \ref{DefofQ} and Proposition \ref{EG>Q}). 
In order to prove $\mathcal{Z}$-stability of crossed products, 
we need one more ingredient, namely the property (SI). 
In \cite{MS2}, we have already shown that 
$\mathcal{Z}$-stability of $A$ implies the property (SI) 
(Theorem \ref{Z=SI}). 
Furthermore, 
by using the weak Rohlin property of $(\alpha,u)$, we will prove that 
$A$ satisfies an $\alpha$-equivariant version of the property (SI) 
(Proposition \ref{equivSI}). 
Together with an exact sequence of central sequence algebras 
(Theorem \ref{exact}), this allows us to conclude that 
$A\rtimes_{(\alpha,u)}\Gamma$ is $\mathcal{Z}$-stable. 

In Section 5, we will also show that 
if a unital separable nuclear $C^*$-algebra $A$ is not of type I, then 
the central sequence algebra has a subquotient which is isomorphic to 
a II$_1$-factor (Theorem \ref{Glimm}). 
This may be thought of as an analogue of Glimm's theorem. 

In the latter half of the paper, 
we study cocycle actions of the Klein bottle group 
on UHF algebras and the Jiang-Su algebra. 
The Klein bottle group is the group 
generated by two elements $a,b$ satisfying $bab^{-1}=a^{-1}$, 
and is isomorphic to the fundamental group of the Klein bottle. 
Classification of group actions is 
one of the most fundamental subjects in the theory of operator algebras. 
We briefly review classification results of automorphisms or group actions 
on simple stably finite $C^*$-algebras known so far. 
For AF and AT algebras, 
A. Kishimoto \cite{K95crelle,K98JOT,K98JFA} showed 
the Rohlin property for a certain class of automorphisms and 
obtained a cocycle conjugacy result. 
The first-named author \cite{M10CMP} extended this result to 
unital simple AH algebras with real rank zero and slow dimension growth. 
The second-named author \cite{S10JFA} proved that 
strongly outer $\Z$-actions on $\mathcal{Z}$ are 
unique up to strong cocycle conjugacy. 
T. Katsura and the first-named author \cite{KM} gave 
a complete classification of strongly outer $\Z^2$-actions on UHF algebras 
by using the Rohlin property, 
and then this result was extended to a certain class of 
strongly outer $\Z^2$-actions on unital simple AF algebras \cite{M10CMP}. 
The uniqueness of strongly outer $\Z^N$-actions 
on UHF algebras of infinite type was obtained in \cite{M11crelle}. 
In our previous paper \cite{MS}, 
we proved that strongly outer $\Z^2$-actions on $\mathcal{Z}$ are 
unique up to strong cocycle conjugacy. 

In the present paper, we first show that 
strongly outer cocycle actions of the Klein bottle group 
on UHF algebras are unique up to cocycle conjugacy (Theorem \ref{KleinUHF}). 
The proof is along the same lines as the proof of \cite[Theorem 6.6]{M10CMP}, 
but we need to look at the $\OrderExt$ invariant carefully and 
to check that it always vanishes, unlike the case of $\Z^2$-actions. 
Then we establish a cohomology vanishing type result 
for cocycle actions of the Klein bottle group 
on the Jiang-Su algebra $\mathcal{Z}$ (Proposition \ref{CVonZ}). 
The $\mathcal{Z}$-stability theorem proved in the first half of the paper 
plays an essential role. 
Finally, by using this cohomology vanishing, 
we obtain the uniqueness of strongly outer cocycle actions 
of the Klein bottle group on $\mathcal{Z}$ (Theorem \ref{KleinZ}). 
The uniqueness up to strong cocycle conjugacy of strongly outer actions 
is also given (Theorem \ref{KleinZ2}). 
These are the first classification results of (cocycle) actions 
of non-abelian infinite groups on $C^*$-algebras.

%%%%%%%%%%%%%%%%%%%%%%%%%%%%%%%%%%%%%%%%%%%%%%%%%%%%%%%%%%%%
\section{Preliminaries}

%%%%%%%%%%%%%%%%%%%%%%%%%%%%%%%%%%%%%%%%%%%%%%%%%%%%%%%%%%%%
\subsection{Notations}

Let $A$ be a $C^*$-algebra. 
For $a,b\in A$, we mean by $[a,b]$ the commutator $ab-ba$. 
The set of tracial states on $A$ is denoted by $T(A)$. 
For $a\in A$, we define 
\[
\lVert a\rVert_2=\sup_{\tau\in T(A)}\tau(a^*a)^{1/2}
\]
If $A$ is simple and $T(A)$ is non-empty, 
then $\lVert\cdot\rVert_2$ is a norm. 
For $\tau\in T(A)$, 
we let $(\pi_\tau,H_\tau)$ denote 
the GNS representation of $A$ associated with $\tau$. 
For $\tau\in T(A)$, 
the dimension function $d_\tau$ associated with $\tau$ is given by 
\[
d_\tau(a)=\lim_{n\to\infty}\tau(a^{1/n}), 
\]
for a positive element $a\in A$. 

When $A$ is unital, we mean by $U(A)$ the set of all unitaries in $A$. 
For $u\in U(A)$, 
the inner automorphism induced by $u$ is written $\Ad u$. 
An automorphism $\alpha\in\Aut(A)$ is called outer, 
when it is not inner. 
For $\alpha\in\Aut(A)$, 
the fixed point subalgebra $\{a\in A\mid\alpha(a)=a\}$ is written $A^\alpha$. 
We let $T(A)^\alpha=\{\tau\in T(A)\mid\tau\circ\alpha=\tau\}$. 
A single automorphism is often identified with a $\Z$-action generated by it. 

Let $A$ be a separable $C^*$-algebra and 
let $\omega\in\beta\N\setminus\N$ be a free ultrafilter. 
Set 
\[
c_0(A)=\{(a_n)\in\ell^\infty(\N,A)\mid
\lim_{n\to\infty}\lVert a_n\rVert=0\},\quad 
A^\infty=\ell^\infty(\N,A)/c_0(A), 
\]
\[
c_\omega(A)=\{(a_n)\in\ell^\infty(\N,A)\mid
\lim_{n\to\omega}\lVert a_n\rVert=0\},\quad 
A^\omega=\ell^\infty(\N,A)/c_\omega(A). 
\]
We identify $A$ with the $C^*$-subalgebra of $A^\infty$ (resp. $A^\omega$)  
consisting of equivalence classes of constant sequences. 
We let 
\[
A_\infty=A^\infty\cap A',\quad A_\omega=A^\omega\cap A'
\]
and call them the central sequence algebras of $A$. 
A sequence $(x_n)_n\in\ell^\infty(\N,A)$ is called a central sequence 
if $\lVert[a,x_n]\rVert\to0$ as $n\to\infty$ for all $a\in A$. 
A central sequence is a representative of an element in $A_\infty$. 
An $\omega$-central sequence is defined in a similar way. 
For $\tau\in T(A)$, we define $\tau_\omega\in T(A^\omega)$ by 
\[
\tau_\omega(x)=\lim_{n\to\omega}\tau(x_n), 
\]
where $(x_n)_n$ is a representative sequence of $x\in A^\omega$. 
When $\alpha$ is an automorphism of $A$, 
we can consider its natural extension 
on $A^\infty$, $A^\omega$, $A_\infty$ and $A_\omega$. 
We denote it by the same symbol $\alpha$. 

We let $\mathcal{Z}$ denote the Jiang-Su algebra 
introduced by X. Jiang and H. Su \cite{JS}. 
They proved that 
any unital endomorphisms of $\mathcal{Z}$ are approximately inner and that 
$\mathcal{Z}$ is isomorphic to the infinite tensor product of its replicas. 
In particular, 
$\mathcal{Z}$ is strongly self-absorbing, cf. \cite{TW07TAMS}. 
When a $C^*$-algebra $A$ satisfies $A\cong A\otimes\mathcal{Z}$, 
we say that $A$ absorbs $\mathcal{Z}$ tensorially, 
or $A$ is $\mathcal{Z}$-stable.

%%%%%%%%%%%%%%%%%%%%%%%%%%%%%%%%%%%%%%%%%%%%%%%%%%%%%%%%%%%%
\subsection{Group actions and cocycle conjugacy}

We set up some terminologies for group actions. 
For a discrete group $\Gamma$, 
the neutral element is denoted by $1\in\Gamma$. 

\begin{definition}\label{Defofca}
Let $A$ be a unital $C^*$-algebra and let $\Gamma$ be a discrete group. 
\begin{enumerate}
\item A pair $(\alpha,u)$ of 
a map $\alpha:\Gamma\to\Aut(A)$ and a map $u:\Gamma\times\Gamma\to U(A)$ 
is called a cocycle action of $\Gamma$ on $A$ 
if 
\[
\alpha_g\circ\alpha_h=\Ad u(g,h)\circ\alpha_{gh}
\]
and 
\[
u(g,h)u(gh,k)=\alpha_g(u(h,k))u(g,hk)
\]
hold for any $g,h,k\in\Gamma$. 
We always assume $\alpha_1=\id$, $u(g,1)=u(1,g)=1$ for all $g\in\Gamma$. 
We denote the cocycle action by $(\alpha,u):\Gamma\curvearrowright A$. 
Notice that $\alpha$ gives rise to 
a (genuine) action of $\Gamma$ on $A_\infty$ and $A_\omega$. 
\item A cocycle action $(\alpha,u)$ is said to be outer 
if $\alpha_g$ is outer for every $g\in\Gamma$ except for the neutral element. 
\item Two cocycle actions $(\alpha,u):\Gamma\curvearrowright A$ and 
$(\beta,v):\Gamma\curvearrowright B$ are said to be cocycle conjugate 
if there exist a family of unitaries $(w_g)_{g\in\Gamma}$ in $B$ and 
an isomorphism $\theta:A\to B$ such that 
\[
\theta\circ\alpha_g\circ\theta^{-1}=\Ad w_g\circ\beta_g
\]
and 
\[
\theta(u(g,h))=w_g\beta_g(w_h)v(g,h)w_{gh}^*
\]
for every $g,h\in\Gamma$. 
Furthermore, when there exists a sequence $(x_n)_n$ of unitaries in $B$ 
such that $x_n\beta_g(x_n^*)\to w_g$ as $n\to\infty$ for every $g\in\Gamma$, 
we say that $(\alpha,u)$ and $(\beta,v)$ are strongly cocycle conjugate. 
\item Let $\alpha:\Gamma\curvearrowright A$ be an action. 
The fixed point subalgebra 
$\{a\in A\mid\alpha_g(a)=a\quad\forall g\in\Gamma\}$ is written $A^\alpha$. 
\item Let $\alpha:\Gamma\curvearrowright A$ be an action. 
A family of unitaries $(x_g)_{g\in\Gamma}$ in $A$ is called 
an $\alpha$-cocycle 
if one has $x_g\alpha_g(x_h)=x_{gh}$ for all $g,h\in\Gamma$. 
\end{enumerate}
\end{definition}

\begin{definition}[{\cite[Definition 2.4]{PR}}]\label{Defoftc}
Let $(\alpha,u):\Gamma\curvearrowright A$ be a cocycle action 
of a discrete group $\Gamma$ on a unital $C^*$-algebra $A$. 
The twisted crossed product $A\rtimes_{(\alpha,u)}\Gamma$ is 
the universal $C^*$-algebra generated by 
$A$ and a family of unitaries $(\lambda^\alpha_g)_{g\in\Gamma}$ satisfying 
\[
\lambda^\alpha_g\lambda^\alpha_h=u(g,h)\lambda^\alpha_{gh}
\quad\text{and}\quad 
\lambda^\alpha_ga\lambda^{\alpha*}_g=\alpha_g(a)
\]
for all $g,h\in\Gamma$ and $a\in A$. 
\end{definition}

If two cocycle actions $(\alpha,u):\Gamma\curvearrowright A$ and 
$(\beta,v):\Gamma\curvearrowright B$ are cocycle conjugate, then 
$A\rtimes_{(\alpha,u)}\Gamma$ and $B\rtimes_{(\beta,v)}\Gamma$ 
are canonically isomorphic. 
Conversely, if there exists an isomorphism 
$\theta:A\rtimes_{(\alpha,u)}\Gamma\to B\rtimes_{(\beta,v)}\Gamma$ satisfying 
\[
\theta(A)=B\quad\text{and}\quad 
\theta(\lambda^\alpha_g)(\lambda^\beta_g)^*\in B\quad\forall g\in\Gamma, 
\]
then it is easy to check that 
$(\alpha,u)$ and $(\beta,v)$ are cocycle conjugate. 
When $\Gamma$ is amenable, 
$A\rtimes_{(\alpha,u)}\Gamma$ is naturally isomorphic to 
the reduced twisted crossed product (\cite[Theorem 3.11]{PR}).

%%%%%%%%%%%%%%%%%%%%%%%%%%%%%%%%%%%%%%%%%%%%%%%%%%%%%%%%%%%%
\subsection{The weak Rohlin property}

We recall the definitions of amenable groups and 
of elementary amenable groups (in the sense of M. M. Day \cite{Day}). 
The cardinality of a set $F$ is written $\lvert F\rvert$. 

\begin{definition}
Let $\Gamma$ be a countable discrete group. 
\begin{enumerate}
\item For a finite subset $F\subset\Gamma$ and $\ep>0$, 
we say that a finite subset $K\subset\Gamma$ is $(F,\ep)$-invariant 
if $\lvert K\cap\bigcap_{g\in F}g^{-1}K\rvert\geq(1{-}\ep)\lvert K\rvert$. 
\item The group $\Gamma$ is said to be amenable 
if for any finite subset $F\subset\Gamma$ and $\ep>0$ 
there exists an $(F,\ep)$-invariant finite subset $K\subset\Gamma$. 
\item The class of elementary amenable groups is defined as 
the smallest family of groups containing all finite groups and 
all abelian groups, and closed under the processes of 
taking subgroups, quotients, group extensions and increasing unions. 
\end{enumerate}
\end{definition}

For countable discrete amenable groups, 
we introduce the property (Q) as follows. 

\begin{definition}\label{DefofQ}
Let $\Gamma$ be a countable discrete amenable group and 
let $\alpha:\Gamma\curvearrowright R$ be an outer action 
on the AFD II$_1$-factor $R$. 
We say that $\Gamma$ has the property (Q) if the following holds: 
For any finite subset $F\subset\Gamma$ and $\ep>0$, 
there exists an $(F,\ep)$-invariant finite subset $K\subset\Gamma$ 
and a sequence of projections $(p_n)_n$ in $R$ such that 
\[
\left\lVert1-\sum_{g\in K}\alpha_g(p_n)\right\rVert_2\to0\quad\text{and}\quad 
\lVert[x,p_n]\rVert_2\to0\quad \forall x\in R
\]
as $n\to\infty$. 
\end{definition}

The definition above does not depend on the choice of 
the outer action $\alpha:\Gamma\curvearrowright R$, 
because all outer actions of $\Gamma$ on $R$ are mutually cocycle conjugate 
(\cite{C75Ann,J80Mem,O,Msd}). 
In Proposition \ref{EG>Q}, it will be shown that 
any elementary amenable group has the property (Q). 

Next we introduce the notions of strong outerness, 
the weak Rohlin property and the tracial Rohlin property 
for group actions on $C^*$-algebras. 

\begin{definition}[{\cite[Definition 2.7]{MS}}]\label{so&wR}
Let $A$ be a unital simple $C^*$-algebra with $T(A)$ non-empty and 
let $\Gamma$ be a discrete countable amenable group. 
\begin{enumerate}
\item We say that an automorphism $\alpha\in\Aut(A)$ 
is not weakly inner for $\tau\in T(A)^\alpha$ 
if the weak extension of $\alpha$ to an automorphism of $\pi_\tau(A)''$ 
is outer. 
\item A cocycle action $(\alpha,u)$ of $\Gamma$ on $A$ 
is said to be strongly outer 
if $\alpha_g$ is not weakly inner 
for any $\tau\in T(A)^{\alpha_g}$ and $g\in\Gamma\setminus\{e\}$. 
\item We say that a cocycle action $(\alpha,u):\Gamma\curvearrowright A$ has 
the weak Rohlin property 
if for any finite subset $F\subset\Gamma$ and $\ep>0$ 
there exist an $(F,\ep)$-invariant finite subset $K\subset\Gamma$ and 
a central sequence $(f_n)_n$ in $A$ such that $0\leq f_n\leq1$, 
\[
\lim_{n\to\infty}\lVert\alpha_g(f_n)\alpha_h(f_n)\rVert=0
\]
for all $g,h\in K$ with $g\neq h$ and 
\[
\lim_{n\to\infty}
\max_{\tau\in T(A)}\lvert\tau(f_n)-\lvert K\rvert^{-1}\rvert=0. 
\]
When $G=\Z$, in addition to the conditions above, 
we further impose the restriction that 
$K$ is of the form $\{0,1,\dots,k\}$ for some $k\in\N$. 
\item Let $A$ be a unital simple $C^*$-algebra with tracial rank zero. 
We say that a cocycle action $(\alpha,u):\Gamma\curvearrowright A$ has 
the tracial Rohlin property 
if the central sequence $(f_n)_n$ in the definition above can be chosen 
as a central sequence consisting of projections. 
\end{enumerate}
\end{definition}

We remark that 
the definition of the tracial Rohlin property which we have introduced here 
is slightly stronger than that given in \cite{OP06ETDS2,P}

%%%%%%%%%%%%%%%%%%%%%%%%%%%%%%%%%%%%%%%%%%%%%%%%%%%%%%%%%%%%
\subsection{The property (SI)}

We give the definition of the property (SI) 
which plays an important role in Section 4. 
The original definition given in \cite{S10JFA} is 
slightly different from this version, 
but they are actually equivalent 
(see the discussion following \cite[Definition 4.1]{MS}). 

\begin{definition}[{\cite[Definition 4.1]{MS}}]\label{DefofSI}
We say that a separable $C^*$-algebra $A$ has the property (SI) 
when for any central sequences $(x_n)_n$ and $(y_n)_n$ in $A$ satisfying 
$0\leq x_n\leq1$, $0\leq y_n\leq1$, 
\[
\lim_{n\to\infty}\max_{\tau\in T(A)}\tau(x_n)=0\quad\text{and}\quad 
\inf_{m\in\N}\liminf_{n\to\infty}\min_{\tau\in T(A)}\tau(y_n^m)>0, 
\]
there exists a central sequence $(s_n)_n$ in $A$ such that 
\[
\lim_{n\to\infty}\lVert s_n^*s_n-x_n\rVert=0\quad\text{and}\quad 
\lim_{n\to\infty}\lVert y_ns_n-s_n\rVert=0. 
\]
\end{definition}

\begin{remark}\label{RemofSI}
The property (SI) introduced in the definition above 
can be reformulated in terms of $\omega$-central sequences. 
Namely, $A$ has the property (SI) if and only if 
for any $\omega$-central sequences $(x_n)_n$ and $(y_n)_n$ of 
positive contractions in $A$ satisfying 
\[
\lim_{n\to\omega}\max_{\tau\in T(A)}\tau(x_n)=0\quad\text{and}\quad 
\inf_{m\in\N}\lim_{n\to\omega}\min_{\tau\in T(A)}\tau(y_n^m)>0, 
\]
there exists an $\omega$-central sequence $(s_n)_n$ in $A$ such that 
\[
\lim_{n\to\omega}\lVert s_n^*s_n-x_n\rVert=0\quad\text{and}\quad 
\lim_{n\to\omega}\lVert y_ns_n-s_n\rVert=0. 
\]
Indeed, these two formulations are both equivalent to the following condition: 
for any finite subset $F\subset A$, $\ep>0$ and $c>0$, 
there exist a finite subset $G\subset A$, $\delta>0$ and $m\in\N$ 
such that the following holds. 
If positive contractions $x,y\in A$ satisfy 
\[
\lVert[x,a]\rVert<\delta,\quad \lVert[y,a]\rVert<\delta\quad\forall a\in G, 
\]
\[
\max_{\tau\in T(A)}\tau(x)<\delta\quad\text{and}\quad 
\min_{\tau\in T(A)}\tau(y^m)>c, 
\]
then one can find $s\in A$ such that 
\[
\lVert[s,a]\rVert<\ep\quad\forall a\in F,\quad 
\lVert s^*s-x\rVert<\ep\quad\text{and}\quad \lVert ys-s\rVert<\ep. 
\]
\end{remark}

In \cite[Lemma 4.3]{MS}, we proved that 
certain $\mathcal{Z}$-stable $C^*$-algebras have the property (SI). 
Recently we have generalized this result and obtained the following 
(\cite{MS2}). 

\begin{theorem}[{\cite[Theorem 1.1, Theorem 4.2]{MS2}}]\label{Z=SI}
Let $A$ be a unital, simple, separable, nuclear, stably finite, 
infinite dimensional $C^*$-algebra. 
\begin{enumerate}
\item If $A$ is $\mathcal{Z}$-stable, then $A$ has the property (SI). 
\item Assume further that $A$ has finitely many extremal tracial states. 
If $A$ has the property (SI), then $A$ is $\mathcal{Z}$-stable. 
\end{enumerate}
\end{theorem}

In Section 4, 
we will give a slightly different proof of (2) of Theorem \ref{Z=SI} 
(Theorem \ref{wR>Z}).

%%%%%%%%%%%%%%%%%%%%%%%%%%%%%%%%%%%%%%%%%%%%%%%%%%%%%%%%%%%%
\subsection{Approximate representability of cocycle actions}

In this subsection, we collect notations and terminologies 
which will be used in Section 6 and Section 7. 

For a Lipschitz continuous function $f$, 
we denote its Lipschitz constant by $\Lip(f)$. 

Let $A$ be a unital $C^*$-algebra. 
The collection of all continuous bounded affine maps 
from $T(A)$ to $\R$ is denoted by $\Aff(T(A))$. 
The dimension map 
\[
D_A:K_0(A)\to\Aff(T(A))
\]
is defined by $D_A([p])(\tau)=\tau(p)$. 
The connected component of the identity in $U(A)$ is denoted by $U(A)_0$. 
For $\tau\in T(A)$, 
the de la Harpe-Skandalis determinant (\cite{HS}) associated with $\tau$ is 
written 
\[
\Delta_\tau:U(A)_0\to\R/\tau(K_0(A)). 
\]
We let $\log$ be the standard branch 
defined on the complement of the negative real axis. 
If $\lVert u-1\rVert<2$, 
then $\Delta_\tau(u)=(2\pi\sqrt{-1})^{-1}\tau(\log u)+\tau(K_0(A))$. 
We frequently use the following fact: 
when $u,v\in U(A)$ satisfy $\lVert u-1\rVert+\lVert v-1\rVert<2$, 
one has $\tau(\log(uv))=\tau(\log u)+\tau(\log v)$ 
for any $\tau\in T(A)$ (see \cite[Lemma 1]{HS}). 

For a homomorphism $\phi$ between $C^*$-algebras, 
$K_0(\phi)$ and $K_1(\phi)$ are the induced homomorphisms on $K$-groups. 
Let $A$ and $B$ be unital $C^*$-algebras. 
Two unital homomorphisms $\phi,\psi:A\to B$ are said to be 
asymptotically unitarily equivalent 
if there exists a continuous family of unitaries 
$(u_t)_{t\in[0,\infty)}$ in $B$ such that 
\[
\phi(a)=\lim_{t\to\infty}\Ad u_t(\psi(a))
\]
for all $a\in A$. 
When there exists a sequence of unitaries $(u_n)_{n\in\N}$ in $B$ 
such that 
\[
\phi(a)=\lim_{n\to\infty}\Ad u_n(\psi(a))
\]
for all $a\in A$, 
$\phi$ and $\psi$ are said to be approximately unitarily equivalent. 
If $\phi$ and $\psi$ are approximately unitarily equivalent, 
then $K_i(\phi)=K_i(\psi)$ for $i=0,1$. 
An automorphism $\alpha\in\Aut(A)$ is said to be 
asymptotically (resp. approximately) inner 
if $\alpha$ is asymptotically (resp. approximately) 
unitarily equivalent to the identity map. 
As usual, we denote by $\aInn(A)$ 
the set of all approximately inner automorphisms of $A$. 
Evidently $\aInn(A)$ is a normal subgroup of $\Aut(A)$. 

Let $\phi\in\Aut(A)$ be an automorphism of a unital $C^*$-algebra $A$. 
The implementing unitary in the crossed product $C^*$-algebra 
$A\rtimes_\phi\Z$ is written $\lambda^\phi$. 
The set of all automorphisms $\psi\in\Aut(A\rtimes_\phi\Z)$ satisfying 
$\psi(A)=A$ and $\psi(\lambda^\phi)\lambda^{\phi*}\in A$ 
is denoted by $\Aut_\T(A\rtimes_\phi\Z)$ 
(similar definitions are found 
in \cite[Section 2]{IM} and \cite[Section 2]{M11crelle}). 
An automorphism $\psi\in\Aut_\T(A\rtimes_\phi\Z)$ 
is said to be $\T$-asymptotically inner 
if there exists a continuous family of unitaries 
$(u_t)_{t\in[0,\infty)}$ in $A$ such that 
\[
\psi(x)=\lim_{t\to\infty}\Ad u_t(x)
\]
for all $x\in A\rtimes_\phi\Z$. 
In an analogous way, one can define $\T$-approximate innerness. 

Next, we recall the notion of 
approximate (resp. asymptotical) representability of group actions. 

\begin{definition}[{\cite[Definition 2.2]{IM}}]\label{apprepre}
Let $\Gamma$ be a countable discrete group and 
let $A$ be a unital $C^*$-algebra. 
A cocycle action $(\alpha,u):\Gamma\curvearrowright A$ is said to be 
approximately representable 
if there exists a family of unitaries $(v_g)_{g\in\Gamma}$ in $A^\infty$ 
such that 
\[
v_gv_h=u(g,h)v_{gh},\quad 
\alpha_g(v_h)=u(g,h)u(ghg^{-1},g)^*v_{ghg^{-1}}
\]
and 
\[
v_gxv_g^*=\alpha_g(x)
\]
hold for all $g,h\in\Gamma$ and $x\in A$. 

Asymptotical representability is defined in an analogous way. 
\end{definition}

It is routine to check that 
approximate (resp. asymptotical) representability is preserved 
under cocycle conjugacy. 

The following proposition lists non-trivial examples of 
strongly outer, asymptotically representable actions 
on stably finite $C^*$-algebras. 

\begin{proposition}\label{examples}
Let $\alpha:\Gamma\curvearrowright A$ be a strongly outer action. 
Assume that one of the following conditions holds. 
\begin{enumerate}
\item $\Gamma=\Z$, 
$\alpha_g$ is asymptotically inner for every non-trivial $g\in\Z$, and 
$A$ is a unital simple AH algebra with real rank zero, slow dimension growth 
and finitely many extremal tracial states. 
\item $\Gamma=\Z^N$ and $A$ is a UHF algebra of infinite type. 
\item $\Gamma=\Z$ and $A$ is the Jiang-Su algebra. 
\item $\Gamma=\Z^2$ and $A$ is the Jiang-Su algebra. 
\end{enumerate}
Then $\alpha:\Gamma\curvearrowright A$ is asymptotically representable. 
\end{proposition}
\begin{proof}
(1) follows from \cite[Theorem 4.8, 4.9]{M10CMP}, 
because one can easily construct a strongly outer, 
asymptotically representable action of $\Z$ on $A$. 
(2) follows from \cite[Theorem 3.4, 4.5]{M11crelle}. 
(3) follows from \cite[Theorem 1.2]{S09} and \cite[Theorem 1.3]{S10JFA}, 
because one can easily construct a strongly outer, 
asymptotically representable action of $\Z$ on $\mathcal{Z}$. 
In a similar fashion, (4) follows from \cite[Theorem 7.9]{MS}. 
\end{proof}

%%%%%%%%%%%%%%%%%%%%%%%%%%%%%%%%%%%%%%%%%%%%%%%%%%%%%%%%%%%%
\section{The weak Rohlin property}

In this section we prove Theorem \ref{so=wR} and Theorem \ref{so=trR}. 
See Definition \ref{DefofQ} for the definition of the property (Q). 

\begin{lemma}\label{Q}
\begin{enumerate}
\item The group of integers $\Z$ has the property (Q). 
\item Any finite group has the property (Q). 
\item If $\Gamma=\bigcup_n\Gamma_n$ is an increasing union of 
countable discrete amenable groups $\Gamma_n$ with the property (Q), 
then $\Gamma$ also has the property (Q). 
\item If $1\to N\to\Gamma\to H\to1$ is a short exact sequence of 
countable discrete amenable groups and $N,H$ have the property (Q), 
then $\Gamma$ also has the property (Q). 
\item Any countable abelian group has the property (Q). 
\end{enumerate}
\end{lemma}
\begin{proof}
(1) follows from \cite{C75Ann} and (2) follows from \cite{J80Mem}. 
(3) is obvious. 
(5) is an immediate consequence of the other assertions, 
because any countable abelian group is an increasing union of 
finitely generated abelian groups. 

It remains for us to show (4). 
Let $N$ and $H$ be countable discrete amenable groups with the property (Q) 
and let $1\to N\to\Gamma\to H\to1$ be a short exact sequence. 
Let $\pi:\Gamma\to H$ be the quotient map and 
let $\sigma:H\to\Gamma$ be its right inverse. 
Let $\alpha:\Gamma\curvearrowright R$ be an outer action of $\Gamma$ 
on the AFD II$_1$-factor $R$ and 
let $\beta:H\curvearrowright R$ be an outer action of $H$ on $R$. 
Suppose that we are given a finite subset $F\subset\Gamma$ and $\ep>0$. 
Without loss of generality, we may assume that 
there exist finite subsets $F_1\subset N$ and $F_2\subset H$ 
such that $F=F_1\cup\sigma(F_2)$. 
Since $H$ has the property (Q), 
we can find an $(F_2,\ep/2)$-invariant finite subset $K_2\subset H$ 
and a sequence of projections $(q_n)_n$ in $R$ satisfying 
\[
\left\lVert1-\sum_{h\in K_2}\beta_h(q_n)\right\rVert_2\to0\quad\text{and}\quad 
\lVert[x,q_n]\rVert_2\to0\quad \forall x\in R. 
\]
Define a finite subset $\tilde F_1\subset N$ by 
\[
\tilde F_1=\{\sigma(h)^{-1}g\sigma(h)\mid h\in K_2,\ g\in F_1\}
\cup\{\sigma(h_2h_1)^{-1}\sigma(h_2)\sigma(h_1)\mid h_1\in K_2,\ h_2\in F_2\}. 
\]
Since $N$ has the property (Q), 
we can find an $(\tilde F_1,\ep/2)$-invariant finite subset 
$K_1\subset N$ and a sequence of projections $(p_n)_n$ in $R$ satisfying 
\[
\left\lVert1-\sum_{g\in K_1}\alpha_g(p_n)\right\rVert_2\to0
\quad\text{and}\quad 
\lVert[x,p_n]\rVert_2\to0\quad \forall x\in R. 
\]
Put $K=\{\sigma(h)g\in\Gamma\mid g\in K_1,\ h\in K_2\}$. 
For any $t\in F_1$, $h\in K_2$ and 
$g\in K_1\cap\bigcap_{s\in\tilde F_1}s^{-1}K_1$, one has 
\[
t\sigma(h)g=\sigma(h)\sigma(h)^{-1}t\sigma(h)g\in K. 
\]
Also, for any $t\in F_2$, $h\in K_2\cap\bigcap_{s\in F_2}s^{-1}K_2$ and 
$g\in K_1\cap\bigcap_{s\in\tilde F_1}s^{-1}K_1$, one has 
\[
\sigma(t)\sigma(h)g=\sigma(th)\sigma(th)^{-1}\sigma(t)\sigma(h)g\in K. 
\]
Hence 
\begin{align*}
\left\lvert K\cap\bigcap_{t\in F}t^{-1}K\right\rvert
&\geq\left\lvert K_2\cap\bigcap_{s\in F_2}s^{-1}K_2\right\rvert\cdot
\left\lvert K_1\cap\bigcap_{s\in\tilde F_1}s^{-1}K_1\right\rvert\\
&\geq(1{-}\ep/2)\lvert K_2\rvert\cdot(1{-}\ep/2)\lvert K_1\rvert
>(1{-}\ep)\lvert K\rvert, 
\end{align*}
thus $K$ is $(F,\ep)$-invariant. 

Define an outer action $\gamma:\Gamma\curvearrowright R\bar\otimes R$ 
by $\gamma_s=\alpha_s\otimes\beta_{\pi(s)}$ for $s\in\Gamma$. 
It is easy to see 
\[
\lim_{n\to\infty}
\lVert[x,p_n\otimes q_n]\rVert_2=0\quad \forall x\in R\bar\otimes R
\]
and 
\begin{align*}
\lim_{n\to\infty}
\left\lVert1-\sum_{s\in K}\gamma_s(p_n\otimes q_n)\right\rVert_2
&=\lim_{n\to\infty}\left\lVert1-\sum_{h\in K_2,g\in K_1}
\gamma_{\sigma(h)}(\alpha_g(p_n)\otimes q_n)\right\rVert_2\\
&=\lim_{n\to\infty}\left\lVert1-\sum_{h\in K_2}
1\otimes\beta_h(q_n)\right\rVert_2
=0, 
\end{align*}
which completes the proof. 
\end{proof}

\begin{remark}
Let $N$ be a countable discrete amenable group with the property (Q) and 
let $1\to N\to\Gamma\to\Z\to1$ be a short exact sequence. 
Thus, $\Gamma$ is isomorphic to a semidirect product $N\rtimes\Z$. 
Letting $a\in\Z$ denote a generator, 
we write $\Gamma=\{ga^i\mid g\in N,\ i\in\Z\}$. 
From the lemma above, $\Gamma$ has the property (Q). 
Moreover, its proof tells us that 
the `invariant' set in $\Gamma$ can be chosen of the form 
$\{ga^i\mid g\in K,\ 0\leq i\leq k{-}1\}$, 
where $K\subset N$ is an `invariant' set. 
More precisely, we get the following. 
Let $\alpha:\Gamma\curvearrowright R$ be 
an outer action of $\Gamma$ on the hyperfinite II$_1$-factor $R$. 
Then, for any finite subset $F\subset N$, $\ep>0$ and $k\in\N$, 
there exists an $(F,\ep)$-invariant finite subset $K\subset N$ and 
a sequence $(p_n)_n$ of projections in $R$ such that 
\[
\left\lVert1-\sum_{g\in K}\sum_{i=0}^{k-1}
\alpha_{ga^i}(p_n)\right\rVert_2\to0\quad\text{and}\quad 
\lVert[x,p_n]\rVert_2\to0\quad \forall x\in R
\]
as $n\to\infty$. 
\end{remark}

\begin{proposition}\label{EG>Q}
Every countable discrete elementary amenable group has the property (Q). 
\end{proposition}
\begin{proof}
This follows from the lemma above and \cite[Proposition 2.2 (b)]{Chou}. 
\end{proof}

In order to state and prove Proposition \ref{Haagerup2} below, 
we need to recall U. Haagerup's result \cite[Theorem 3.1]{H83Invent}, 
which says that any nuclear $C^*$-algebra has 
a virtual diagonal in the sense of \cite{J72Amer}. 
In particular, any nuclear $C^*$-algebra is amenable. 
Here we quote the following theorem from \cite{FKK}, 
which is an interpretation of \cite[Theorem 3.1]{H83Invent}. 
See also \cite[Lemma 4.2]{KOS}. 

\begin{theorem}[{\cite[Theorem 2.1]{FKK}}]\label{Haagerup1}
Let $A$ be a unital nuclear $C^*$-algebra. 
For any finite subset $F\subset A$ and $\ep>0$, 
there exist $w_1,w_2,\dots,w_m\in A$ such that 
\[
\sum_{i=1}^mw_iw_i^*=1\quad\text{and}\quad 
\left\lVert\sum_{i=1}^maw_ixw_i^*-\sum_{i=1}^mw_ixw_i^*a\right\rVert
<\ep\lVert x\rVert\quad 
\forall a\in F,\ x\in A. 
\]
\end{theorem}

The following is a variant of \cite[Lemma 2.1]{S11}. 
We include the proof for completeness. 

\begin{proposition}\label{Haagerup2}
Let $A$ be a unital nuclear $C^*$-algebra. 
For any finite subset $F\subset A$ and $\ep>0$, 
there exist a finite subset $G\subset A$ and $\delta>0$ such that 
the following hold. 
If $x\in A$ is a positive element such that 
\[
\lVert x\rVert\leq1\quad\text{and}\quad 
\lVert[x,a]\rVert_2<\delta\quad \forall a\in G, 
\]
then there exists a positive element $y\in A$ such that 
\[
\lVert y\rVert\leq1,\quad \lVert x-y\rVert_2<\ep\quad\text{and}\quad 
\lVert[y,a]\rVert<\ep\quad \forall a\in F. 
\]
\end{proposition}
\begin{proof}
By applying the theorem above to $F\subset A$ and $\ep>0$, 
we get $G=\{w_1,w_2,\dots,w_m\}\subset A$. 
Let $\delta=\ep/m$. 
Suppose that 
$x\in A$ is a positive contraction satisfying
$\lVert[x,w_i]\rVert_2<\delta$ for any $w_i\in G$. 
Put $y=\sum_iw_ixw_i^*$. 
Clearly 
\[
0\leq y\leq\sum_iw_iw_i^*=1\quad\text{and}\quad 
\lVert x-y\rVert_2<m\delta=\ep. 
\]
Furthermore, for any $a\in F$, 
\[
\lVert[y,a]\rVert
=\left\lVert\sum_{i=1}^maw_ixw_i^*-\sum_{i=1}^mw_ixw_i^*a\right\rVert
<\ep\lVert x\rVert\leq\ep. 
\]
\end{proof}

We are now ready to show 
the equivalence of strong outerness and the weak Rohlin property 
(see Definition \ref{so&wR} for the definitions) 
for cocycle actions of groups with the property (Q). 

\begin{theorem}\label{so=wR}
Let $A$ be a unital, simple, separable, nuclear, stably finite, 
infinite dimensional $C^*$-algebra with finitely many extremal tracial states 
and let $\Gamma$ be a countable discrete amenable group with the property (Q). 
For a cocycle action $(\alpha,u)$ of $\Gamma$ on $A$, 
the following conditions are equivalent. 
\begin{enumerate}
\item $(\alpha,u)$ is strongly outer. 
\item $(\alpha,u)$ has the weak Rohlin property. 
\end{enumerate}
\end{theorem}
\begin{proof}
The implication (2)$\Rightarrow$(1) was proved in \cite[Remark 2.8]{MS}. 
We show the other implication. 
Suppose that $(\alpha,u):\Gamma\curvearrowright A$ is strongly outer. 
Let $E$ be the set of extremal tracial states on $A$. 
For $\tau\in E$, it is well-known that 
$\pi_\tau(A)''$ is the AFD II$_1$-factor. 
Set $\pi=\bigoplus_{\tau\in E}\pi_\tau$ and 
$M=\pi(A)''=\bigoplus_{\tau\in E}\pi_\tau(A)''$. 
We identify $A$ with $\pi(A)$ and omit $\pi$. 
The cocycle action $(\alpha,u)$ on $A$ naturally extends to 
the cocycle action $(\bar\alpha,u)$ on $M$. 
Note that, for a bounded sequence $(x_n)_n$ in $M$, 
$x_n$ converges to zero in the strong operator topology 
if and only if $\lVert x_n\rVert_2$ converges to zero. 

Suppose that we are given a finite subset $F\subset\Gamma$ and $\ep>0$. 
Let $R$ be the AFD II$_1$-factor and 
let $\gamma:\Gamma\curvearrowright R$ be an outer action. 
Since $\Gamma$ has the property (Q), 
there exist an $(F,\ep)$-invariant finite subset $K\subset\Gamma$ and 
a sequence of projections $(p_n)_n$ in $R$ such that 
\[
\left\lVert1-\sum_{g\in K}\gamma_g(p_n)\right\rVert_2\to0\quad\text{and}\quad 
\lVert[x,p_n]\rVert_2\to0\quad \forall x\in R
\]
as $n\to\infty$. 

First, we would like to show that 
there exists a sequence of projections $(q_n)_n$ in $M$ such that 
\[
\lVert[x,q_n]\rVert_2\to0,\quad 
\lVert\bar\alpha_g(q_n)\bar\alpha_h(q_n)\rVert_2\to0\quad\text{and}\quad 
\tau(q_n)\to1/\lvert K\rvert
\]
for all $x\in M$, $g,h\in K$ with $g\neq h$ and $\tau\in T(M)$ 
as $n\to\infty$. 
The cocycle action $\bar\alpha$ induces an action of $\Gamma$ 
on the set of minimal central projections in $M$. 
For each $\Gamma$-orbit we choose and fix a minimal central projection. 
Let $E_0$ denote the set of such projections. 
There exist finite subsets $H_e\subset\Gamma$ for each $e\in E_0$ such that 
\[
\sum_{e\in E_0}\sum_{h\in H_e}\bar\alpha_h(e)=1. 
\]
Put $\Gamma_e=\{g\in\Gamma\mid\bar\alpha_g(e)=e\}$. 
Then $\Gamma_e$ is a subgroup of $\Gamma$ and 
\[
\Gamma=\bigcup_{h\in H_e}h\Gamma_e
\]
holds for every $e\in E_0$. 
Let $(\bar\alpha^e,u^e)$ be the restriction of $(\bar\alpha,u)$ 
to $\Gamma_e$ and $Me$, 
i.e. $\bar\alpha^e_g=\bar\alpha_g|Me$, $u^e(g,h)=u(g,h)e$ 
for $g,h\in\Gamma_e$. 
Since $(\alpha,u)$ is strongly outer, 
the cocycle action $(\bar\alpha^e,u^e)$ of $\Gamma_e$ on $Me$ is outer. 
Likewise, we let $\gamma^e$ denote the restriction of 
$\gamma:\Gamma\curvearrowright R$ to $\Gamma_e$. 
It follows from \cite[Chapter 1]{O} (see also \cite[Theorem 2.3]{Msd}) that 
$(\bar\alpha^e,u^e)$ is cocycle conjugate to 
$(\bar\alpha^e\otimes\gamma^e,u^e\otimes1)$. 
Hence, there exist an isomorphism $\pi_e:Me\bar\otimes R\to Me$ and 
a family of unitaries $(v^e_g)_{g\in\Gamma_e}$ in $Me$ such that 
\[
\Ad v^e_g\circ\bar\alpha_g
=\pi_e\circ(\bar\alpha^e_g\otimes\gamma^e_g)\circ\pi_e^{-1}
\quad \forall g\in\Gamma_e. 
\]
We define an isomorphism $\pi:M\bar\otimes R\to M$ by 
\[
\pi((\bar\alpha_h\otimes\gamma_h)(x))=\bar\alpha_h(\pi_e(x))\quad 
\forall x\in Me\bar\otimes R,\ h\in H_e. 
\]
Let $q_n=\pi(1\otimes p_n)$. 
We would like to show 
$\lVert\bar\alpha_g(q_n)-\pi(1\otimes\gamma_g(p_n))\rVert_2\to0$ 
for every $g\in\Gamma$ as $n\to\infty$. 
Once this is done, it is evident that 
$(q_n)_n$ is the desired sequence of projections. 
We have 
\begin{align*}
q_n&=\pi(1\otimes p_n)\\
&=\sum_{e\in E_0}\sum_{h\in H_e}\pi(\bar\alpha_h(e)\otimes p_n)\\
&=\sum_{e\in E_0}\sum_{h\in H_e}
\pi((\bar\alpha_h\otimes\gamma_h)(e\otimes\gamma_h^{-1}(p_n)))\\
&=\sum_{e\in E_0}\sum_{h\in H_e}
\bar\alpha_h(\pi_e(e\otimes\gamma_h^{-1}(p_n))). 
\end{align*}
Take $g\in\Gamma$. 
For each $h\in H_e$ there exists $k\in H_e$ such that $gh\in k\Gamma_e$, 
and the map $h\mapsto k$ is a bijection. 
Then one can verify 
\begin{align*}
&\lim_{n\to\infty}
\lVert\bar\alpha_g(\bar\alpha_h(\pi_e(e\otimes\gamma_h^{-1}(p_n))))
-\pi(\bar\alpha_k(e)\otimes\gamma_g(p_n))\rVert_2\\
&=\lim_{n\to\infty}
\lVert\bar\alpha_k(\bar\alpha_{k^{-1}gh}(\pi_e(e\otimes\gamma_h^{-1}(p_n))))
-\pi(\bar\alpha_k(e)\otimes\gamma_g(p_n))\rVert_2\\
&=\lim_{n\to\infty}
\lVert\bar\alpha_k(\pi_e((\bar\alpha_{k^{-1}gh}\otimes\gamma_{k^{-1}gh})
(e\otimes\gamma_h^{-1}(p_n))))
-\pi(\bar\alpha_k(e)\otimes\gamma_g(p_n))\rVert_2\\
&=\lim_{n\to\infty}
\lVert\bar\alpha_k(\pi_e(e\otimes\gamma_{k^{-1}g}(p_n))
-\pi(\bar\alpha_k(e)\otimes\gamma_g(p_n))\rVert_2=0. 
\end{align*}
Therefore we get 
\begin{align*}
&\lim_{n\to\infty}
\lVert\bar\alpha_g(q_n)-\pi(1\otimes\gamma_g(p_n))\rVert_2\\
&=\lim_{n\to\infty}
\left\lVert\sum_{e\in E_0}\sum_{h\in H_e}
\bar\alpha_g(\bar\alpha_h(\pi_e(e\otimes\gamma_h^{-1}(p_n))))
-\pi(1\otimes\gamma_g(p_n))\right\rVert_2\\
&=\lim_{n\to\infty}
\left\lVert\sum_{e\in E_0}\sum_{k\in H_e}
\pi(\bar\alpha_k(e)\otimes\gamma_g(p_n))
-\pi(1\otimes\gamma_g(p_n))\right\rVert_2=0. 
\end{align*}

By Kaplansky's density theorem, we may replace $(q_n)_n$ 
with a sequence $(x_n)_n$ of positive contractions in $A$. 
Thanks to Proposition \ref{Haagerup2}, we can further replace $(x_n)_n$ 
with a sequence $(y_n)_n$ of positive contractions in $A$ satisfying 
\[
\lVert[a,y_n]\rVert\to0,\quad 
\lVert\alpha_g(y_n)\alpha_h(y_n)\rVert_2\to0\quad\text{and}\quad 
\tau(y_n)\to1/\lvert K\rvert
\]
for all $a\in A$, $g,h\in K$ with $g\neq h$ and $\tau\in T(A)$ 
as $n\to\infty$. 
Finally, by the same way as \cite[Proposition 3.3]{MS}, 
we obtain a sequence $(z_n)_n$ of positive contractions in $A$ satisfying 
\[
\lVert[a,z_n]\rVert\to0,\quad 
\lVert\alpha_g(z_n)\alpha_h(z_n)\rVert\to0\quad\text{and}\quad 
\tau(z_n)\to1/\lvert K\rvert
\]
for all $a\in A$, $g,h\in K$ with $g\neq h$ and $\tau\in T(A)$ 
as $n\to\infty$. 
Consequently $(\alpha,u)$ has the weak Rohlin property. 
\end{proof}

For a $C^*$-algebra $A$ with tracial rank zero, 
we can prove that 
strong outerness is equivalent to the tracial Rohlin property 
in the sense of Definition \ref{so&wR} (4). 

\begin{theorem}\label{so=trR}
Let $A$ be a unital simple separable $C^*$-algebra with tracial rank zero and 
with finitely many extremal tracial states. 
Let $\Gamma$ be a countable discrete amenable group with the property (Q). 
For a cocycle action $(\alpha,u)$ of $\Gamma$ on $A$, 
the following conditions are equivalent. 
\begin{enumerate}
\item $(\alpha,u)$ is strongly outer. 
\item $(\alpha,u)$ has the tracial Rohlin property 
in the sense of Definition \ref{so&wR} (4). 
\end{enumerate}
\end{theorem}
\begin{proof}
The implication (2)$\Rightarrow$(1) can be shown 
in a similar fashion to \cite[Remark 2.8]{MS}. 
We prove (1)$\Rightarrow$(2). 
Suppose that $(\alpha,u):\Gamma\curvearrowright A$ is strongly outer. 
Define $M$ and $(\bar\alpha,u):\Gamma\curvearrowright M$ 
in the same way as Theorem \ref{so=wR}. 
It is well-known that 
$M$ is isomorphic to a finite direct sum of 
the hyperfinite II$_1$-factor 
(see \cite[Lemma 2.16]{OP06ETDS2}, \cite[Remark 2.6 (8)]{MS}). 
Exactly the same argument as Theorem \ref{so=wR} shows that 
there exists a sequence of projections $(q_n)_n$ in $M$ such that 
\[
\lVert[x,q_n]\rVert_2\to0,\quad 
\lVert\bar\alpha_g(q_n)\bar\alpha_h(q_n)\rVert_2\to0\quad\text{and}\quad 
\tau(q_n)\to1/\lvert K\rvert
\]
for all $x\in M$, $g,h\in K$ with $g\neq h$ and $\tau\in T(M)$ 
as $n\to\infty$. 
By \cite[Lemma 2.15]{OP06ETDS2}, 
we may replace $q_n$ with projections $e_n$ in $A$, that is, 
there exists a sequence of projections $(e_n)_n$ in $A$ such that 
\[
\lVert[a,e_n]\rVert_2\to0,\quad 
\lVert\alpha_g(e_n)\alpha_h(e_n)\rVert_2\to0\quad\text{and}\quad 
\tau(e_n)\to1/\lvert K\rvert
\]
for all $a\in A$, $g,h\in K$ with $g\neq h$ and $\tau\in T(A)$ 
as $n\to\infty$. 
Then, from \cite[Proposition 4.1]{M10CMP}, 
we can find a sequence $(f_n)_n$ of projections in $A$ such that 
\[
\lVert[a,f_n]\rVert\to0,\quad 
\lVert\alpha_g(f_n)\alpha_h(f_n)\rVert\to0\quad\text{and}\quad 
\tau(f_n)\to1/\lvert K\rvert
\]
for all $a\in A$, $g,h\in K$ with $g\neq h$ and $\tau\in T(A)$ 
as $n\to\infty$. 
Thus $(\alpha,u)$ has the tracial Rohlin property. 
\end{proof}

%%%%%%%%%%%%%%%%%%%%%%%%%%%%%%%%%%%%%%%%%%%%%%%%%%%%%%%%%%%%
\section{$\mathcal{Z}$-stability of crossed products}

In this section, we prove Theorem \ref{wR>Z}. 
We also prove that 
the fixed point subalgebra of the central sequence algebra 
has strict comparison (Theorem \ref{sc}). 
%As an application, 
%we establish a certain Rohlin property for strongly outer cocycle actions 
%(Theorem ?). 

%%%%%%%%%%%%%%%%%%%%%%%%%%%%%%%%%%%%%%%%%%%%%%%%%%%%%%%%%%%%
\subsection{Fixed point subalgebras of central sequence algebras}

Throughout this subsection, we let $A$ be 
a unital, simple, separable, nuclear, stably finite, infinite dimensional 
$C^*$-algebra with finitely many extremal tracial states. 
We also fix a countable discrete amenable group $\Gamma$ and 
a cocycle action $(\alpha,u):\Gamma\curvearrowright A$ 
with the weak Rohlin property. 
Note that we allow $\Gamma$ to be trivial. 

Let $E$ be the set of extremal tracial states on $A$. 
Set $\pi=\bigoplus_{\tau\in E}\pi_\tau$ and 
$M=\pi(A)''=\bigoplus_{\tau\in E}\pi_\tau(A)''$. 
It is well-known that 
$\pi_\tau(A)''$ is the AFD II$_1$-factor $R$ for $\tau\in E$. 
We identify $A$ with $\pi(A)$ and omit $\pi$. 
It is clear that $T(A)$ is identified with $T(M)$. 
We let 
\[
\mathcal{R}=(\ell^\infty(\N,R)
/\{(x_n)_n\mid\lVert x_n\rVert_2\to0\text{ as }n\to\omega\})\cap R', 
\]
be the central sequence algebra of $R$ in the sense of von Neumann algebra, 
where $R$ is identified with the subalgebra 
consisting of equivalence classes of constant sequences. 
It is also well-known that $\mathcal{R}$ is a II$_1$-factor 
(see \cite[Theorem XIV.4.18]{TIII} for example). 
Similarly, we define $\mathcal{M}$ by 
\[
\mathcal{M}=(\ell^\infty(\N,M)
/\{(x_n)_n\mid\lVert x_n\rVert_2\to0\text{ as }n\to\omega\})\cap M'. 
\]
Clearly $\mathcal{M}$ is isomorphic to $\bigoplus_{\tau\in E}\mathcal{R}$. 
In particular, $T(\mathcal{M})$ is canonically identified with $T(M)$, 
and hence with $T(A)$. 
We let $T(A)/\Gamma\cong T(M)/\Gamma\cong T(\mathcal{M})/\Gamma$ denote 
the quotient space by the $\Gamma$-action $\tau\mapsto\tau\circ\alpha_g$. 
The cocycle action $(\alpha,u):\Gamma\curvearrowright A$ gives rise to 
the cocycle action $(\bar\alpha,u):\Gamma\curvearrowright M$. 
Notice that 
$\alpha:\Gamma\curvearrowright A_\omega$ and 
$\bar\alpha:\Gamma\curvearrowright\mathcal{M}$ are genuine actions. 

The inclusion map $A\hookrightarrow M$ induces the homomorphism 
\[
\rho:A_\omega\to\mathcal{M}. 
\]
Put $J=\Ker\rho$. 
For an $\omega$-central sequence $(x_n)_n$ in $A$, 
its image in $A_\omega$ belongs to $J$ if and only if 
$\lVert x_n\rVert_2\to0$ as $n\to\omega$. 
In other words, for $x\in A_\omega$, 
$x$ belongs to $J$ if and only if 
$\tau_\omega(x^*x)=0$ for any $\tau\in T(A)$ 
(see Section 2.1 for the definition of $\tau_\omega$). 
The action $\alpha$ globally preserves $J$, 
and so the fixed point subalgebra $J^\alpha$ is well-defined. 
Throughout this subsection, we keep the above notations. 

The following lemma is folklore among experts. 
For the reader's convenience, we include a proof here. 

\begin{lemma}
Let $(\gamma,w):\Gamma\curvearrowright R$ be 
an outer cocycle action of $\Gamma$ on the AFD II$_1$-factor $R$. 
Then $\mathcal{R}^\gamma$ is a II$_1$-factor. 
\end{lemma}
\begin{proof}
By \cite[Chapter 1]{O} (see also \cite[Theorem 2.3]{Msd}), 
$(\gamma,w)$ is cocycle conjugate to 
$(\gamma\otimes\id,w\otimes1):\Gamma\curvearrowright R\bar\otimes R$. 
Hence there exists a unital homomorphism $\phi:R\to\mathcal{R}^\gamma$ 
(or one may use \cite[Lemma 8.3]{O}, 
which says that $\mathcal{R}^\gamma$ is of type II$_1$). 
Let $\tau$ be the unique trace on $\mathcal{R}$. 
Suppose that 
$p$ is a projection belonging to the center of $\mathcal{R}^\gamma$. 
Take a projection $q\in R$ such that $\tau(\phi(q))=\tau(p)$. 
There exists a unitary $x\in\mathcal{R}$ such that $xpx^*=\phi(q)$, 
because $\mathcal{R}$ is a II$_1$-factor. 
Then $(x\gamma_g(x^*))_g$ is 
a $\gamma$-cocycle in $\mathcal{R}\cap\{\phi(q)\}'$. 
By the cohomology vanishing theorem \cite[Proposition 7.2]{O}, 
we can find a unitary $y\in\mathcal{R}\cap\{\phi(q)\}'$ such that 
$x\gamma_g(x^*)=y\gamma_g(y^*)$ for all $g\in\Gamma$. 
Therefore $y^*xpx^*y=\phi(q)$ and $y^*x\in\mathcal{R}^\gamma$. 
Since $p$ is in the center of $\mathcal{R}^\gamma$, we obtain $p=\phi(q)$. 
This means $p=0$ or $p=1$, 
and whence $\mathcal{R}^\gamma$ is a II$_1$-factor. 
\end{proof}

\begin{lemma}\label{II_1}
\begin{enumerate}
\item The fixed point algebra $\mathcal{M}^{\bar\alpha}$ is 
a finite direct sum of II$_1$-factors. 
\item The map $\tau\mapsto\tau|\mathcal{M}^{\bar\alpha}$ induces 
an isomorphism from $T(\mathcal{M})/\Gamma$ to $T(\mathcal{M}^{\bar\alpha})$. 
\end{enumerate}
\end{lemma}
\begin{proof}
(1) 
The cocycle action $\bar\alpha$ induces an action of $\Gamma$ 
on the set of minimal central projections of $M$. 
For each $\Gamma$-orbit we choose and fix a minimal central projection. 
Let $\{e_1,e_2,\dots,e_n\}$ denote the set of such projections. 
For each $i=1,2,\dots,n$, 
there exist finite subsets $H_i\subset\Gamma$ such that 
\[
\sum_{i=1}^n\sum_{h\in H_i}\bar\alpha_h(e_i)=1. 
\]
Put $\Gamma_i=\{g\in\Gamma\mid\bar\alpha_g(e_i)=e_i\}$. 
Then $\Gamma_i$ is a subgroup of $\Gamma$ and 
\[
\Gamma=\bigcup_{h\in H_i}h\Gamma_i
\]
holds. 
Let $(\bar\alpha^{(i)},u^{(i)}):\Gamma_i\curvearrowright Me_i$ be 
the restriction of $(\bar\alpha,u)$ to $\Gamma_i$ and $Me_i$. 
It is not so hard to see 
\[
\mathcal{M}^{\bar\alpha}=\left\{\sum_{i=1}^n\sum_{h\in H_i}\bar\alpha_h(x)\mid 
x\in(\mathcal{M}e_i)^{\bar\alpha^{(i)}}\right\}, 
\]
which is isomorphic to 
\[
\bigoplus_{i=1}^n(\mathcal{M}e_i)^{\bar\alpha^{(i)}}. 
\]
Since $(\bar\alpha^{(i)},u^{(i)}):\Gamma_i\curvearrowright Me_i$ 
is an outer cocycle action of a countable discrete amenable group 
on the AFD II$_1$-factor, by the lemma above, 
the fixed point algebra $(\mathcal{M}e_i)^{\bar\alpha^{(i)}}$ is 
a II$_1$-factor, which completes the proof. 

(2) immediately follows from the proof of (1). 
\end{proof}

\begin{theorem}\label{exact}
\begin{enumerate}
\item The homomorphism $\rho:A_\omega\to\mathcal{M}$ is surjective, that is, 
\[
\begin{CD}
0@>>>J@>>>A_\omega@>\rho>>\mathcal{M}@>>>0
\end{CD}
\]
is exact. 
\item The restriction of $\rho$ to $(A_\omega)^\alpha$ induces 
the following short exact sequence: 
\[
\begin{CD}
0@>>>J^\alpha@>>>(A_\omega)^\alpha@>\rho>>\mathcal{M}^{\bar\alpha}@>>>0. 
\end{CD}
\]
\end{enumerate}
\end{theorem}
\begin{proof}
(1)
Take a positive contraction $x\in\mathcal{M}$. 
Let $(x_n)_n\in\ell^\infty(\N,M)$ be a representative sequence of $x$ 
consisting of positive contractions. 
By Kaplansky's density theorem, we may assume that $x_n$ is in $A$. 
We choose an increasing sequence $(F_m)_m$ of finite subsets of $A$ 
whose union is dense in $A$. 
Applying Proposition \ref{Haagerup2} to $F_m$ and $1/m$, 
we get a finite subset $G_m\subset A$ and $\delta_m>0$. 
Define 
\[
\Omega_m=\{n\in\N\mid n\geq m,\quad 
\lVert[x_n,a]\rVert_2<1/\delta_m\quad\forall a\in G_m\}. 
\]
Then $\Omega_m$ belongs to $\omega$ because $x$ is in $\mathcal{M}$. 
For each $n\in\Omega_m\setminus\Omega_{m+1}$, 
by means of Proposition \ref{Haagerup2}, 
we obtain a positive contraction $y_n\in A$ such that 
\[
\lVert x_n-y_n\rVert_2<1/m\quad\text{and}\quad 
\lVert[y_n,a]\rVert<1/m\quad\forall a\in F_m. 
\]
Put $y_n=0$ for $n\in\N\setminus\Omega_1$. 
It is easy to see that $(y_n)_n$ is an $\omega$-central sequence. 
Let $y\in A_\omega$ be the image of $(y_n)_n$. 
Since $\lVert x_n-y_n\rVert_2\to0$ as $n\to\omega$, 
we have $\rho(y)=x$ which completes the proof. 

(2)
Clearly 
$\rho((A_\omega)^\alpha)$ is contained in $\mathcal{M}^{\bar\alpha}$ 
and the kernel of $\rho|(A_\omega)^\alpha$ equals $J^\alpha$. 
It remains to show the surjectivity of $\rho|(A_\omega)^\alpha$. 

Take a positive element $x\in\mathcal{M}^{\bar\alpha}$. 
We would like to prove that $\rho((A_\omega)^\alpha)$ contains $x$. 
Since $\rho$ is surjective, 
there exists $a\in A_\omega$ such that $\rho(a)=x$. 
Choose a finite subset $F\subset\Gamma$, $\ep>0$ and $m\in\N$ arbitrarily. 
Choose $\delta>0$ so that 
\[
\delta(1+\lvert F\rvert+\dots+\lvert F\rvert^{m-1})<\ep. 
\]
Since $(\alpha,u)$ has the weak Rohlin property, 
there exists an $(F,\delta)$-invariant finite subset $K\subset\Gamma$ 
and a positive contraction $f\in A_\omega$ such that 
\[
\tau_\omega(f)=\vert K\rvert^{-1}\quad\text{and}\quad 
\alpha_g(f)\alpha_h(f)=0
\]
for all $\tau\in T(A)$ and $g,h\in K$ with $g\neq h$. 
Evidently $(\rho(\alpha_g(f)))_{g\in K}$ forms 
a partition of unity consisting of projections in $\mathcal{M}$. 
By the standard reindexation trick, we may assume that $f$ commutes with $a$. 
Define a function $\ell:K\to\N$ by 
\[
\ell(g)=\min\{n\in\N\mid\exists g_1,g_2,\dots,g_n\in F,\ 
gg_1\dots g_n\notin K\}. 
\]
One has $\lvert\ell^{-1}(1)\rvert\leq\delta\lvert K\rvert$, 
because $K$ is $(F,\delta)$-invariant. 
Then 
$\lvert\ell^{-1}(n)\rvert\leq\delta\lvert F\rvert^{n-1}\lvert K\rvert$ 
is obtained inductively, so that 
\[
\lvert\ell^{-1}(\{1,2,\dots,m\})\rvert
\leq\delta(1+\lvert F\rvert+\dots+\lvert F\rvert^{m-1})
\lvert K\rvert<\ep\lvert K\rvert. 
\]
Let $h:K\to[0,1]$ be a function 
defined by $h(g)=m^{-1}\min\{\ell(g){-}1,m\}$. 
The estimate above implies 
$\lvert h^{-1}(1)\rvert>(1{-}\ep)\lvert K\rvert$. 
We define positive contractions $b,c\in A_\omega$ by 
\[
b=\sum_{g\in K}\alpha_g(af)\quad\text{and}\quad 
c=\sum_{g\in K}h(g)\alpha_g(af)
\]
It is easy to see that 
$\lVert c-\alpha_g(c)\rVert$ is less than $1/m$ for all $g\in F$. 
Also, we have 
\[
\rho(b)=\sum_{g\in K}\bar\alpha_g(x)\bar\alpha_g(\rho(f))=x, 
\]
and so $a-b\in J$. 
Furthermore, $c\leq b$ and 
\[
\tau_\omega(b-c)\leq\sum_{h(g)\neq1}\tau_\omega(\alpha_g(af))
\leq\sum_{h(g)\neq1}\tau_\omega(\alpha_g(f))<\ep
\]
for all $\tau\in T(A)$. 
Hence $\tau_\omega((a-c)^*(a-c))=\tau_\omega((b-c)^*(b-c))<\ep$ holds 
for all $\tau\in T(A)$. 

Since $F$, $\ep$ and $m$ were arbitrary, 
the standard trick on central sequences implies that 
there exists a positive contraction $c\in(A_\omega)^\alpha$ 
such that $a-c\in J$, 
which completes the proof. 
\end{proof}

\begin{lemma}\label{appidentity}
For any $x\in J^\alpha$, 
there exists a positive contraction $e\in J^\alpha$ such that $ex=xe=x$. 
\end{lemma}
\begin{proof}
Take $x\in J^\alpha$. 
Let $(x_n)_n$ be a representative sequence of $x$. 
We choose an increasing sequence $(F_m)_m$ of finite subsets of $A$ 
whose union is dense in $A$. 
Let $\Gamma=\{g_m\mid m\in\N\}$. 
For $m\in\N$, define a continuous function $h_m$ on $[0,\infty)$ by 
\[
h_m(t)=\begin{cases}mt&0\leq t\leq1/m\\1&1/m\leq t. \end{cases}
\]
Put $y_n=x_n^*x_n+x_nx_n^*$. 
It is easy to see that 
$(y_n)_n$ is an $\omega$-central sequence of positive elements 
satisfying $\tau(y_n)\to0$ and $\alpha_g(y_n)-y_n\to0$ as $n\to\omega$ 
for any $\tau\in T(A)$ and $g\in\Gamma$. 
Let $\Omega_m$ be the set of all natural numbers $n\geq m$ such that 
\[
\tau(h_m(y_n))<\frac{1}{m},\quad 
\lVert \alpha_{g_i}(h_m(y_n))-h_m(y_n)\rVert<\frac{1}{m}\quad\text{and}\quad 
\lVert[h_m(y_n),a]\rVert<\frac{1}{m}
\]
for all $\tau\in T(A)$, $i=1,2,\dots,m$ and $a\in F_m$. 
Then $\Omega_m$ belongs to $\omega$. 
Set $\widetilde\Omega_m=\Omega_1\cap\Omega_2\cap\dots\cap\Omega_m$. 
Define $e_n\in A$ by 
\[
e_n=\begin{cases}h_m(y_n)&
n\in\widetilde\Omega_m\setminus\widetilde\Omega_{m+1}\\
0&n\in\N\setminus\widetilde\Omega_1. \end{cases}
\]
Then $(e_n)_n$ is an $\omega$-central sequence 
satisfying $\tau(e_n)\to0$ and $\alpha_g(e_n)-e_n\to0$ as $n\to\omega$ 
for all $\tau\in T(A)$ and $g\in\Gamma$. 
Thus, the image $e\in A_\omega$ of $(e_n)_n$ is in $J^\alpha$. 
Moreover, for every $n\in\widetilde\Omega_m\setminus\widetilde\Omega_{m+1}$, 
we have 
\begin{align*}
\lVert x_n-e_nx_n\rVert^2
&=\lVert(1-h_m(y_n))x_nx_n^*(1-h_m(y_n))\rVert\\
&\leq\lVert(1-h_m(y_n))y_n(1-h_m(y_n))\rVert<m^{-1}. 
\end{align*}
Hence $x=ex$. 
Likewise, we obtain $x=xe$, which completes the proof. 
\end{proof}

In the following proposition, 
we consider an equivariant version of the property (SI). 
Notice that we need not assume that 
$A$ has finitely many extremal tracial states for this proposition. 
See Definition \ref{DefofSI} for the definition of the property (SI). 

\begin{proposition}\label{equivSI}
Suppose that $A$ has the property (SI). 
For any $\omega$-central sequences $(x_n)_n$ and $(y_n)_n$ of 
positive contractions in $A$ satisfying 
\[
\lim_{n\to\omega}\max_{\tau\in T(A)}\tau(x_n)=0,\quad 
\inf_{m\in\N}\lim_{n\to\omega}\min_{\tau\in T(A)}\tau(y_n^m)>0, 
\]
and 
\[
\lim_{n\to\omega}\lVert\alpha_g(x_n)-x_n\rVert
=\lim_{n\to\omega}\lVert\alpha_g(y_n)-y_n\rVert=0\quad\forall g\in\Gamma, 
\]
there exists an $\omega$-central sequence $(s_n)_n$ in $A$ such that 
\[
\lim_{n\to\omega}\lVert s_n^*s_n-x_n\rVert=0,\quad 
\lim_{n\to\omega}\lVert y_ns_n-s_n\rVert=0, 
\]
and 
\[
\lim_{n\to\omega}\lVert\alpha_g(s_n)-s_n\rVert=0\quad\forall g\in\Gamma. 
\]
\end{proposition}
\begin{proof}
The assertion is actually contained in the proof of \cite[Theorem 4.7]{MS}, 
but we include a proof here for completeness. 

We first note that, by Remark \ref{RemofSI},  
the assertion is equivalent to the property (SI) itself, 
when $\Gamma$ is trivial. 
Let us consider the general case. 
Let $(x_n)_n$ and $(y_n)_n$ be as in the statement. 
Put 
\[
\nu=\inf_{m\in\N}\lim_{n\to\omega}\min_{\tau\in T(A)}\tau(y_n^m)>0. 
\]
Take a finite subset $F\subset\Gamma$ and $\ep>0$ arbitrarily. 
Since $(\alpha,u):\Gamma\curvearrowright A$ has the weak Rohlin property, 
we can find an $(F,\ep)$-invariant finite subset $K\subset\Gamma$ and 
a central sequence $(f_l)_l$ of positive contractions satisfying 
\[
\lim_{l\to\infty}
\max_{\tau\in T(A)}\lvert\tau(f_l)-\lvert K\rvert^{-1}\rvert=0
\quad\text{and}\quad 
\lim_{l\to\infty}\alpha_g(f_l)\alpha_h(f_l)=0
\]
for all $g,h\in K$ with $g\neq h$. 

We claim that there exists an $\omega$-central sequence $(\tilde y_n)_n$ 
of positive contractions in $A$ such that 
\[
\tilde y_n\leq y_n,\quad 
\inf_{m\in\N}\lim_{n\to\omega}
\min_{\tau\in T(A)}\tau(\tilde y_n^m)=\nu\lvert K\rvert^{-1}
\quad\text{and}\quad 
\lim_{n\to\omega}\alpha_g(\tilde y_n)\alpha_h(\tilde y_n)=0
\]
for all $g,h\in K$ with $g\neq h$. 
To this end, it suffices to show that 
for any finite subset $G\subset A$, $\delta>0$ and $m\in\N$, 
there exists a sequence $(\tilde y_n)_n$ of positive contractions in $A$ 
such that 
\[
\tilde y_n\leq y_n,\quad 
\lim_{n\to\omega}
\min_{\tau\in T(A)}\tau(\tilde y_n^m)>\nu\lvert K\rvert^{-1}-2\delta,\quad 
\lim_{n\to\omega}\lVert\alpha_g(\tilde y_n)\alpha_h(\tilde y_n)\rVert<\delta
\]
for all $g,h\in K$ with $g\neq h$ and 
\[
\lim_{n\to\omega}\lVert[\tilde y_n,a]\rVert<\delta\quad\forall a\in G. 
\]
Choose $l\in\N$ so that 
\[
\max_{\tau\in T(A)}\lvert\tau(f_l)-\lvert K\rvert^{-1}\rvert<\delta,\quad 
\lVert\alpha_g(f_l^{1/m})\alpha_h(f_l^{1/m})\rVert<\delta
\]
for all $g,h\in K$ with $g\neq h$ and 
\[
\lVert[f_l^{1/m},a]\rVert<\delta\quad\forall a\in G. 
\]
Set $\tilde y_n=y_n^{1/2}f_l^{1/m}y_n^{1/2}$. 
Clearly $\tilde y_n\leq y_n$. 
By means of \cite[Lemma 4.6]{MS}, we get 
\begin{align*}
\lim_{n\to\omega}\min_{\tau\in T(A)}\tau(\tilde y_n^m)
&=\lim_{n\to\omega}\min_{\tau\in T(A)}\tau(y_n^mf_l)\\
&>\lim_{n\to\omega}\min_{\tau\in T(A)}\tau(y_n^m)\lvert K\rvert^{-1}-2\delta
\geq\nu-2\delta. 
\end{align*}
We also have 
\[
\lim_{n\to\omega}\lVert\alpha_g(\tilde y_n)\alpha_h(\tilde y_n)\rVert
\leq\lVert\alpha_g(f_l^{1/m})\alpha_h(f_l^{1/m})\rVert<\delta
\]
for all $g,h\in K$ with $g\neq h$ and 
\[
\lim_{n\to\omega}\lVert[\tilde y_n,a]\rVert
\leq\lVert[f_l^{1/m},a]\rVert<\delta
\]
for all $a\in G$. 
Thus, the proof of the claim is completed. 

Because $A$ has the property (SI) (see also Remark \ref{RemofSI}), 
there exists an $\omega$-central sequence $(r_n)_n$ in $A$ such that 
\[
r_n^*r_n-x_n\to0\quad\text{and}\quad 
\tilde y_nr_n-r_n\to0
\]
as $n\to\omega$. 
Define $s_n\in A$ by 
\[
s_n=\frac{1}{\sqrt{\lvert K\rvert}}\sum_{g\in K}\alpha_g(r_n). 
\]
Then it is easy to see 
\[
s_n^*s_n-x_n\to0\quad\text{and}\quad y_ns_n-s_n\to0
\]
as $n\to\omega$. 
Besides, for any $h\in F$, 
\begin{align*}
\lim_{n\to\omega}\lVert s_n-\alpha_h(s_n)\rVert
&\leq\lim_{n\to\omega}\frac{1}{\sqrt{\lvert K\rvert}}
\left\lVert\sum_{g\in K\setminus h^{-1}K}\alpha_g(r_n)\right\rVert
+\frac{1}{\sqrt{\lvert K\rvert}}
\left\lVert\sum_{g\in K\setminus hK}\alpha_g(r_n)\right\rVert \\
&\leq\frac{\lvert K\setminus h^{-1}K\rvert^{1/2}}{\lvert K\rvert^{1/2}}
+\frac{\lvert K\setminus hK\rvert^{1/2}}{\lvert K\rvert^{1/2}}
\leq2\sqrt{\ep}. 
\end{align*}
Since $F\subset\Gamma$ and $\ep>0$ were arbitrary, the proof is completed. 
\end{proof}

\begin{lemma}\label{taukillJ}
Suppose that $A$ has the property (SI). 
For any $\tau\in T((A_\omega)^\alpha)$ and $x\in J^\alpha$, 
one has $\tau(x)=0$. 
\end{lemma}
\begin{proof}
Let $x\in J^\alpha$ be a positive contraction. 
It suffices to show $\tau(x)=0$ for all $\tau\in T((A_\omega)^\alpha)$. 
By Lemma \ref{appidentity}, 
we can find a positive contraction $e_1\in J^\alpha$ such that $xe_1=x$. 
By Proposition \ref{equivSI}, 
there exists $s_1\in J^\alpha$ such that $s_1^*s_1=x$ and $(1-e_1)s_1=s_1$. 
Applying Lemma \ref{appidentity} to the positive contraction $x+s_1s_1^*$, 
we obtain a positive contraction $e_2\in J^\alpha$ 
such that $(x+s_1s_1^*)e_2=x+s_1s_1^*$. 
By Proposition \ref{equivSI}, 
there exists $s_2\in J^\alpha$ such that $s_2^*s_2=x$ and $(1-e_2)s_2=s_2$. 
Then $x+s_1s_1^*+s_2s_2^*$ is again a positive contraction in $J^\alpha$. 
Repeating this argument, 
one can find $(s_i)_i\subset J^\alpha$ satisfying 
\[
s_i^*s_i=x\quad\text{and}\quad 
x+s_1s_1^*+s_2s_2^*+\dots+s_ns_n^*\leq1
\]
for every $n\in\N$. 
Hence $\tau(x)=0$ for all $\tau\in T((A_\omega)^\alpha)$. 
\end{proof}

\begin{lemma}
Suppose that $A$ has the property (SI). 
\begin{enumerate}
\item The map $\tau\mapsto\tau_\omega|(A_\omega)^\alpha$ induces 
an isomorphism from $T(A)/\Gamma$ to $T((A_\omega)^\alpha)$. 
\item The map $\tau\mapsto\tau_\omega$ is 
an isomorphism from $T(A)$ to $T(A_\omega)$. 
\item If $A$ has a unique tracial state, then so does $(A_\omega)^\alpha$. 
\end{enumerate}
\end{lemma}
\begin{proof}
(1) immediately follows from 
Lemma \ref{II_1}, Theorem \ref{exact} and Lemma \ref{taukillJ}. 
(2) and (3) are direct consequences of (1). 
\end{proof}

The following proposition says that 
$(A_\omega)^\alpha$ has strict comparison in a certain sense. 
We remark that the condition 
$d_{\tau_\omega}(a)<d_{\tau_\omega}(b)$ for all $\tau\in T(A)$ is equivalent 
to the condition $d_\tau(a)<d_\tau(b)$ for all $\tau\in T((A_\omega)^\alpha)$, 
thanks to the lemma above. 

\begin{proposition}\label{sc}
Suppose that $A$ has the property (SI). 
Let $a,b\in(A_\omega)^\alpha$ be positive elements 
satisfying $d_{\tau_\omega}(a)<d_{\tau_\omega}(b)$ for all $\tau\in T(A)$. 
Then there exists $r\in(A_\omega)^\alpha$ such that $rbr^*=a$. 
\end{proposition}
\begin{proof}
For a subset $X\subset\R$, 
we let $1_X$ be the characteristic function of $X$. 
Notice that we keep the identification 
between $T(A)$, $T(M)$ and $T(\mathcal{M})$. 
By $\tau\circ\rho=\tau_\omega$, we have 
\[
d_{\tau_\omega}(a)=\lim_{n\to\infty}\tau_\omega(a^{1/n})
=\lim_{n\to\infty}\tau(\rho(a)^{1/n})
=\tau(1_{(0,\infty)}(\rho(a))), 
\]
where the support projection $1_{(0,\infty)}(\rho(a))$ is well-defined 
because $\mathcal{M}^{\bar\alpha}$ is a von Neumann algebra. 
Similarly we get 
\[
d_{\tau_\omega}(b)=\tau(1_{(0,\infty)}(\rho(b))). 
\]
Choose $k\in\N$ so that 
\[
d_{\tau_\omega}(a)<\frac{k-1}{k}d_{\tau_\omega}(b)
\quad\forall \tau\in T(A). 
\]
By Lemma \ref{II_1} (1), 
there exists a unital homomorphism $\phi:M_k\to\mathcal{M}^{\bar\alpha}$. 
By \cite[Theorem 10.2.1]{Lo}, $C_0((0,1])\otimes M_k$ is projective. 
It follows from Theorem \ref{exact} that we can find a homomorphism 
$\tilde\phi:C_0((0,1])\otimes M_k\to(A_\omega)^\alpha$ such that 
$\rho(\tilde\phi(\iota\otimes x))=\phi(x)$ holds for every $x\in M_k$, 
where $\iota\in C_0((0,1])$ is the identity map. 
By the usual fast reindexation trick, 
we may and do assume that the range of $\tilde\phi$ commutes with $b$. 
Let $p\in M_k$ be a rank one projection. 
Since $\rho(b)$ commutes with elements of $\phi(M_k)$, we have 
\[
\tau(1_{(0,\infty)}(\rho(b)\phi(p)))
=\tau(1_{(0,\infty)}(\rho(b))\phi(p))
=\frac{1}{k}\tau(1_{(0,\infty)}(\rho(b)))
\]
and 
\[
\tau(1_{(0,\infty)}(\rho(b)\phi(1{-}p)))
=\tau(1_{(0,\infty)}(\rho(b))\phi(1{-}p))
=\frac{k{-}1}{k}\tau(1_{(0,\infty)}(\rho(b)))
\]
for all $\tau\in T(\mathcal{M})$. 
As $T(\mathcal{M})$ has finitely many extremal points, 
there exists $\ep>0$ such that 
\begin{equation}
\min_{\tau\in T(\mathcal{M})}\tau(1_{(\ep,\infty)}(\rho(b)\phi(p)))>0
\label{mouse}
\end{equation}
and 
\begin{equation}
\tau(1_{(0,\infty)}(\rho(a)))<\tau(1_{(\ep,\infty)}(\rho(b)\phi(1{-}p)))
\quad\forall \tau\in T(\mathcal{M}). 
\label{cow}
\end{equation}
Define continuous functions $g$ and $h$ on $[0,\infty)$ 
by $g(t)=\min\{\ep^{-1},t^{-1}\}$ and $h(t)=tg(t)$. 
By Lemma \ref{II_1}, 
$\mathcal{M}^{\bar\alpha}$ is a finite direct sum of II$_1$-factors and 
every trace on $\mathcal{M}^{\bar\alpha}$ is 
the restriction of a trace on $\mathcal{M}$. 
It follows from \eqref{cow} that 
we can find a unitary $v\in\mathcal{M}^{\bar\alpha}$ such that 
\[
1_{(0,\infty)}(\rho(a))\leq v1_{(\ep,\infty)}(\rho(b)\phi(1{-}p))v^*. 
\]
By Theorem \ref{exact}, 
there exists a unitary $w\in(A_\omega)^\alpha$ such that $\rho(w)=v$. 
Set 
\[
r=a^{1/2}wg(b)^{1/2}\tilde\phi(\iota\otimes(1{-}p))^{1/2}. 
\]
Then 
\begin{align*}
rbr^*
&=a^{1/2}wg(b)^{1/2}\tilde\phi(\iota\otimes(1{-}p))^{1/2}b
\tilde\phi(\iota\otimes(1{-}p))^{1/2}g(b)^{1/2}w^*a^{1/2}\\
&=a^{1/2}wh(b)\tilde\phi(\iota\otimes(1{-}p))w^*a^{1/2}\leq a 
\end{align*}
and 
\begin{align*}
\rho(rbr^*)
&=\rho(a^{1/2}wg(b)^{1/2}\tilde\phi(\iota\otimes(1{-}p))^{1/2}b
\tilde\phi(\iota\otimes(1{-}p))^{1/2}g(b)^{1/2}w^*a^{1/2})\\
&=\rho(a^{1/2})(v1_{(\ep,\infty)}(\rho(b)\phi(1{-}p))v^*)
v\rho(h(b))\phi(1{-}p)v^*\rho(a^{1/2})\\
&=\rho(a^{1/2})v1_{(\ep,\infty)}(\rho(b)\phi(1{-}p))
h(\rho(b)\phi(1{-}p))v^*\rho(a^{1/2})\\
&=\rho(a^{1/2})v1_{(\ep,\infty)}(\rho(b)\phi(1{-}p))v^*\rho(a^{1/2})\\
&=\rho(a). 
\end{align*}
Thus, $a-rbr^*$ is a positive element of $J^\alpha$. 
On the other hand, one has 
\begin{align*}
\tau_\omega((h(b)\tilde\phi(\iota\otimes p))^m)
&=\tau((\rho(h(b))\phi(p))^m)\\
&=\tau(h(\rho(b)\phi(p))^m)\\
&\geq\tau(1_{(\ep,\infty)}(\rho(b)\phi(p))). 
\end{align*}
This, together with \eqref{mouse}, implies 
\[
\inf_{m\in\N}\min_{\tau\in T(\mathcal{M})}
\tau_\omega((h(b)\tilde\phi(\iota\otimes p))^m)>0. 
\]
Therefore, by appealing to Proposition \ref{equivSI}, 
we can find $s\in(A_\omega)^\alpha$ such that 
\[
s^*s=a-rbr^*\quad\text{and}\quad h(b)\tilde\phi(\iota\otimes p)s=s. 
\]
Put 
\[
t=s^*g(b)^{1/2}\tilde\phi(\iota\otimes p)^{1/2}. 
\]
Clearly $rbt^*=0$, and so we get 
\begin{align*}
(r+t)b(r+t)^*&=rbr^*+tbt^*\\
&=rbr^*+s^*g(b)^{1/2}\tilde\phi(\iota\otimes p)^{1/2}b
\tilde\phi(\iota\otimes p)^{1/2}g(b)^{1/2}s\\
&=rbr^*+s^*h(b)\tilde\phi(\iota\otimes p)s=a. 
\end{align*}
\end{proof}

%%%%%%%%%%%%%%%%%%%%%%%%%%%%%%%%%%%%%%%%%%%%%%%%%%%%%%%%%%%%
\subsection{$\mathcal{Z}$-stability of crossed products}

In this subsection, we state direct consequences 
of the results obtained in Section 3 and Section 4.1. 
Throughout this subsection, we let $A$ be 
a unital, simple, separable, nuclear, stably finite, infinite dimensional 
$C^*$-algebra with finitely many extremal tracial states. 
See Definition \ref{Defofca} and Definition \ref{Defoftc} 
for the definitions of (strong) cocycle conjugacy 
and twisted crossed products. 

We note that 
the following theorem contains Theorem \ref{Z=SI} (2) 
as a special case (namely $\Gamma=\{1\}$). 

\begin{theorem}\label{wR>Z}
Suppose that $A$ has the property (SI). 
Let $(\alpha,u):\Gamma\curvearrowright A$ be a cocycle action 
of a countable discrete amenable group $\Gamma$ with the weak Rohlin property. 
Then there exists a unital homomorphism 
from $\mathcal{Z}$ to $(A_\omega)^\alpha$, 
and hence $(\alpha,u):\Gamma\curvearrowright A$ is 
strongly cocycle conjugate to 
$(\alpha\otimes\id,u\otimes1):\Gamma\curvearrowright A\otimes\mathcal{Z}$. 
In particular, 
the twisted crossed product $A\rtimes_{(\alpha,u)}\Gamma$ 
is $\mathcal{Z}$-stable. 
\end{theorem}
\begin{proof}
We first prove the existence of a unital homomorphism 
from $\mathcal{Z}$ to $(A_\omega)^\alpha$. 
As shown in \cite{JS}, $\mathcal{Z}$ is 
an inductive limit of the prime dimension drop algebras $I(k,k{+}1)$'s. 
Hence, in the light of \cite[Proposition 2.2]{TW08CJM}, 
it suffices to construct 
a unital homomorphism from $I(k,k{+}1)$ to $(A_\omega)^\alpha$. 
Let $(e_{i,j})_{i,j}$ be a system of matrix units for $M_k$. 
As in the proof of Proposition \ref{sc}, 
we can find a homomorphism 
$\tilde\phi:C_0((0,1])\otimes M_k\to(A_\omega)^\alpha$ such that 
$\tau_\omega(\tilde\phi(\iota^m\otimes e_{1,1}))=1/k$ 
for every $\tau\in T(A)$ and $m\in\N$, 
where $\iota\in C_0((0,1])$ is the identity map. 
Clearly 
\[
\tilde\phi(\iota\otimes e_{1,1})\geq0\quad\text{and}\quad 
\tilde\phi(\iota\otimes e_{1,i})\tilde\phi(\iota\otimes e_{1,j})^*
=\begin{cases}\tilde\phi(\iota\otimes e_{1,1})^2&i=j\\0&i\neq j. \end{cases}
\]
By Lemma \ref{equivSI}, there exists $s\in(A_\omega)^\alpha$ such that 
\[
s^*s=1-\tilde\phi(\iota^2\otimes1)\quad\text{and}\quad 
\tilde\phi(\iota\otimes e_{1,1})s=s. 
\]
It follows from \cite[Proposition 2.1]{S10JFA} that 
there exists a unital homomorphism from $I(k,k{+}1)$ to $(A_\omega)^\alpha$. 

Once we get a unital embedding $\mathcal{Z}\to(A_\omega)^\alpha$, 
strong cocycle conjugacy 
between $(\alpha,u)$ and $(\alpha\otimes\id,u\otimes1)$ follows 
by essentially the same method as \cite[Theorem 7.2.2]{Rtext}. 
Then we can conclude that 
the twisted crossed product $A\rtimes_{(\alpha,u)}\Gamma$ 
is $\mathcal{Z}$-stable. 
\end{proof}

\begin{corollary}
Let $\Gamma$ be a countable discrete amenable group with the property (Q). 
Suppose that $A$ has the property (SI). 
Let $(\alpha,u):\Gamma\curvearrowright A$ be a strongly outer cocycle action. 
Then there exists a unital homomorphism 
from $\mathcal{Z}$ to $(A_\omega)^\alpha$, 
and hence $(\alpha,u)$ is strongly cocycle conjugate to 
$(\alpha\otimes\id,u\otimes1):\Gamma\curvearrowright A\otimes\mathcal{Z}$. 
In particular, 
the twisted crossed product $A\rtimes_{(\alpha,u)}\Gamma$ 
is $\mathcal{Z}$-stable. 
\end{corollary}
\begin{proof}
By Theorem \ref{so=wR}, 
the cocycle action $(\alpha,u)$ has the weak Rohlin property. 
Then the conclusion follows from the theorem above. 
\end{proof}

\begin{corollary}\label{so>Z}
Let $\Gamma$ be an elementary amenable group. 
Suppose that $A$ has the property (SI). 
Let $(\alpha,u):\Gamma\curvearrowright A$ be a strongly outer cocycle action. 
Then there exists a unital homomorphism 
from $\mathcal{Z}$ to $(A_\omega)^\alpha$, 
and hence $(\alpha,u)$ is strongly cocycle conjugate to 
$(\alpha\otimes\id,u\otimes1):\Gamma\curvearrowright A\otimes\mathcal{Z}$. 
In particular, 
the twisted crossed product $A\rtimes_{(\alpha,u)}\Gamma$ 
is $\mathcal{Z}$-stable. 
\end{corollary}
\begin{proof}
By Proposition \ref{EG>Q}, $\Gamma$ has the property (Q). 
Then the conclusion follows from the corollary above. 
\end{proof}

\begin{remark}
In the corollaries above, the assumption of the property (SI) 
can be replaced with $\mathcal{Z}$-stability of $A$, 
thanks to Theorem \ref{Z=SI} (1). 
\end{remark}

%%%%%%%%%%%%%%%%%%%%%%%%%%%%%%%%%%%%%%%%%%%%%%%%%%%%%%%%%%%%
%\subsection{An application to the Rohlin property}

%%%%%%%%%%%%%%%%%%%%%%%%%%%%%%%%%%%%%%%%%%%%%%%%%%%%%%%%%%%%
\section{A remark on central sequence algebras}

In this section, we show that 
for any unital separable nuclear $C^*$-algebra $A$ which is not of type I, 
the central sequence algebra $A_\omega$ has a subquotient 
which is isomorphic to a II$_1$-factor (Theorem \ref{Glimm}). 
This may be regarded as a variant of Glimm's theorem, 
which says that if a separable $C^*$-algebra is not of type I, 
then it has a subquotient isomorphic to a UHF algebra. 

Let $M$ be a von Neumann algebra with separable predual $M_*$. 
We recall the definition of the central sequence algebra of $M$ 
in the sense of von Neumann algebra from \cite[Section II]{C74JFA} 
(see also \cite[Chapter XIV \S4]{TIII}). 

For a normal positive functional $\phi\in M_*$, we let 
\[
\lVert x\rVert_\phi^\#
=\left(\frac{\phi(x^*x)+\phi(xx^*)}{2}\right)^{1/2}. 
\]
Since $x\mapsto\phi(x^*x)^{1/2}$ is a seminorm, 
$x\mapsto\lVert x\rVert_\phi^\#$ is also a seminorm on $M$. 
The locally convex topology on $M$ 
defined by the family of these seminorms $\lVert\cdot\rVert_\phi^\#$ 
is called the $*$-strong topology 
(see \cite[Definition II.2.3]{TI} for example). 

A bounded sequence $(x_n)_n$ in $M$ is called strongly $\omega$-central if 
\[
\lim_{n\to\omega}
\sup\{\lvert\phi(x_ny-yx_n)\rvert\mid y\in M,\ \lVert y\rVert\leq1\}=0
\]
for every $\phi\in M_*$. 
Let $\mathcal{C}$ denote the set of all strongly $\omega$-central sequences. 
Clearly $\mathcal{C}$ is a unital $C^*$-subalgebra of $\ell^\infty(\N,M)$. 
Let $\mathcal{I}$ denote the set of all bounded sequences in $M$ 
$\omega$-converging to zero in the $*$-strong topology. 
It is easy to see that 
$\mathcal{I}$ is a norm-closed two-sided ideal of $\mathcal{C}$. 
We call the quotient $\mathcal{M}=\mathcal{C}/\mathcal{I}$ 
the central sequence algebra of $M$. 
It is well-known that 
if $M$ is a hyperfinite type II factor 
then $\mathcal{M}$ is a II$_1$-factor 
(see \cite[Theorem XIV.4.18]{TIII} and \cite[Lemma 2.11.(b)]{C74JFA}). 

\begin{theorem}\label{Glimm}
Let $A$ be a unital separable nuclear $C^*$-algebra 
which is not of type I. 
Then there exists a subquotient of $A_\omega=A^\omega\cap A'$ 
which is isomorphic to a II$_1$-factor. 
\end{theorem}
\begin{proof}
Since $A$ is not of type I, by Glimm's theorem, 
there exist a type II factor $M$ with separable predual and 
a unital homomorphism $\pi:A\to M$ 
such that $\pi(A)$ is $*$-strongly dense in $M$. 
Moreover, $M$ is hyperfinite, because $A$ is assumed to be nuclear. 

Let $\mathcal{C}$, $\mathcal{I}$ and $\mathcal{M}$ be as above. 
We let $B\subset\ell^\infty(\N,A)$ be 
the set of all $\omega$-central sequences $(a_n)_n$ such that 
the sequence $(\pi(a_n))_n$ in $M$ is strongly $\omega$-central. 
It is routine to check that 
$B$ is a unital $C^*$-subalgebra of $\ell^\infty(\N,A)$ and 
that $B$ contains $c_\omega(A)$. 
Let $\tilde\pi:\ell^\infty(\N,A)\to\ell^\infty(\N,M)$ be 
the homomorphism which is naturally induced by $\pi$. 
Then $\tilde\pi(B)$ is contained in $\mathcal{C}$ and 
$\tilde\pi(c_\omega(A))$ is contained in $\mathcal{I}$. 
Hence $\tilde\pi|B$ induces a unital homomorphism 
from $B/c_\omega(A)\subset A_\omega$ to $\mathcal{M}$. 
As mentioned above, $\mathcal{M}$ is known to be a II$_1$-factor. 
Therefore, it remains for us to show that 
the homomorphism $B/c_\omega(A)\to\mathcal{M}$ is surjective. 
The proof is essentially the same as \cite[Lemma 2.1]{S11}, 
but we include it for completeness. 

Take a strongly $\omega$-central sequence $(x_n)_n\in\mathcal{C}$ arbitrarily. 
We assume $\lVert x_n\rVert\leq1$. 
Let $\psi\in M_*$ be a normal positive faithful functional. 
Then $\lVert\cdot\rVert_\psi^\#$ is a norm and 
the $*$-strong topology agrees with 
the topology induced by the norm $\lVert\cdot\rVert_\psi^\#$ 
on any bounded subset of $M$ (see \cite[Proposition III.5.3]{TI} for example). 
Hence it suffices to show that 
there exists an $\omega$-central sequence $(a_n)_n$ in $A$ such that 
$\lVert x_n-\pi(a_n)\rVert_\psi^\#\to0$ as $n\to\omega$. 
By Kaplansky's density theorem (see \cite[Theorem II.4.8]{TI} for example), 
we can find $(b_n)_n\in\ell^\infty(\N,A)$ such that 
$\lVert b_n\rVert\leq1$ and 
$\lVert x_n-\pi(b_n)\rVert_\psi^\#\to0$ as $n\to\infty$. 
We choose an increasing sequence $(F_m)_m$ of finite subsets of $A$ 
whose union is dense in $A$. 
Since $A$ is nuclear, 
we can apply Theorem \ref{Haagerup1} to $F_m$ and $1/m$, 
and get a finite subset $G_m\subset A$. 
Define 
\[
\Omega_m=\{n\in\N\mid n\geq m,\quad 
\lVert[\pi(b_n),\pi(w)]\pi(w^*)\rVert_\psi^\#<m^{-1}\lvert G_m\rvert^{-1}
\quad\forall w\in G_m\}. 
\]
As $(x_n)_n$ is strongly $\omega$-central, 
$(\pi(b_n))_n$ is also strongly $\omega$-central. 
Therefore, by \cite[Lemma XIV.3.4 (i)]{TIII}, 
$[\pi(b_n),\pi(w)]$ tends to zero in the $*$-strong topology 
as $n\to\omega$ for every $w\in A$. 
It follows that $\Omega_m$ belongs to $\omega$. 
For each $n\in\Omega_m\setminus\Omega_{m+1}$, we define $a_n\in A$ by 
\[
a_n=\sum_{w\in G_m}wb_nw^*. 
\]
By the choice of $G_m$, we have $\lVert[a_n,c]\rVert<1/m$ for all $c\in F_m$, 
that is, $(a_n)_n$ is an $\omega$-central sequence. 
Besides, 
\begin{align*}
\lVert x_n-\pi(a_n)\rVert_\psi^\#
&=\left\lVert x_n-\sum_{w\in G_m}\pi(wb_nw^*)\right\rVert_\psi^\#\\
&\leq\left\lVert x_n-\sum_{w\in G_m}\pi(b_n)\pi(ww^*)\right\rVert_\psi^\#
+1/m\\
&=\lVert x_n-\pi(b_n)\rVert_\psi^\#+1/m, 
\end{align*}
and so $\lVert x_n-\pi(a_n)\rVert_\psi^\#\to0$ as $n\to\omega$. 
\end{proof}

\begin{remark}
In the theorem above, suppose that $M$ is the hyperfinite II$_1$-factor. 
Then, by \cite[Lemma XIV.3.4 (ii)]{TIII}, 
we can see that for any $\omega$-central sequence $(a_n)_n$ in $A$, 
the sequence $(\pi(a_n))_n$ in $M$ is strongly $\omega$-central. 
Thus, $B$ above equals $A_\omega=A^\omega\cap A'$, 
and whence there exists a unital homomorphism 
from $A_\omega$ to the II$_1$-factor $\mathcal{M}$. 
\end{remark}

%%%%%%%%%%%%%%%%%%%%%%%%%%%%%%%%%%%%%%%%%%%%%%%%%%%%%%%%%%%%
\section{Actions of the Klein bottle group I}

In this section, we would like to prove that 
all strongly outer cocycle actions of the Klein bottle group 
on a UHF algebra are mutually cocycle conjugate (Theorem \ref{KleinUHF}). 
We note that some of the ideas and techniques used in this section 
are borrowed from \cite{IM2}. 

%%%%%%%%%%%%%%%%%%%%%%%%%%%%%%%%%%%%%%%%%%%%%%%%%%%%%%%%%%%%
\subsection{The OrderExt invariant}

In this subsection, 
we would like to collect some basic facts on the $\OrderExt$ group 
for later use. 
The $\OrderExt$ invariant was introduced 
by A. Kishimoto and A. Kumjian \cite{KK} 
in order to determine 
when an approximately inner automorphism of 
a unital simple AT algebra with real rank zero is asymptotically inner. 
In the later subsections, we use this invariant 
for classification of certain group actions on AF algebras. 

We begin with the definition of the $\OrderExt$ group (\cite[Section 2]{KK}). 
Let $G_0,G_1,F$ be abelian groups and let $D:G_0\to F$ be a homomorphism. 
When 
\[
\begin{CD}
\xi:\quad 0@>>>G_0@>\iota>>E_\xi@>q>>G_1@>>>0
\end{CD}
\]
is exact, $R$ is in $\Hom(E_\xi,F)$ and $R\circ\iota=D$, 
the pair $(\xi,R)$ is called an order-extension. 
Two order-extensions $(\xi,R)$ and $(\xi',R')$ are equivalent 
if there exists an isomorphism $\theta:E_\xi\to E_{\xi'}$ such that 
$R=R'\circ\theta$ and 
\[
\begin{CD}
\xi:\quad 0@>>>G_0@>>>E_\xi@>>>G_1@>>>0\\
@.@|@VV\theta V@|@.\\
\xi':\quad 0@>>>G_0@>>>E_{\xi'}@>>>G_1@>>>0
\end{CD}
\]
is commutative. 
Then $\OrderExt(G_1,G_0,D)$ consists of 
equivalence classes of all order-extensions. 
As shown in \cite{KK}, 
$\OrderExt(G_1,G_0,D)$ is equipped with an abelian group structure. 
The map sending $(\xi,R)$ to $\xi$ induces 
a homomorphism from $\OrderExt(G_1,G_0,D)$ onto $\Ext(G_1,G_0)$. 

When $(\xi,R)$ is an order-extension, 
there exists $r\in\Hom(G_1,F/\Ima D)$ 
such that the following diagram is commutative: 
\[
\begin{CD}
0@>>>G_0@>\iota>>E_\xi@>q>>G_1@>>>0\\
@.@VVDV@VVRV@VVr V@.\\
0@>>>\Ima D@>>>F@>>>F/\Ima D@>>>0, 
\end{CD}
\]
Clearly $r$ depends only on the $\OrderExt$ class of $(\xi,R)$ and 
the map $(\xi,R)\mapsto r$ gives rise to a homomorphism. 
We denote it by $\rho:\OrderExt(G_1,G_0,D)\to\Hom(G_1,F/\Ima D)$. 

\begin{lemma}\label{rho}
Let $G_0,G_1,F$ be abelian groups and let $D:G_0\to F$ be a homomorphism. 
If $D$ is injective, then 
the homomorphism $\rho:\OrderExt(G_1,G_0,D)\to\Hom(G_1,F/\Ima D)$ 
is an isomorphism. 
\end{lemma}
\begin{proof}
Suppose that $r\in\Hom(G_1,F/\Ima D)$ is given. 
Define $\xi\in\Ext(G_1,G_0)$ by 
\[
E_\xi=\{(g,f)\in G_1\times F\mid r(g)=f+\Ima D\},\quad 
\iota(h)=(0,D(h)),\quad q(g,f)=g. 
\]
Note that $\iota$ is injective, because $D$ is injective. 
We define $R:E_\xi\to F$ by $R(g,f)=f$. 
Then one has $\rho([(\xi,R)])=r$. 
To prove the injectivity of $\rho$, 
we assume that an order-extension $(\xi,R)$ satisfies $\rho([(\xi,R)])=0$. 
Then $R(E_\xi)$ is contained in $\Ima D$. 
Since $\Ker D=0$, 
by Lemma 2.4 and Proposition 2.5 of \cite{KK}, we get $[(\xi,R)]=0$. 
Therefore $\rho$ is an isomorphism. 
\end{proof}

We would like to recall the main result of \cite{KK}. 
Let $B$ be a unital simple AT algebra with real rank zero. 
As described in \cite{KK}, 
there exist natural homomorphisms 
\[
\tilde\eta_0:\aInn(B)\to\OrderExt(K_1(B),K_0(B),D_B)
\]
and 
\[
\eta_1:\aInn(B)\to\Ext(K_0(B),K_1(B)). 
\]
The following is the main result of \cite{KK}. 

\begin{theorem}[{\cite[Theorem 4.4]{KK}}]\label{KishiKum}
Suppose that $B$ is a unital simple AT algebra with real rank zero. 
Then the homomorphism 
\[
\tilde\eta_0\oplus\eta_1:\aInn(B)\to
\OrderExt(K_1(B),K_0(B),D_B)\oplus\Ext(K_0(B),K_1(B))
\]
is surjective and 
its kernel equals the set of all asymptotically inner automorphisms of $B$. 
\end{theorem}

In the rest of this subsection, 
we let $A$ be a unital simple AF algebra and 
let $\phi\in\aInn(A)$ be an approximately inner automorphism 
such that $\phi^n$ is not weakly inner for every $n\in\N$. 
By the Pimsner-Voiculescu exact sequence, 
$(K_0(A\rtimes_\phi\Z),K_1(A\rtimes_\phi\Z))$ are 
naturally identified with $(K_0(A),K_0(A))$ (see \cite[Section 6]{M10CMP}). 
The space of tracial states $T(A\rtimes_\phi\Z)$ is also 
identified with $T(A)$. 
These identifications induce the natural isomorphism 
\[
\OrderExt(K_1(A\rtimes_\phi\Z),K_0(A\rtimes_\phi\Z),D_{A\rtimes_\phi\Z})
\cong\OrderExt(K_0(A),K_0(A),D_A). 
\]
From now on we will freely identify these two $\OrderExt$ groups. 

\begin{lemma}\label{eta1}
For any approximately inner automorphism $\psi\in\Aut_\T(A\rtimes_\phi\Z)$, 
one has $\eta_1(\psi)=0$. 
\end{lemma}
\begin{proof}
This follows from the proof of \cite[Lemma 6.3]{M10CMP}. 
\end{proof}

From Theorem \ref{KishiKum} and the lemma above, 
we obtain the following. 

\begin{proposition}\label{eta>asymp}
Suppose that 
$A\rtimes_\phi\Z$ is a unital simple AT algebra with real rank zero. 
Let $\psi\in\Aut_\T(A\rtimes_\phi\Z)$ be an approximately inner automorphism. 
If $\tilde\eta_0(\psi)=0$, then $\psi$ is asymptotically inner. 
\end{proposition}

For $t\in\R$ we define $\hat\phi_t\in\Aut_\T(A\rtimes_\phi\Z)$ by 
\[
\hat\phi_t(x)=x\quad\forall x\in A\quad\text{and}\quad 
\hat\phi_t(\lambda^\phi)=\exp(2\pi\sqrt{-1}t)\lambda^\phi. 
\]

\begin{lemma}\label{dualaction}
Let $\rho:\OrderExt(K_0(A),K_0(A),D_A)\to\Hom(K_0(A),\Aff(T(A))/\Ima D_A)$ 
be the homomorphism described above. 
For every $t\in\R$ and $x\in K_0(A)$, one has 
\[
\rho(\tilde\eta_0(\hat\phi_t))(x)=tD_A(x)+\Ima D_A. 
\]
\end{lemma}
\begin{proof}
Let $t\in\R$ and let $p\in A$ be a projection. 
It suffices to show the equation for $x=[p]$. 
Under the identification between $K_0(A)$ and $K_1(A\rtimes_\phi\Z)$, 
$[p]$ corresponds to $[\lambda^\phi p+v(1-p)]$, 
where $v$ is a unitary in $A$ satisfying $vpv^*=\phi(p)$. 
Define $f:[0,1]\to U(A\rtimes_\phi\Z)$ by 
\[
f(s)=\exp(2\pi\sqrt{-1}st)\lambda^\phi p+v(1-p). 
\]
Then $f$ is a unitary in the mapping torus of $\hat\phi_t$ and 
\[
\frac{1}{2\pi\sqrt{-1}}\int_0^1\tau(\dot{f}(s)f(s)^*)\,ds
=\frac{1}{2\pi\sqrt{-1}}\int_0^1\tau(2\pi\sqrt{-1}tp)\,ds
=t\tau(p)
\]
for $\tau\in T(A)$. 
Hence $\rho(\tilde\eta_0(\hat\phi_t))([p])=tD_A([p])+\Ima D_A$. 
\end{proof}

We conclude this subsection by the following lemma, 
which will be used in Section 6.3. 

\begin{lemma}\label{reverse}
Suppose that 
$\sigma\in\Aut(A\rtimes_\phi\Z)$ satisfies 
$\sigma|A\in\aInn(A)$ and $\sigma(\lambda^\phi)\lambda^\phi\in A$. 
Let $\psi\in\Aut_\T(A\rtimes_\phi\Z)$ be an approximately inner automorphism. 
Then $\sigma\circ\psi\circ\sigma^{-1}$ is in $\Aut_\T(A\rtimes_\phi\Z)$ and 
$\tilde\eta_0(\sigma\circ\psi\circ\sigma^{-1})=-\tilde\eta_0(\psi)$. 
\end{lemma}
\begin{proof}
Let $B=A\rtimes_\phi\Z$. 
It is straightforward to check 
$\sigma\circ\psi\circ\sigma^{-1}\in\Aut_\T(B)$. 
Let 
\[
M(B,\psi)=\{f\in C([0,1])\otimes B\mid\psi(f(0))=f(1)\}
\]
and 
\[
M(B,\sigma\circ\psi\circ\sigma^{-1})
=\{f\in C([0,1])\otimes B\mid(\sigma\circ\psi\circ\sigma^{-1})(f(0))=f(1)\}
\]
be mapping tori. 
Then we have the following commutative diagram: 
\[
\begin{CD}
0@>>>C_0(0,1)\otimes B@>>>M(B,\psi)@>>>B@>>>0\\
@.@VV\id\otimes\sigma V@VVV@VV\sigma V@.\\
0@>>>C_0(0,1)\otimes B@>>>M(B,\sigma\circ\psi\circ\sigma^{-1})@>>>B@>>>0, 
\end{CD}
\]
where the horizontal sequences are exact. 
As mentioned above, 
$K_*(B)=K_*(A\rtimes_\phi\Z)$ is naturally isomorphic to $K_0(A)$. 
Since $\sigma|A$ is approximately inner, we get $K_0(\sigma)=\id$ on $K_0(B)$. 
Also, from $\sigma(\lambda^\phi)\lambda^\phi\in A$, 
we obtain $K_1(\sigma)=-\id$ on $K_1(B)$. 
Thus, we get the commutative diagram: 
\[
\begin{CD}
\xi:\quad 0@>>>K_0(B)@>>>K_1(M(B,\psi))@>>>K_1(B)@>>>0\\
@.@VV\id V@VV\theta V@VV-\id V@.\\
\xi':\quad 0@>>>K_0(B)@>>>K_1(M(B,\sigma\circ\psi\circ\sigma^{-1}))
@>>>K_1(B)@>>>0, 
\end{CD}
\]
where the horizontal sequences $\xi$ and $\xi'$ are exact. 
Hence $\xi'=-\xi$ in $\Ext(K_1(B),K_0(B))$. 
Moreover, 
letting $R$ and $R'$ be the corresponding rotation maps 
for $\psi$ and $\sigma\circ\psi\circ\sigma^{-1}$ respectively, 
we can easily check $R=R'\circ\theta$. 
Then the conclusion follows from the definition of $\tilde\eta_0(\cdot)$. 
\end{proof}

%%%%%%%%%%%%%%%%%%%%%%%%%%%%%%%%%%%%%%%%%%%%%%%%%%%%%%%%%%%%
\subsection{Cocycle actions of $\Z^2$ on UHF algebras}

In this subsection, we would like to 
summarize the results obtained in \cite[Section 6]{M10CMP} 
and generalize them to cocycle actions of $\Z^2$. 
Throughout this subsection, 
we write $\Z^2=\langle a,b\mid bab^{-1}=a\rangle$. 
For a cocycle action $(\alpha,u):\Z^2\curvearrowright A$, 
we put 
\[
\check u=u(b,a)u(a,b)^*, 
\]
so that 
\[
\alpha_b\circ\alpha_a=\Ad\check u\circ\alpha_a\circ\alpha_b
\]
holds. 
When $A$ is a UHF algebra or the Jiang-Su algebra, 
$(\alpha,u)$ is cocycle conjugate to a genuine action 
if and only if $\Delta_\tau(\check u)=0$ (\cite[Theorem 8.6, 8.8]{MS}). 

First, 
let us recall the invariant of cocycle actions of $\Z^2$ 
taking its values in an $\OrderExt$ group (\cite[Definition 6.4]{M10CMP}). 
Let $(\alpha,u):\Z^2\curvearrowright A$ be 
a strongly outer and locally $KK$-trivial (i.e. $K_0(\alpha_g)=\id$) 
cocycle action of $\Z^2$ on a unital simple AF algebra $A$. 
The automorphism $\alpha_b\in\Aut(A)$ extends to 
$\tilde\alpha_b\in\Aut_\T(A\rtimes_{\alpha_a}\Z)$ 
by letting $\tilde\alpha_b(x)=\alpha_b(x)$ for all $x\in A$ and 
$\tilde\alpha_b(\lambda^{\alpha_a})=\check u\lambda^{\alpha_a}$. 
As mentioned in the previous subsection, 
by the Pimsner-Voiculescu exact sequence, 
$(K_0(A\rtimes_{\alpha_a}\Z),K_1(A\rtimes_{\alpha_a}\Z))$ are 
naturally identified with $(K_0(A),K_0(A))$. 
The space of tracial states $T(A\rtimes_{\alpha_a}\Z)$ is also 
identified with $T(A)$. 
Under these identifications, 
we regard $\tilde\eta_0(\tilde\alpha_b)$ 
as an element in $\OrderExt(K_0(A),K_0(A),D_A)$, 
and define our invariant $c(\alpha,u)\in\OrderExt(K_0(A),K_0(A),D_A)$ by 
\[
c(\alpha,u)=\tilde\eta_0(\tilde\alpha_b)
\in\OrderExt(K_0(A),K_0(A),D_A). 
\]

\begin{remark}\label{detofu}
Let $\rho:\OrderExt(K_0(A),K_0(A),D_A)\to\Hom(K_0(A),\Aff(T(A))/\Ima D)$ 
be the homomorphism described in the previous subsection. 
Then one has $\rho(c(\alpha,u))([1])=\Delta_\tau(\check u)$. 
Indeed, 
under the identification between $K_0(A)$ and $K_1(A\rtimes_{\alpha_a}\Z)$, 
$[1]$ corresponds to $[\lambda^{\alpha_a}]$. 
Let $z:[0,1]\to U(A)$ be a piecewise smooth path 
such that $z(0)=1$ and $z(1)=\check u$. 
Then $f(t)=z(t)\lambda^{\alpha_a}$ is 
a unitary in the mapping torus of $\tilde\alpha_b$. 
Therefore $\rho(c(\alpha,u))([1])=\Delta_\tau(\check u)$. 
\end{remark}

\begin{lemma}
Let $(\alpha,u),(\beta,v):\Z^2\curvearrowright A$ be 
strongly outer and locally $KK$-trivial cocycle actions of $\Z^2$ 
on a unital simple AF algebra $A$. 
If $(\alpha,u)$ and $(\beta,v)$ are cocycle conjugate 
via an approximately inner automorphism $\theta\in\aInn(A)$, 
then $c(\alpha,u)=c(\beta,v)$. 
\end{lemma}
\begin{proof}
There exists a family of unitaries $(w_g)_{g\in\Z^2}$ in $A$ such that 
\[
\theta\circ\alpha_g\circ\theta^{-1}=\Ad w_g\circ\beta_g\quad\text{and}\quad 
\theta(u(g,h))=w_g\beta_g(w_h)v(g,h)w_{gh}^*
\]
hold for every $g,h\in\Z^2$. 
Then $\theta$ extends to an isomorphism 
from $A\rtimes_{\alpha_a}\Z$ to $A\rtimes_{\beta_a}\Z$ by letting 
\[
\theta(\lambda^{\alpha_a})=w_a\lambda^{\beta_a}. 
\]
As in \cite[Section 6]{M10CMP}, 
we let $M(A\rtimes_{\alpha_a}\Z,\tilde\alpha_b)$ and 
$M(A\rtimes_{\beta_a}\Z,\tilde\beta_b)$ denote the mapping tori. 
Take a continuous path of unitaries $z:[0,1]\to U(A)$ 
such that $z(0)=1$ and $z(1)=w_b$. 
Define an isomorphism 
$\tilde\theta:M(A\rtimes_{\alpha_a}\Z,\tilde\alpha_b)
\to M(A\rtimes_{\beta_a}\Z,\tilde\beta_b)$ by 
\[
\tilde\theta(f)(t)=z(t)^*\theta(f(t))z(t). 
\]
Since 
\[
w_b^*\theta(\alpha_b(x))w_b
=(\Ad w_b^*\circ\theta\circ\alpha_b)(x)
=\beta_b(\theta(x))
\]
for any $x\in A$ and 
\begin{align*}
w_b^*\theta(\tilde\alpha_b(\lambda^{\alpha_a}))w_b
&=w_b^*\theta(\check u\lambda^{\alpha_a})w_b\\
&=w_b^*\theta(\check u)w_a\lambda^{\beta_a}w_b\\
&=w_b^*w_b\beta_b(w_a)\check v
\beta_a(w_b^*)w_a^*w_a\lambda^{\beta_a}w_b\\
&=\beta_b(w_a)\check v\lambda^{\beta_a}\\
&=\tilde\beta_b(w_a\lambda^{\beta_a})\\
&=\tilde\beta_b(\theta(\lambda^{\alpha_a})), 
\end{align*}
$\tilde\theta$ is well-defined. 
Hence we obtain the following commutative diagram: 
\[
\begin{CD}
0@>>>C_0(0,1)\otimes(A\rtimes_{\alpha_a}\Z)@>>>
M(A\rtimes_{\alpha_a}\Z,\tilde\alpha_b)@>>>A\rtimes_{\alpha_a}\Z@>>>0\\
@.@VV\tilde\theta V@VV\tilde\theta V@VV\theta V\\
0@>>>C_0(0,1)\otimes(A\rtimes_{\beta_a}\Z)@>>>
M(A\rtimes_{\beta_a}\Z,\tilde\beta_b)@>>>A\rtimes_{\beta_a}\Z@>>>0. 
\end{CD}
\]
From this it is not so hard to see 
$\tilde\eta_0(\tilde\alpha_b)=\tilde\eta_0(\tilde\beta_b)$, 
because $K_0(\theta)=\id$. 
Therefore we get $c(\alpha,u)=c(\beta,v)$. 
\end{proof}

In the same way as \cite[Theorem 6.6]{M10CMP}, 
we can prove the following. 

\begin{theorem}
Let $(\alpha,u),(\beta,v):\Z^2\curvearrowright A$ be 
strongly outer, locally $KK$-trivial cocycle actions of $\Z^2$ 
on a unital simple AF algebra $A$ with finitely many extremal tracial states. 
The following are equivalent. 
\begin{enumerate}
\item $c(\alpha,u)=c(\beta,v)$ in $\OrderExt(K_0(A),K_0(A),D_A)$. 
\item $(\alpha,u)$ and $(\beta,v)$ are cocycle conjugate 
via an approximately inner automorphism $\theta\in\Aut(A)$. 
\end{enumerate}
\end{theorem}
\begin{proof}
We give only sketchy arguments 
as the proof of \cite[Theorem 6.6]{M10CMP} applies almost verbatim. 
(2)$\Rightarrow$(1) was shown in the lemma above 
without assuming that $A$ has finitely many extremal traces. 
Let us consider the other implication (1)$\Rightarrow$(2). 
It is well-known (see \cite[Section 4]{M10CMP} for example) that 
$A\rtimes_{\alpha_a}\Z$ and $A\rtimes_{\beta_a}\Z$ are 
unital simple AT algebras with real rank zero. 
By \cite[Theorem 4.8, 4.9]{M10CMP}, 
we may assume that 
there exists $w\in U(A)$ such that $\beta_a=\Ad w\circ\alpha_a$. 
Define an isomorphism $\theta:A\rtimes_{\beta_a}\Z\to A\rtimes_{\alpha_a}\Z$ 
by $\theta(x)=x$ for $x\in A$ and 
$\theta(\lambda^{\beta_a})=w\lambda^{\alpha_a}$. 
We can observe 
\[
\tilde\eta_0(\tilde\alpha_b)=c(\alpha,u)=c(\beta,v)=\tilde\eta_0(\tilde\beta_b)
=\tilde\eta_0(\theta\circ\tilde\beta_b\circ\theta^{-1}). 
\]
It follows from Proposition \ref{eta>asymp} that 
$\tilde\alpha_b$ and $\theta\circ\tilde\beta_b\circ\theta^{-1}$ are 
asymptotically unitarily equivalent. 
Then \cite[Theorem 6.1]{M10CMP} applies and yields 
an approximately inner automorphism $\mu\in\Aut_\T(A\rtimes_{\alpha_a}\Z)$ 
and $t\in U(A)$ such that $\mu|A$ is in $\aInn(A)$ and 
\[
\mu\circ\theta\circ\tilde\beta_b\circ\theta^{-1}\circ\mu^{-1}
=\Ad t\circ\tilde\alpha_b. 
\]
(Actually, in \cite[Theorem 6.1]{M10CMP}, 
$A$ is assumed to have a unique trace. 
We can weaken this hypothesis 
by using Theorem \ref{so=trR} instead of \cite[Lemma 5.1]{M10CMP}. 
See also Theorem  \ref{so>R}.) 
Let $s=\mu(\lambda^{\alpha_a})(\lambda^{\alpha_a})^*\in U(A)$. 
One can deduce 
\[
(\mu|A)\circ\beta_b\circ(\mu|A)^{-1}=\Ad t\circ\alpha_b
\]
and 
\[
(\mu|A)\circ\beta_a\circ(\mu|A)^{-1}=(\Ad\mu(w)s)\circ\alpha_a. 
\]
Moreover, we can verify 
\[
t\alpha_b(\mu(w)s)\check u\alpha_a(t^*)s^*\mu(w^*)=\mu(\check v), 
\]
which implies that $(\alpha,u)$ and $(\beta,v)$ are cocycle conjugate. 
\end{proof}

Let us consider the case that $A$ is a UHF algebra. 
As in \cite[Section 5]{KM}, for every prime number $p$, we define 
\[
\zeta_A(p)=\sup\{k\in\N\cup\{0\}\mid[1]\text{ is divisible by }p^k
\text{ in }K_0(A)\}\in\{0,1,2,\ldots,\infty\}. 
\]
In \cite{KM}, it was shown that 
there exists a bijective correspondence between 
the set of (strong) cocycle conjugacy classes of 
strongly outer $\Z^2$-actions on $A$ and the abelian group 
\[
\prod_{1\leq\zeta_A(p)<\infty}\Z/p^{\zeta_A(p)}\Z. 
\]
As mentioned in \cite[Remark 6.7]{M10CMP}, 
it is easy to see that this group is isomorphic to 
\[
\{r\in\Hom(K_0(A),\R/\Ima D_A)\mid r([1])=0\}. 
\]
It is also known that 
this group is isomorphic to the fundamental group $\pi_1(\Aut(A))$ (\cite{T}). 

The following theorem is 
a strengthened version of the main result of \cite{KM} 
(see also \cite[Theorem 8.6]{MS}). 

\begin{theorem}
Assume that $A$ is a UHF algebra. 
There exist bijective correspondences between the following three sets. 
\begin{enumerate}
\item Cocycle conjugacy classes of 
strongly outer cocycle actions of $\Z^2$ on $A$. 
\item $\OrderExt(K_0(A),K_0(A),D_A)$. 
\item $\Hom(K_0(A),\R/K_0(A))$. 
\end{enumerate}
The correspondences are given 
by $(\alpha,u)\mapsto c(\alpha,u)$ and $(\alpha,u)\mapsto\rho(c(\alpha,u))$. 
Furthermore, 
a strongly outer cocycle action $(\alpha,u):\Z^2\curvearrowright A$ is 
cocycle conjugate to a genuine action if and only if 
$\rho(c(\alpha,u))([1])=0$. 
\end{theorem}
\begin{proof}
Since $D_A:K_0(A)\to\Aff(T(A))\cong\R$ is injective, 
by Lemma \ref{rho}, the homomorphism 
\[
\rho:\OrderExt(K_0(A),K_0(A),D_A)\to\Hom(K_0(A),\R/K_0(A)) 
\]
is an isomorphism. 
By virtue of the theorem above, 
the correspondence from (1) to (2) is injective. 
Take $r\in\Hom(K_0(A),\R/\Ima D_A)$. 
Let $t\in\R$ be such that $r([1])=t+\Ima D_A$. 
Define $r'\in\Hom(K_0(A),\R/\Ima D_A)$ by 
\[
r'(x)=r(x)-tD_A(x)+\Ima D_A
\]
so that $r'([1])=0$. 
By \cite[Lemma 5.3]{KM} and the observation above, 
there exists a strongly outer action $\gamma:\Z^2\curvearrowright A$ 
such that $\rho(c(\gamma,1))=r'$. 
By Lemma \ref{dualaction}, we have 
\[
\rho(\tilde\eta_0((\widehat\gamma_a)_t\circ\tilde\gamma_b))(x)
=tD_A(x)+r'(x)=r(x). 
\]
Take a scalar valued 2-cocycle $u:\Z^2\times\Z^2\to\T\subset U(A)$ 
such that $\check u=\exp(2\pi\sqrt{-1}t)$ 
(note that the 2-cohomology group $H^2(\Z^2,\T)$ is isomorphic to $\T$). 
Then $\rho(c(\gamma,u))=r$. 
Thus the correspondence from (1) to (2) is surjective. 

Let $(\alpha,u):\Z^2\curvearrowright A$ be 
a strongly outer cocycle action. 
Evidently, if $(\alpha,u)$ is cocycle conjugate to a genuine action, 
then $\rho(c(\alpha,u))([1])=0$ 
thanks to the theorem above and Remark \ref{detofu}. 
If $\rho(c(\alpha,u))([1])=0$, then 
by \cite[Lemma 5.3]{KM} and the observation above, 
there exists a strongly outer action $\gamma:\Z^2\curvearrowright A$ 
such that $c(\gamma,1)=c(\alpha,u)$. 
By using the theorem above again, we can conclude that 
$(\alpha,u)$ is cocycle conjugate to the genuine action $\gamma$. 
\end{proof}

We will use the following lemma in the next subsection. 

\begin{lemma}\label{toadjust}
Let $A$ be a UHF algebra and 
let $\phi$ be an automorphism of $A$ 
such that $\phi^n$ is not weakly inner for every $n\in\N$. 
For any $c\in\OrderExt(K_0(A),K_0(A),D_A)$, 
there exists $\psi\in\Aut_\T(A\rtimes_\phi\Z)$ 
such that $\tilde\eta_0(\psi)=c$. 
\end{lemma}
\begin{proof}
From the theorem above, there exists 
a strongly outer cocycle action $(\alpha,u):\Z^2\curvearrowright A$ 
such that $c(\alpha,u)=c$. 
By \cite[Theorem 1.3]{K95crelle}, 
we may assume that there exists $w\in U(A)$ 
such that $\alpha_a=\Ad w\circ\phi$. 
Define an isomorphism $\theta:A\rtimes_{\alpha_a}\Z\to A\rtimes_\phi\Z$ 
by $\theta(x)=x$ for $x\in A$ and $\theta(\lambda^{\alpha_a})=w\lambda^\phi$. 
Then 
$\theta\circ\tilde\alpha_b\circ\theta^{-1}$ is in $\Aut_\T(A\rtimes_\phi\Z)$ 
and 
\[
\tilde\eta_0(\theta\circ\tilde\alpha_b\circ\theta^{-1})
=\tilde\eta_0(\tilde\alpha_b)=c(\alpha,u)=c. 
\]
\end{proof}

%%%%%%%%%%%%%%%%%%%%%%%%%%%%%%%%%%%%%%%%%%%%%%%%%%%%%%%%%%%%
\subsection{Actions of the Klein bottle group on UHF algebras}

Throughout this subsection, 
we let $\Gamma=\langle a,b\mid bab^{-1}=a^{-1}\rangle$. 
The group $\Gamma$ is isomorphic to 
the fundamental group of the Klein bottle and 
is called the Klein bottle group. 
In this subsection, we will show that 
all strongly outer cocycle actions of $\Gamma$ on a UHF algebra 
are cocycle conjugate to each other. 

First, we would like to show 
the existence of strongly outer, approximately representable $\Gamma$-actions 
on the Jiang-Su algebra $\mathcal{Z}$. 
See Definition \ref{apprepre} 
for the definition of approximate representability. 

\begin{lemma}
There exists a homomorphism $\pi:\Gamma\to U(C([0,1])\otimes M_2)$ such that 
$\pi(g)$ is not a scalar multiple of $1$ for any $g\in\Gamma\setminus\{1\}$ 
and $\pi(g)(0)=1$ for all $g\in\Gamma$. 
\end{lemma}
\begin{proof}
Put 
\[
u(t)
=\begin{bmatrix}\exp\pi\sqrt{-1}t&0\\0&\exp(-\pi\sqrt{-1}t)\end{bmatrix}. 
\]
Let $v:[0,1]\to M_2$ be a continuous path of unitaries 
such that 
\[
v(0)=\begin{bmatrix}1&0\\0&1\end{bmatrix},\quad 
v(1)=\begin{bmatrix}0&1\\1&0\end{bmatrix}. 
\]
One can define a homomorphism $\pi:\Gamma\to U(C([0,1])\otimes M_2)$ by 
\[
\pi(a)(t)=\begin{cases}1&t\leq1/2\\
u(t-1/2)&1/2\leq t\end{cases}
\quad\text{and}\quad 
\pi(b)(t)=\begin{cases}v(2t)&t\leq1/2\\
v(1)&1/2\leq t. \end{cases}
\]
Clearly 
$\pi(g)$ is not a scalar multiple of $1$ for any $g\in\Gamma\setminus\{1\}$ 
and $\pi(a)(0)=\pi(b)(0)=1$. 
\end{proof}

\begin{proposition}\label{exist}
There exists a strongly outer, approximately representable action 
$\alpha:\Gamma\curvearrowright\mathcal{Z}$. 
In particular, for any unital $\mathcal{Z}$-stable $C^*$-algebra $A$, 
there exists a strongly outer, approximately representable action 
$\alpha:\Gamma\curvearrowright A$. 
\end{proposition}
\begin{proof}
Since the $C^*$-algebra 
\[
\{f\in C([0,1])\otimes M_2\mid f(0)\in\C\}
\]
is embeddable to $\mathcal{Z}$, by the lemma above, 
there exists a homomorphism $\pi:\Gamma\to U(\mathcal{Z})$ 
such that $\Ad\pi(g)\neq\id$ for all $g\in\Gamma\setminus\{1\}$. 
Define an action $\alpha$ of $\Gamma$ 
on the infinite tensor product of $\mathcal{Z}$ by 
\[
\alpha_g=\bigotimes_{k\in\N}\Ad\pi(g). 
\]
It is clear that 
$\alpha$ is strongly outer and approximately representable. 
For any unital $C^*$-algebra $A$, 
the action $\id\otimes\alpha:\Gamma\curvearrowright A\otimes\mathcal{Z}$ is 
strongly outer and approximately representable. 
\end{proof}

The following is a Rohlin type theorem 
for $\Gamma$-actions on unital simple AF algebras. 
The proof is almost the same 
as \cite[Theorem 3]{N} and \cite[Theorem 5.5]{M10CMP}. 

\begin{theorem}\label{so>R}
Let $A$ be a unital simple AF algebra 
with finitely many extremal tracial states and 
let $(\alpha,u):\Gamma\curvearrowright A$ be a strongly outer cocycle action. 
Assume further that 
$\alpha_a^r$ and $\alpha_b^s$ are in $\aInn(A)$ for some $r,s\in\N$. 
Then for any $m\in\N$, 
there exist projections $e,f\in(A_\infty)^{\alpha_a}$ such that 
\[
\alpha_b^m(e)=e,\quad \alpha_b^{m+1}(f)=f
\]
and 
\[
\sum_{i=0}^{m-1}\alpha_b^i(e)+\sum_{j=0}^m\alpha_b^j(f)=1. 
\]
\end{theorem}
\begin{proof}
For any $\ep>0$, 
by Proposition \ref{EG>Q} and Theorem \ref{so=trR}, 
there exists an $(\{a,b\},\ep)$-invariant finite subset $K\subset\Gamma$ and 
a central sequence $(e_n)_n$ of projections in $A$ such that 
\[
\lim_{n\to\infty}\lVert\alpha_g(f_n)\alpha_h(f_n)\rVert=0
\]
for all $g,h\in K$ with $g\neq h$ and 
\[
\lim_{n\to\infty}
\max_{\tau\in T(A)}\lvert\tau(f_n)-\lvert K\rvert^{-1}\rvert=0. 
\]
In addition, by the proof of Lemma \ref{Q} (4), 
we may assume that $K$ is of the form 
\[
\{b^ja^i\in\Gamma\mid0\leq i\leq m_1-1,\quad 0\leq j\leq m_2-1\}
\]
for some $m_1,m_2\in\N$. 
Since $\alpha_a^r$ is approximately inner, 
when $m_1$ is large enough, 
we can construct a central sequence of projections $(f'_n)_n$ in $A$ 
such that 
\[
\lim_{n\to\infty}\lVert f'_n(f_n+\alpha_a(f_n)+\dots+\alpha_a^{m_1-1}(f_n))
-f'_n\rVert=0, 
\]
\[
\alpha_a(f'_n)\approx f'_n\quad\text{and}\quad 
\tau(f'_n)\approx m_1\tau(f_n)\quad\forall\tau\in T(A), 
\]
by using the arguments in \cite[Lemma 3.1]{K95crelle} 
(see also \cite[Lemma 6]{N}). 
Consequently we obtain the following: 
for any $m\in\N$, 
there exists a central sequence of projections $(e_n)_n$ in $A$ such that 
\[
\lim_{n\to\infty}
\max_{\tau\in T(A)}\lvert\tau(e_n)-m^{-1}\rvert=0,\quad 
\lim_{n\to\infty}\lVert e_n-\alpha_a(e_n)\rVert=0
\]
and 
\[
\lim_{n\to\infty}\lVert e_n\alpha_b^j(e_n)\rVert=0
\]
for all $j=1,2,\dots,m{-}1$. 
By using this instead of \cite[Lemma 5.2]{M10CMP}, 
one can complete the proof in exactly the same way 
as that of \cite[Theorem 5.5]{M10CMP}. 
\end{proof}

Next we would like to establish 
a $\Gamma$-action version of \cite[Theorem 6.1]{M10CMP} 
(or \cite[Theorem 4.11]{IM}). 
We note that the idea of the proof is taken from \cite{IM2}. 

\begin{theorem}\label{KleinAF}
Let $A$ be a unital simple AF algebra 
with finitely many extremal tracial states and 
let $\phi\in\aInn(A)$. 
For $i=1,2$, assume that 
$\psi_i\in\Aut(A\rtimes_\phi\Z)$ satisfies the following conditions. 
\begin{enumerate}
\item $\psi_i(A)=A$ and $\psi_i(\lambda^\phi)\lambda^\phi\in A$. 
\item $\phi^n\circ(\psi_i|A)^m\in\Aut(A)$ is not weakly inner 
for each $(n,m)\in\Z^2\setminus\{(0,0)\}$. 
\item $\psi_i^r|A$ is approximately inner for some $r\in\N$. 
\end{enumerate}
Assume further that $\psi_1\circ\psi_2^{-1}$ is asymptotically inner. 
Then there exists an approximately inner automorphism 
$\mu\in\Aut_\T(A\rtimes_\phi\Z)$ and a unitary $v\in U(A)$ 
such that $\mu|A$ is also approximately inner and 
\[
\mu\circ\psi_1\circ\mu^{-1}=\Ad v\circ\psi_2. 
\]
\end{theorem}
\begin{proof}
We would like to 
apply the Evans-Kishimoto intertwining argument to $\psi_1$ and $\psi_2$. 
It is easy to see that 
$\psi_1\circ\psi_2^{-1}$ is in $\Aut_\T(A\rtimes_\phi\Z)$. 
By Proposition \ref{examples} (1), 
(the $\Z$-action generated by) $\phi$ is asymptotically representable. 
Then \cite[Theorem 4.8]{IM} implies that 
$\psi_1\circ\psi_2^{-1}$ is $\T$-asymptotically inner. 
For each $i=1,2$, 
the cocycle action of $\Gamma$ generated by $\phi$ and $\psi_i|A$ 
is strongly outer. 
It follows from the theorem above that 
we can find Rohlin projections for $\psi_i$ 
in the fixed point algebra $(A_\infty)^\phi$. 
Hence, 
with the help of \cite[Lemma 4.10]{M10CMP} (or \cite[Lemma 3.3]{M11crelle}), 
the (equivariant version of) Evans-Kishimoto intertwining argument 
shows the statement. 
\end{proof}

We are now ready to prove the uniqueness of 
strongly outer cocycle actions of the Klein bottle group on UHF algebras. 

\begin{theorem}\label{KleinUHF}
Let $(\alpha,u)$ and $(\beta,v)$ be strongly outer cocycle actions 
of $\Gamma$ on a UHF algebra $A$. 
Then $(\alpha,u)$ and $(\beta,v)$ are cocycle conjugate. 
\end{theorem}
\begin{proof}
Put $\check u=u(b,a)u(a^{-1},b)^*u(a^{-1},a)$ and 
$\check v=v(b,a)v(a^{-1},b)^*v(a^{-1},a)$. 
Because 
\[
\check u(\lambda^{\alpha_a})^*\alpha_b(x)\lambda^{\alpha_a}\check u^*
=\Ad\check u(\alpha^{-1}_a(\alpha_b(x))=\alpha_b(\alpha_a(x))
\]
for $x\in A$, 
the automorphism $\alpha_b$ extends to 
$\tilde\alpha_b\in\Aut(A\rtimes_{\alpha_a}\Z)$ by letting 
$\tilde\alpha_b(\lambda^{\alpha_a})=\check u(\lambda^{\alpha_a})^*$. 
Likewise, we can extend $\beta_b$ 
to $\tilde\beta_b\in\Aut(A\rtimes_{\beta_a}\Z)$. 
By \cite[Theorem 1.3]{K95crelle}, we may assume that 
there exists $w\in U(A)$ such that $\beta_a=\Ad w\circ\alpha_a$. 
Define an isomorphism $\theta:A\rtimes_{\beta_a}\Z\to A\rtimes_{\alpha_a}\Z$ 
by $\theta(x)=x$ for $x\in A$ and 
$\theta(\lambda^{\beta_a})=w\lambda^{\alpha_a}$. 
Then $\theta\circ\tilde\beta_b\circ\theta^{-1}\circ\tilde\alpha_b^{-1}$ is 
in $\Aut_\T(A\rtimes_{\alpha_a}\Z)$. 
Since 
\[
\OrderExt(K_0(A),K_0(A),D_A)\cong\Hom(K_0(A),\R/\Ima D_A)
\]
is divisible, 
there exists $c\in\OrderExt(K_0(A),K_0(A),D_A)$ such that 
\[
-2c=\tilde\eta_0(\theta\circ\tilde\beta_b\circ\theta^{-1}
\circ\tilde\alpha_b^{-1}). 
\]
By Lemma \ref{toadjust}, 
there exists $\psi\in\Aut_\T(A\rtimes_{\alpha_a}\Z)$ 
such that $\tilde\eta_0(\psi)=c$. 
Then 
\begin{align*}
&\tilde\eta_0(
\psi\circ\theta\circ\tilde\beta_b\circ\theta^{-1}\circ\psi^{-1}
\circ\tilde\alpha_b^{-1})\\
&=\tilde\eta_0(
\psi\circ\theta\circ\tilde\beta_b\circ\theta^{-1}\circ\psi^{-1}
\circ\theta\circ\tilde\beta_b^{-1}\circ\theta^{-1}
\circ\theta\circ\tilde\beta_b\circ\theta^{-1}\circ\tilde\alpha_b^{-1})\\
&=\tilde\eta_0(\psi)+\tilde\eta_0(
\theta\circ\tilde\beta_b\circ\theta^{-1}\circ\psi^{-1}
\circ\theta\circ\tilde\beta_b^{-1}\circ\theta^{-1})
+\tilde\eta_0(
\theta\circ\tilde\beta_b\circ\theta^{-1}\circ\tilde\alpha_b^{-1})\\
&=c+c+\tilde\eta_0(
\theta\circ\tilde\beta_b\circ\theta^{-1}\circ\tilde\alpha_b^{-1})=0, 
\end{align*}
where we have used Lemma \ref{reverse}. 
Therefore, by Lemma \ref{eta>asymp}, 
\[
\psi\circ\theta\circ\tilde\beta_b\circ\theta^{-1}\circ\psi^{-1}
\circ\tilde\alpha_b^{-1}\in\Aut_\T(A\rtimes_{\alpha_a}\Z)
\]
is asymptotically inner. 
It follows from the theorem above that 
there exist an approximately inner automorphism 
$\mu\in\Aut_\T(A\rtimes_{\alpha_a}\Z)$ and a unitary $t\in U(A)$ 
such that 
\begin{equation}
\mu\circ\psi\circ\theta\circ\tilde\beta_b\circ\theta^{-1}\circ\psi^{-1}
\circ\mu^{-1}=\Ad t\circ\tilde\alpha_b. 
\label{tiger}
\end{equation}
Let $\nu=\mu\circ\psi$ and 
let $s=\nu(\lambda^{\alpha_a})(\lambda^{\alpha_a})^*\in U(A)$. 
One can deduce 
\begin{equation}
(\nu|A)\circ\beta_b\circ(\nu|A)^{-1}=\Ad t\circ\alpha_b
\label{rabbit}
\end{equation}
and 
\begin{equation}
(\nu|A)\circ\beta_a\circ(\nu|A)^{-1}=(\Ad\nu(w)s)\circ\alpha_a. 
\label{dragon}
\end{equation}
Moreover, from \eqref{tiger}, 
\begin{align*}
(\nu\circ\theta\circ\tilde\beta_b\circ\theta^{-1}\circ\nu^{-1})
(\lambda^{\alpha_a})
&=(\nu\circ\theta\circ\tilde\beta_b\circ\theta^{-1})
(\nu^{-1}(s^*)\lambda^{\alpha_a})\\
&=(\nu\circ\theta\circ\tilde\beta_b)(\nu^{-1}(s^*)w^*\lambda^{\beta_a})\\
&=(\nu\circ\theta)
(\beta_b(\nu^{-1}(\nu(w)s)^*)\check v(\lambda^{\beta_a})^*)\\
&=\nu(\beta_b(\nu^{-1}(\nu(w)s)^*)\check v(\lambda^{\alpha_a})^*w^*)\\
&=\Ad t(\alpha_b(\nu(w)s)^*)\nu(\check v)(\lambda^{\alpha_a})^*(\nu(w)s)^*\\
&=\Ad t(\alpha_b(\nu(w)s)^*)\nu(\check v)
\alpha_a^{-1}(\nu(w)s)^*(\lambda^{\alpha_a})^*
\end{align*}
is equal to 
\[
\Ad t(\tilde\alpha_b(\lambda^{\alpha_a}))
=t\check u\alpha_a^{-1}(t^*)(\lambda^{\alpha_a})^*. 
\]
Hence we can verify 
\[
\nu(\check v)
=t\alpha_b(\nu(w)s)\check u\alpha_a^{-1}(t^*)\alpha_a^{-1}(\nu(w)s). 
\]
This, together with \eqref{rabbit}, \eqref{dragon}, 
implies that $(\alpha,u)$ and $(\beta,v)$ are cocycle conjugate. 
\end{proof}

\begin{corollary}\label{UHFapprepre}
Any strongly outer cocycle action of $\Gamma$ on a UHF algebra $A$ 
is approximately representable. 
\end{corollary}
\begin{proof}
This immediately follows from the theorem above and Proposition \ref{exist}. 
\end{proof}

\begin{remark}\label{Koftwistedcc}
Let $(\alpha,u):\Gamma\curvearrowright A$ be 
a strongly outer cocycle action on a UHF algebra $A$. 
By \cite[Corollary 5.9]{MS}, 
$A\rtimes_{(\alpha,u)}\Gamma
\cong(A\rtimes_{\alpha_a}\Z)\rtimes_{\tilde\alpha_b}\Z$ is 
a unital simple AH algebra with real rank zero and slow dimension growth. 
By means of the Pimsner-Voiculescu exact sequence, we get 
\[
K_0(A\rtimes_{(\alpha,u)}\Gamma)\cong K_0(A),\quad 
K_1(A\rtimes_{(\alpha,u)}\Gamma)\cong
\begin{cases}K_0(A)&\zeta_A(2)=\infty\\
K_0(A)\oplus\Z/2\Z&\text{otherwise. }\end{cases}
\]
Note that $\zeta_A(2)=\infty$ if and only if 
$A$ absorbs the UHF algebra of type $2^\infty$ tensorially. 
\end{remark}

%%%%%%%%%%%%%%%%%%%%%%%%%%%%%%%%%%%%%%%%%%%%%%%%%%%%%%%%%%%%
\section{Actions of the Klein bottle group II}

Throughout this section, we let $\Gamma$ denote 
the Klein bottle group $\langle a,b\mid bab^{-1}=a^{-1}\rangle$. 
In this section we will prove that 
all strongly outer cocycle actions of $\Gamma$ on the Jiang-Su algebra 
are cocycle conjugate to each other (Theorem \ref{KleinZ}).

%%%%%%%%%%%%%%%%%%%%%%%%%%%%%%%%%%%%%%%%%%%%%%%%%%%%%%%%%%%%
\subsection{Homotopy of unitaries}

In this subsection 
we establish Lemma \ref{homotopy1} and Lemma \ref{homotopy2}, 
which will be used in the next subsection 
to obtain certain cohomology vanishing type results. 

The following lemma is 
a variant of \cite[Lemma 4.10]{M10CMP} or \cite[Lemma 6.3]{MS}, 
and is essentially contained in the proof of \cite[Theorem 5.3]{S10JFA}. 
But we include a proof here for completeness. 

\begin{lemma}\label{homotopy1}
Let $B$ be a UHF algebra with a unique trace $\tau$ and 
let $\phi$ be an automorphism of $B$ 
such that $\phi^n$ is not weakly inner for all $n\in\N$. 
For any finite subset $F\subset B$ and $\ep>0$, 
there exist a finite subset $G\subset B$ and $\delta>0$ 
such that the following holds. 
If $x\in U(B)$ satisfies 
\[
\lVert[x,c]\rVert<\delta\quad\forall c\in G,\quad 
\lVert\phi(x)-x\rVert<\delta,\quad 
\tau(\log(x\phi(x^*)))=0, 
\]
then we can find a continuous map $z:[0,1]\to U(B)$ such that 
$z(0)=1$, $z(1)=x$, 
\[
\lVert[z(t),c]\rVert<\ep\quad\forall c\in F,\quad 
\lVert\phi(z(t))-z(t)\rVert<\ep\quad\forall t\in[0,1]\quad\text{and}\quad 
\Lip(z)<4. 
\]
In addition, if $F=\emptyset$, then $G=\emptyset$ is possible. 
\end{lemma}
\begin{proof}
The case $F=G=\emptyset$ is contained in \cite[Lemma 8.1]{MS}, 
and so we treat the general case here. 
Let $(d_n)_n$ be a sequence of natural numbers such that 
$d_n<d_{n+1}$ and 
\[
B\cong M_{d_1}\otimes M_{d_2}\otimes M_{d_3}\otimes\dots. 
\]
We regard $B_n=M_{d_1}\otimes\dots\otimes M_{d_n}$ 
as a subalgebra of $B$. 
Let $y_n\in B_n\cap B_{n-1}'$ be a unitary 
whose spectrum is $\{\zeta_n^k\mid k=1,2,\dots,d_n\}$, 
where $\zeta_n=\exp(2\pi\sqrt{-1}/d_n)$. 
Define an automorphism $\sigma$ of $B$ 
by $\sigma=\lim_{n\to\infty}\Ad(y_1y_2\dots y_n)$. 
Then $\sigma^m$ is not weakly inner for all $m\in\N$. 
By \cite[Theorem 1.3, 1.4]{K95crelle}, 
the $\Z$-actions generated by $\phi$ and $\sigma$ are 
strongly cocycle conjugate. 
Hence it suffices to show the assertion for $\sigma$. 

Suppose that we are given $F\subset B$ and $\ep>0$. 
Without loss of generality, 
we may assume that there exists $n\in\N$ such that 
$F$ is contained in the unit ball of $B_n$. 
Applying \cite[Lemma 4.2]{KM} to $\ep/2$, 
we obtain a positive real number $\delta_1>0$. 
We may assume $\delta_1$ is less than $\min\{2,\ep\}$. 
Choose a finite subset $G\subset B$ and $\delta_2>0$ so that 
if $u\in U(B)$ satisfies $\lVert[u,c]\rVert<\delta_2$ for all $c\in G$, 
then there exists $u_0\in U(B\cap B_n')$ 
such that $\lVert u-u_0\rVert<\delta_1/4$. 
Let $\delta=\min\{\delta_1/2,\delta_2\}$. 

Suppose that $x\in U(B)$ satisfies 
\[
\lVert[x,c]\rVert<\delta\quad\forall c\in G,\quad 
\lVert\sigma(x)-x\rVert<\delta,\quad 
\tau(\log(x\sigma(x^*)))=0. 
\]
By the choice of $\delta$, 
we can find $x_0\in B\cap B_n'$ such that $\lVert x-x_0\rVert<\delta_1/4$. 
We may assume that there exists $m>n$ such that $x_0\in B_m\cap B_n'$. 
Put $y=y_{n+1}y_{n+2}\dots y_m\in B_m\cap B_n'$. 
Then 
\[
\lVert[y,x_0]\rVert=\lVert\sigma(x_0)-x_0\rVert
<\lVert\sigma(x)-x\rVert+\delta_1/2<\delta+\delta_1/2\leq\delta_1. 
\]
Also, from $\lVert x-x_0\rVert<\delta_1/4<1/2$ and 
$\lVert x-\sigma(x)\rVert<\delta<1$, one has 
\begin{align*}
\tau(\log(x_0yx_0^*y^*))
&=\tau(\log(x_0\sigma(x_0^*)))\\
&=\tau(\log(x_0x^*x\sigma(x^*)\sigma(xx_0^*)))\\
&=\tau(\log(x\sigma(x^*)))=0. 
\end{align*}
It follows from \cite[Lemma 4.2]{KM} that 
one can find a path of unitaries $w:[0,1]\to B_m\cap B_n'$ such that 
$w(0)=1$, $w(1)=x_0$, $\Lip(w)\leq\pi+\ep$ and 
$\lVert[y,w(t)]\rVert<\ep/2$ for every $t\in[0,1]$. 
Note that $yw(t)y^*$ is equal to $\sigma(w(t))$. 
Take a path of unitaries $w':[0,1]\to B$ such that 
$w'(0)=1$, $w'(1)=xx_0^*$ and $\Lip(w')<\delta_1/4$. 
Then the path $z(t)=w'(t)w(t)$ meets the requirement. 
\end{proof}

In order to prove Lemma \ref{homotopy2}, 
we need the following lemma. 
The proof uses \cite{L10Mem}. 

\begin{lemma}
Let $C$ be a unital simple AT algebra with a unique trace $\tau$. 
Assume that 
$K_i(C)$ is isomorphic to a non-zero subgroup of $\Q$ for each $i=0,1$. 
Let $w\in U(C)$ be such that $[w]\neq0$. 
Suppose that $(x_n)_n$ is a central sequence of unitaries in $C$ 
such that $[x_n]=0$ and 
$\tau(\log(x_nwx_n^*w^*))=0$ for all sufficiently large $n$. 
Then there exists a central sequence $(z_n)_n$ of unitaries in $C([0,1],C)$ 
such that $z_n(0)=1$, $z_n(1)=x_n$ and $\Lip(z_n)<7$ 
for all sufficiently large $n$. 
\end{lemma}
\begin{proof}
For almost commuting unitaries $u,v$ in a unital $C^*$-algebra $A$, 
we denote their Bott class by $\Bott(u,v)\in K_0(A)$ 
(\cite[Section 1]{BEEK}). 
When $A$ is a unital simple separable $C^*$-algebra with tracial rank zero, 
by \cite[Theorem 3.6]{L09AJM}, we have 
\[
\tau(\Bott(u,v))=\frac{1}{2\pi\sqrt{-1}}\tau(\log(vuv^*u^*))
\quad\forall\tau\in T(A). 
\]

Let $x\in C_\infty$ be the image of $(x_n)_n$. 
Since the unitary $x$ commutes with $C$, 
there exists a homomorphism $\pi:C_0(\T\setminus\{1\})\otimes C\to C^\infty$ 
such that $\pi(f\otimes c)=f(x)c$ 
for every $f\in C_0(\T\setminus\{1\})$ and $c\in C$. 
We regard $K_i(\pi)$ as homomorphisms from $K_{1-i}(C)$ to $K_i(C^\infty)$. 
Then $K_1(\pi)([1])=[x]=0$, 
because $[x_n]=0$ in $K_1(C)$ for all sufficiently large $n$. 
Also, one has $K_0(\pi)([w])=\Bott(w,x)=0$, 
because $\tau(\Bott(w,x_n))=(2\pi\sqrt{-1})^{-1}\tau(\log(x_nwx_n^*w^*))=0$ 
for all sufficiently large $n$ and $K_0(C)$ has no infinitesimal. 
As $K_i(C)$ is isomorphic to a non-zero subgroup of $\Q$ for each $i=0,1$, 
we obtain $K_1(\pi)=0$ and $K_0(\pi)=0$. 
Hence we have the following. 
\begin{itemize}
\item Let $p\in C$ be a projection. 
Then for all sufficiently large $n$, the $K_1$-class of $x_np+1-p$ is zero. 
\item Let $u\in C$ be a unitary. 
Then for all sufficiently large $n$, $\Bott(u,x_n)$ equals zero. 
\end{itemize}

Once this is done, one can construct $(z_n)_n$ as follows. 
Let $F\subset C$ be a finite subset and let $\ep>0$. 
By applying \cite[Corollary 17.9]{L10Mem} 
(see also \cite[Theorem 5.4]{M11JFA}), 
we get a finite subset $G\subset C$, $\delta>0$, 
a finite subset $P$ of projections in $C$ and 
a finite subset $U$ of unitaries in $C$. 
Notice that 
we do not need to work with the entire $K$-group $\underline{K}(C)$ of $C$, 
because $K_i(C)$ are torsion free. 
For all sufficiently large $n$, 
\[
\lVert[c,x_n]\rVert<\delta,\quad [x_np+1-p]=0\quad\text{and}\quad 
\Bott(u,x_n)=0
\]
hold for all $c\in G$, $p\in P$ and $u\in U$. 
It follows from \cite[Corollary 17.9]{L10Mem} that 
there exists a continuous map $z_n:[0,1]\to U(C)$ such that 
$z_n(0)=1$, $z_n(1)=x_n$, $\Lip(z_n)<2\pi+\ep$ and 
$\lVert[c,z_n(t)]\rVert<\ep$ for all $c\in F$ and $t\in[0,1]$. 
Since $F$ and $\ep$ were arbitrary, the proof is completed. 
\end{proof}

\begin{lemma}\label{homotopy2}
Let $B$ be a UHF algebra with a unique trace $\tau$ and 
let $(\alpha,u):\Gamma\curvearrowright B$ be a strongly outer cocycle action. 
Suppose that $B$ absorbs the UHF algebra of type $2^\infty$ tensorially. 
For any finite subset $F\subset B$ and $\ep>0$, 
there exist a finite subset $G\subset B$ and $\delta>0$ 
such that the following holds. 
If $x\in U(B)$ satisfies 
\[
\lVert[x,c]\rVert<\delta\quad\forall c\in G,\quad 
\lVert\alpha_a(x)-x\rVert<\delta,\quad 
\lVert\alpha_b(x)-x\rVert<\delta,\quad 
\tau(\log(x\alpha_b(x^*)))=0, 
\]
then we can find a continuous map $z:[0,1]\to U(B)$ such that 
$z(0)=1$, $z(1)=x$, $\Lip(z)<7$, 
\[
\lVert[z(t),c]\rVert<\ep,\quad 
\lVert\alpha_a(z(t))-z(t)\rVert<\ep\quad\text{and}\quad 
\lVert\alpha_b(z(t))-z(t)\rVert<\ep
\]
for all $c\in F$ and $t\in[0,1]$. 
\end{lemma}
\begin{proof}
The proof is by contradiction. 
Assume that there exist a finite subset $F\subset B$ and $\ep>0$ 
such that the assertion does not hold. 
Then 
we would have a central sequence $(x_n)_n$ of unitaries in $B$ satisfying 
\[
\lim_{n\to\infty}\lVert\alpha_g(x_n)-x_n\rVert=0\quad\forall g\in\Gamma
\quad\text{and}\quad 
\tau(\log(x_n\alpha_b(x_n^*)))=0, 
\]
and such that there does not exist a unitary $z\in C([0,1],B)$ 
as described in the statement. 

Let $C=B\rtimes_{(\alpha,u)}\Gamma$ be 
the twisted crossed product $C^*$-algebra. 
Since $B$ absorbs the UHF algebra of type $2^\infty$, 
by Remark \ref{Koftwistedcc}, 
$C$ is a unital simple AT algebra with real rank zero 
and $K_0(C)\cong K_1(C)\cong K_0(B)$. 
Evidently $[\lambda^\alpha_b]$ is not zero in $K_1(C)$. 
We regard $B$ as a subalgebra of $C$ and 
think of $\tau$ as the unique tracial state on $C$. 
Then the sequence $(x_n)_n$ consisting of unitaries in $B$ 
is a central sequence in $C$ such that $[x_n]=0$ in $K_1(C)$ and 
\[
\tau(\log(x_n\lambda^\alpha_bx_n^*(\lambda^\alpha_b)^*)))
=\tau(\log(x_n\alpha_b(x_n^*)))=0. 
\]
It follows from the lemma above that 
there exists a central sequence $(z_n)_n$ of unitaries in $C([0,1],C)$ 
such that $z_n(0)=1$, $z_n(1)=x_n$ and $\Lip(z_n)<7$ 
for all sufficiently large $n$. 
By Corollary \ref{UHFapprepre}, 
$(\alpha,u):\Gamma\curvearrowright B$ is approximately representable. 
Therefore, in the same way as \cite[Lemma 3.3]{M11crelle}, 
we may further assume that 
$z_n(t)$ is in $B$ for all $n\in\N$ and $t\in[0,1]$. 
When $n$ is sufficiently large, we have 
\[
\lVert[z_n(t),c]\rVert<\ep,\quad 
\lVert\alpha_a(z_n(t))-z_n(t)\rVert<\ep\quad\text{and}\quad 
\lVert\alpha_b(z_n(t))-z_n(t)\rVert<\ep
\]
for all $c\in F$ and $t\in[0,1]$, 
which is a contradiction. 
\end{proof}

%%%%%%%%%%%%%%%%%%%%%%%%%%%%%%%%%%%%%%%%%%%%%%%%%%%%%%%%%%%%
\subsection{Cohomology vanishing}

In this subsection we prove Proposition \ref{CVonZ}, 
which is a cohomology vanishing type result 
for cocycle actions of $\Gamma$ on $\mathcal{Z}$. 
We also prove that 
all strongly outer actions of $\Gamma$ on a certain UHF algebra 
are strongly cocycle conjugate (Theorem \ref{KleinUHF2}). 

\begin{lemma}\label{CVonUHF1}
Let $(\alpha,u):\Gamma\curvearrowright B$ be a strongly outer cocycle action 
on a UHF algebra $B$ with a unique trace $\tau$. 
For any finite subset $F\subset B$ and $\ep>0$, 
there exist a finite subset $G\subset B$ and $\delta>0$ 
such that the following holds. 
If $x\in U(B)$ satisfies 
\[
\lVert[x,c]\rVert<\delta\quad\forall c\in G,\quad 
\lVert\alpha_a(x)-x\rVert<\delta,\quad 
\tau(\log(x\alpha_a(x^*)))=0, 
\]
then there exists $y\in U(B)$ such that 
\[
\lVert[y,c]\rVert<\ep\quad\forall c\in F,\quad 
\lVert\alpha_a(y)-y\rVert<\ep\quad\text{and}\quad 
\lVert x-y\alpha_b(y^*)\rVert<\ep. 
\]
In addition, if $F=\emptyset$ and $u(g,h)=1$ for all $g,h\in\Gamma$, 
then $G=\emptyset$ is possible. 
\end{lemma}
\begin{proof}
The proof is by contradiction. 
Assume that there exist a finite subset $F\subset B$ and $\ep>0$ 
such that the assertion does not hold. 
Then 
we would have a central sequence $(x_n)_n$ of unitaries in $B$ satisfying 
$\alpha_a(x_n)-x_n\to0$ as $n\to\infty$ and 
$\tau(\log(x_n\alpha_a(x_n^*)))=0$, 
and such that there does not exist a unitary $y\in B$ 
as described in the statement. 
Let $x\in U((B_\infty)^{\alpha_a})$ be the image of $(x_n)_n$. 
Choose $m\in\N$ so that $4/m<\ep$. 
Define $x_{i,n}\in U(B)$ for $i=0,1,2,\dots$ 
by $x_{0,n}=1$ and $x_{i+1,n}=x_n\alpha_b(x_{i,n})$. 
It is easy to see that $(x_{i,n})_n$ is also a central sequence and 
$\alpha_a(x_{i,n})-x_{i,n}\to0$ as $n\to\infty$. 
Put $\check u=u(b,a)u(a^{-1},b)^*u(a^{-1},a)$ 
(we need $u(a^{-1},a)$ 
because $\alpha_{a^{-1}}$ is not necessarily equal to $\alpha_a^{-1}$). 
For each $i$, when $n$ is sufficiently large, 
\begin{align*}
&\tau(\log(x_{i+1,n}\alpha_a(x_{i+1,n}^*)))\\
&=\tau(\log(x_n\alpha_b(x_{i,n})\alpha_a(\alpha_b(x_{i,n}^*)x_n^*)))\\
&=\tau(\log(\alpha_b(x_{i,n})\alpha_a(\alpha_b(x_{i,n}^*))))
+\tau(\log(\alpha_a(x_n^*)x_n))\\
&=\tau(\log(\alpha_a^{-1}(\alpha_b(x_{i,n}))\alpha_b(x_{i,n}^*)))\\
&=\tau(\log(\check u^*
\alpha_b(\alpha_a(x_{i,n}))\check u\alpha_b(x_{i,n}^*)))\\
&=\tau(\log
(\check u^*\alpha_b(\alpha_a(x_{i,n}))\alpha_b(x_{i,n}^*)\check u))
+\tau(\log(\check u^*\alpha_b(x_{i,n})\check u\alpha_b(x_{i,n}^*)))\\
&=\tau(\log(\alpha_a(x_{i,n})x_{i,n}^*)), 
\end{align*}
where the last equality follows 
since $(\alpha_b(x_{i,n}))_n$ is central and 
$\check u$ is homotopic to $1$. 
Therefore, for each $i$, 
we have $\tau(\log(x_{i,n}\alpha_a(x_{i,n}^*)))=0$ 
for all sufficiently large $n$. 

Let $x_i\in U((B_\infty)^{\alpha_a})$ be the image of $(x_{i,n})_n$. 
By Lemma \ref{homotopy1}, 
there exist paths of unitaries $z_0,z_1\in C([0,1],(B_\infty)^{\alpha_a})$ 
such that $z_0(1)=z_1(1)=1$, $z_0(0)=x_m$, $z_1(0)=x_{m+1}$, 
$\Lip(z_0)\leq4$ and $\Lip(z_1)\leq4$. 
By Theorem \ref{so>R}, 
there exist projections $e,f\in(B_\infty)^{\alpha_a}$ such that 
\[
\alpha_b^m(e)=e,\quad \alpha_b^{m+1}(f)=f
\]
and 
\[
\sum_{i=0}^{m-1}\alpha_b^i(e)+\sum_{j=0}^m\alpha_b^j(f)=1. 
\]
We may further assume that 
the projections $e,f$ commute with 
$\{\alpha_b^k(x_i),z_0(t),z_1(t)\mid k\in\Z,i\in\N,t\in[0,1]\}$. 
Define a unitary $y\in(B_\infty)^{\alpha_a}$ by 
\[
y=\sum_{i=0}^{m-1}x_i\alpha_b^i(z_0(i/m))\alpha_b^i(e)
+\sum_{j=0}^mx_j\alpha_b^j(z_1(j/(m{+}1)))\alpha_b^j(f). 
\]
It is easy to check $\lVert x-y\alpha_b(y^*)\rVert\leq4/m<\ep$. 
This is a contradiction. 

We consider the case $F=\emptyset$ and $u(g,h)=1$ for all $g,h\in\Gamma$. 
In this case we would have a unitary $x$ in $(B^\infty)^{\alpha_a}$, 
and $\check u$ defined above is equal to $1$. 
By replacing $(B_\infty)^{\alpha_a}$ with $(B^\infty)^{\alpha_a}$, 
the same proof works. 
\end{proof}

\begin{lemma}
Let $B$ be a UHF algebra of infinite type with a unique trace $\tau$ and 
let $(\alpha,u):\Gamma\curvearrowright B$ be a strongly outer cocycle action. 
For any finite subset $F\subset B$, $\ep>0$ and 
$r\in\tau(K_0(B))$ with $\lvert r\rvert<1/2$, 
there exists $z\in U(B)$ such that 
\[
\lVert[z,c]\rVert<\ep\quad\forall c\in F,\quad 
\lVert z-\alpha_a(z)\rVert<\ep,\quad 
\lVert z-\alpha_b(z)\rVert<\lvert e^{2\pi\sqrt{-1}r}-1\rvert+\ep
\]
and 
\[
\tau(\log(z\alpha_b(z^*)))=2\pi\sqrt{-1}r. 
\]
\end{lemma}
\begin{proof}
Suppose that we are given $r\in\tau(K_0(B))$ with $\lvert r\rvert<1/2$. 
Let $\sigma=\lim\Ad(y_1y_2\dots y_n)$ be the automorphism of 
$B=M_{d_1}\otimes M_{d_2}\otimes M_{d_3}\otimes\dots$ 
constructed in the proof of Lemma \ref{homotopy1}. 
Since $B$ is of infinite type, we may further impose the condition 
that $d_n$ divides $d_{n+1}$. 
We regard $M_{d_n}$ as a subalgebra of $B$. 
When $n$ is sufficiently large, there exists a unitary $z_n\in M_{d_n}$ 
such that $z_ny_nz_n^*y_n^*=e^{2\pi\sqrt{-1}r}$. 
Hence we have $\tau(\log(z_n\sigma(z_n^*)))=2\pi\sqrt{-1}r$ and 
$\lVert z_n-\sigma(z_n)\rVert=\lvert e^{2\pi\sqrt{-1}r}-1\rvert$. 
Moreover $(z_n)_n$ is a central sequence in $B$. 

Take a strongly outer action $\gamma:\Gamma\curvearrowright\mathcal{Z}$. 
Define $\beta:\Gamma\curvearrowright\mathcal{Z}\otimes B$ 
by $\beta_a=\gamma_a\otimes\id$ and $\beta_b=\gamma_b\otimes\sigma$. 
Set $z'_n=1\otimes z_n\in\mathcal{Z}\otimes B$. 
It is clear that 
$(z'_n)_n$ is a central sequence in $\mathcal{Z}\otimes B$ 
satisfying $z'_n=\beta_a(z'_n)$, 
$\lVert z'_n-\beta_b(z'_n)\rVert=\lvert e^{2\pi\sqrt{-1}r}-1\rvert$ and 
$\tau(\log(z'_n\beta_b(z'_n)^*))=2\pi\sqrt{-1}r$. 

Let $(\alpha,u):\Gamma\curvearrowright\mathcal{Z}\otimes B$ be 
a strongly outer cocycle action. 
By Theorem \ref{KleinUHF}, $(\alpha,u)$ is cocycle conjugate to $\beta$. 
Thus, there exist $\theta\in\Aut(\mathcal{Z}\otimes B)$ and 
a family of unitaries $(w_g)_{g\in\Gamma}$ in $\mathcal{Z}\otimes B$ such that 
\[
\theta\circ\alpha_g\circ\theta^{-1}=\Ad w_g\circ\beta_g
\quad\forall g\in\Gamma. 
\]
Put $z''_n=\theta^{-1}(z'_n)$. 
Then $(z''_n)_n$ is again a central sequence in $\mathcal{Z}\otimes B$. 
It is straightforward to see 
\[
\lim_{n\to\infty}\lVert z''_n-\alpha_a(z''_n)\rVert=0\quad\text{and}\quad 
\lim_{n\to\infty}\lVert z''_n-\alpha_b(z''_n)\rVert
=\lvert e^{2\pi\sqrt{-1}r}-1\rvert. 
\]
In addition, for sufficiently large $n$, 
\[
\tau(\log(z''_n\alpha_b(z''_n)^*))
=\tau(\log(z'_nw_b\beta_b(z'_n)^*w_b^*))
=\tau(\log(z'_n\beta_b(z'_n)^*))
=2\pi\sqrt{-1}r, 
\]
because $(z'_n)_n$ is central and $w_b$ is homotopic to $1$. 
The proof is completed. 
\end{proof}

\begin{lemma}\label{CVonUHF2}
Let $(\alpha,u):\Gamma\curvearrowright B$ be a strongly outer cocycle action 
on a UHF algebra $B$ of infinite type with a unique trace $\tau$. 
For any finite subset $F\subset B$ and $\ep>0$, 
there exist a finite subset $G\subset B$ and $\delta>0$ 
such that the following holds. 
If $x\in U(B)$ satisfies 
\[
\lVert[x,c]\rVert<\delta\quad\forall c\in G,\quad 
\lVert\alpha_a(x)-x\rVert<\delta,\quad 
\tau(\log(x\alpha_a(x^*)))=0,\quad 
\Delta_\tau(x)=0
\]
then there exists $y\in U(B)$ such that 
\[
\lVert[y,c]\rVert<\ep\quad\forall c\in F,\quad 
\lVert\alpha_a(y)-y\rVert<\ep,\quad 
\lVert x-y\alpha_b(y^*)\rVert<\ep
\]
and 
\[
\tau(\log(y^*x\alpha_b(y)))=0. 
\]
\end{lemma}
\begin{proof}
Suppose that we are given $F\subset B$ and $\ep>0$. 
By applying Lemma \ref{CVonUHF1} to $F$ and $\ep/2$, 
we obtain a finite subset $G\subset B$ and $\delta>0$. 
We would like to show that $G$ and $\delta$ meet the requirement. 
For $x$ as in the statement of the lemma, 
there exists $y\in U(B)$ such that 
\[
\lVert[y,c]\rVert<\ep/2\quad\forall c\in F,\quad 
\lVert\alpha_a(y)-y\rVert<\ep/2\quad\text{and}\quad 
\lVert x-y\alpha_b(y^*)\rVert<\ep/2. 
\]
Set 
\[
r=\frac{1}{2\pi\sqrt{-1}}\tau(\log(y^*x\alpha_b(y))). 
\]
Then 
\[
\lvert e^{2\pi\sqrt{-1}r}-1\rvert\leq\lVert x-y\alpha_b(y^*)\rVert<\ep/2. 
\]
From $\Delta_\tau(x)=0$, one also obtains $r\in\tau(K_0(B))$. 
By using the lemma above, 
we can find a unitary $z\in B$ such that 
\[
\lVert[z,c]\rVert<\ep/2\quad\forall c\in F,\quad 
\lVert z-\alpha_a(z)\rVert<\ep/2,\quad 
\lVert z-\alpha_b(z)\rVert<\lvert e^{2\pi\sqrt{-1}r}-1\rvert+\ep/2
\]
and 
\[
\tau(\log(z\alpha_b(z^*)))=2\pi\sqrt{-1}r. 
\]
Then 
\[
\tau(\log(z^*y^*x\alpha_b(yz)))
=\tau(\log(y^*x\alpha_b(y)))+\tau(\log(\alpha_b(z)z^*))=0. 
\]
Thus $yz$ does the job. 
\end{proof}

The following is a cohomology vanishing type result 
for cocycle actions of the Klein bottle group $\Gamma$ 
on the Jiang-Su algebra $\mathcal{Z}$. 
We will use this proposition in the next subsection to prove 
the uniqueness of strongly outer cocycle actions of $\Gamma$ on $\mathcal{Z}$. 
In the proof, Corollary \ref{so>Z} plays a central role. 
We remark that the basic idea of the proof is the same as that of 
\cite[Lemma 7.5, 7.6]{MS} or \cite[Theorem 5.3]{S10JFA}. 

\begin{proposition}\label{CVonZ}
Let $(\alpha,u):\Gamma\curvearrowright\mathcal{Z}$ be 
a strongly outer cocycle action and 
let $\tau$ denote the unique trace on $\mathcal{Z}$. 
For any finite subset $F\subset\mathcal{Z}$ and $\ep>0$, 
there exist a finite subset $G\subset\mathcal{Z}$ and $\delta>0$ 
such that the following holds. 
If $x\in U(\mathcal{Z})$ satisfies 
\[
\lVert[x,c]\rVert<\delta\quad\forall c\in G,\quad 
\lVert\alpha_a(x)-x\rVert<\delta,\quad 
\Delta_\tau(x)=0
\]
then there exists $y\in U(\mathcal{Z})$ such that 
\[
\lVert[y,c]\rVert<\ep\quad\forall c\in F,\quad 
\lVert\alpha_a(y)-y\rVert<\ep\quad\text{and}\quad 
\lVert x-y\alpha_b(y^*)\rVert<\ep. 
\]
\end{proposition}
\begin{proof}
According to \cite[Proposition 3.3]{RW10crelle}, 
the Jiang-Su algebra $\mathcal{Z}$ contains a unital subalgebra 
isomorphic to 
\[
Z=\{f\in C([0,1],B_0\otimes B_1)\mid 
f(0)\in B_0\otimes\C, \ f(1)\in \C\otimes B_1\}, 
\]
where $B_0$ and $B_1$ are the UHF algebras of type $2^\infty$ and $3^\infty$, 
respectively. 
We identify $B_0$ and $B_1$ with $B_0\otimes\C$ and $\C\otimes B_1$. 
Set $B=B_0\otimes B_1$ and 
denote the unique trace on $\mathcal{Z}\otimes B$ by $\tau$. 
Clearly 
$(\alpha\otimes\id,u\otimes1):\Gamma\curvearrowright\mathcal{Z}\otimes B_j$ 
and $(\alpha\otimes\id,u\otimes1):\Gamma\curvearrowright\mathcal{Z}\otimes B$ 
are strongly outer. 

Suppose that we are given a finite subset $F\subset\mathcal{Z}$ and $\ep>0$. 
By applying Lemma \ref{homotopy2} to 
$(\alpha\otimes\id,u\otimes1):\Gamma\curvearrowright\mathcal{Z}\otimes B$, 
$\{c\otimes1\mid c\in F\}$ and $\ep/2$, 
we obtain a finite subset $G_1\subset\mathcal{Z}\otimes B$ and $\delta_1>0$. 
By taking a smaller $\delta_1$, 
we may assume that 
there exist finite subsets $G_1'\subset\mathcal{Z}\otimes\C$, 
$G_{1,0}\subset\C\otimes B_0$ and 
$G_{1,1}\subset\C\otimes B_1$ such that $G_1=G_1'\cup G_{1,0}\cup G_{1,1}$, 
and that $\delta_1$ is smaller than $\ep$. 
For each $j=0,1$, by applying Lemma \ref{CVonUHF2} to 
$(\alpha\otimes\id,u\otimes1):\Gamma\curvearrowright\mathcal{Z}\otimes B_j$, 
$G_1'\cup G_{1,j}\cup\{c\otimes 1\mid c\in F\}$ and $\delta_1/2$, 
we obtain a finite subset $G_{2,j}\subset\mathcal{Z}\otimes B_j$ and 
$\delta_{2,j}>0$. 
We may assume that $G_{2,j}$ is contained in 
$\{c\otimes1\mid c\in G_{3,j}\}\cup\{1\otimes d\mid d\in B_j\}$ 
for some finite subsets $G_{3,j}\subset\mathcal{Z}$. 
Let $G=G_{3,0}\cup G_{3,1}$ and $\delta=\min\{\delta_{2,0},\delta_{2,1},2\}$. 

Suppose that we are given a unitary $x\in\mathcal{Z}$ satisfying 
\[
\lVert[x,c]\rVert<\delta\quad\forall c\in G,\quad 
\lVert\alpha_a(x)-x\rVert<\delta,\quad 
\Delta_\tau(x)=0. 
\]
For each $j=0,1$, we have 
\[
\lVert[x\otimes1,c]\rVert<\delta_{2,j}\quad\forall c\in G_{2,j},\quad 
\lVert(\alpha_a\otimes\id)(x\otimes1)-x\otimes1\rVert
<\delta_{2,j}
\]
and 
\[
\Delta_\tau(x\otimes1)=0. 
\]
Since $\tau(K_0(\mathcal{Z}))$ equals $\Z$, we also have 
\[
\tau(\log((x\otimes1)((\alpha_a\otimes\id)(x\otimes1))^*))=0. 
\]
Then Lemma \ref{CVonUHF2} implies that 
there exists $y_j\in U(\mathcal{Z}\otimes B_j)$ such that 
\[
\lVert[y_j,c]\rVert<\delta_1/2
\quad\forall c\in G_1'\cup G_{1,j}\cup\{c\otimes1\mid c\in F\},\quad 
\lVert(\alpha_a\otimes\id)(y_j)-y_j\rVert<\delta_1/2, 
\]
\[
\lVert x\otimes1-y_j(\alpha_b\otimes\id)(y_j^*)\rVert<\delta_1/2
\quad\text{and}\quad \tau(\log(y_j^*(x\otimes1)(\alpha_b\otimes\id)(y_j)))=0. 
\]
Put $y=y_0^*y_1\in\mathcal{Z}\otimes B$. 
Then one can verify 
\[
\lVert[y,c]\rVert<\delta_1\quad\forall c\in G_1,\quad 
\lVert(\alpha_a\otimes\id)(y)-y\rVert<\delta_1,\quad 
\lVert(\alpha_b\otimes\id)(y)-y\rVert<\delta_1
\]
and 
\begin{align*}
&\tau(\log(y(\alpha_b\otimes\id)(y^*)))\\
&=\tau(\log(y_0^*y_1(\alpha_b\otimes\id)(y_1^*y_0)))\\
&=\tau(\log(y_1(\alpha_b\otimes\id)(y_1^*y_0)y_0^*))\\
&=\tau(\log((x\otimes1)^*y_1
(\alpha_b\otimes\id)(y_1^*y_0)y_0^*(x\otimes1)))\\
&=\tau(\log((x\otimes1)^*y_1(\alpha_b\otimes\id)(y_1^*)))
+\tau(\log((\alpha_b\otimes\id)(y_0)y_0^*(x\otimes1)))\\
&=0. 
\end{align*}
It follows from Lemma \ref{homotopy2} that 
there exists a continuous map $z:[0,1]\to U(\mathcal{Z}\otimes B)$ such that 
$z(0)=1$, $z(1)=y$, 
\[
\lVert[z(t),c\otimes1]\rVert<\ep/2,\quad 
\lVert(\alpha_a\otimes\id)(z(t))-z(t)\rVert<\ep/2
\]
and 
\[
\lVert(\alpha_b\otimes\id)(z(t))-z(t)\rVert<\ep/2
\]
for all $c\in F$ and $t\in[0,1]$. 
Let $\tilde w(t)=y_0z(t)$. 
Then one may think of $w$ as a unitary in 
$\mathcal{Z}\otimes Z\subset\mathcal{Z}\otimes\mathcal{Z}$ 
because $w(j)=y_j$ belongs to $\mathcal{Z}\otimes B_j$ for each $j=0,1$. 
Furthermore, 
\[
\lVert[w,c\otimes1]\rVert<\ep\quad\forall c\in F,\quad 
\lVert(\alpha_a\otimes\id)(w)-w\rVert<\ep
\]
and 
\[
\lVert x\otimes1-w(\alpha_b\otimes\id)(w^*)\rVert<\ep. 
\]

By Corollary \ref{so>Z}, there exists a unital embedding 
of $\mathcal{Z}$ into $(\mathcal{Z}_\infty)^\alpha$. 
Hence we can find a unital homomorphism 
$\pi:\mathcal{Z}\otimes\mathcal{Z}\to\mathcal{Z}^\infty$ such that 
\[
\pi(c\otimes1)=c\quad\forall c\in\mathcal{Z},\quad 
\alpha_g\circ\pi=\pi\circ(\alpha_g\otimes\id)\quad\forall g\in\Gamma. 
\]
Consequently, we obtain 
\[
\lVert[\pi(w),c]\rVert<\ep\quad\forall c\in F,\quad 
\lVert\alpha_a(\pi(w))-\pi(w)\rVert<\ep
\]
and 
\[
\lVert x-\pi(w)\alpha_b(\pi(w)^*)\rVert<\ep, 
\]
thereby completing the proof. 
\end{proof}

As a byproduct of Lemma \ref{CVonUHF1}, we get the following. 

\begin{proposition}\label{CVonUHF3}
Let $\alpha:\Gamma\curvearrowright B$ be a strongly outer action 
on a UHF algebra $B$ with a unique trace $\tau$. 
For an $\alpha$-cocycle $(x_g)_{g\in\Gamma}$ in $B$, 
the following conditions are equivalent. 
\begin{enumerate}
\item $\Delta_\tau(x_a)=0$. 
\item There exists a sequence $(w_n)_n$ of unitaries in $B$ such that 
$w_n\alpha_g(w_n^*)\to x_g$ as $n\to\infty$ for every $g\in\Gamma$. 
\end{enumerate}
\end{proposition}
\begin{proof}
We show (2)$\Rightarrow$(1). 
Suppose that there exists a unitary $w\in B$ such that 
$\lVert x_g-w\alpha_g(w^*)\rVert<1$ for each $g=a,b,a^{-1}$. 
Let $x'_g=w^*x_g\alpha_g(w)$. 
Then $(x'_g)_g$ is also an $\alpha$-cocycle, and so we have 
\[
x'_a\alpha_a(x'_b)=x'_b\alpha_b(x'_{a^{-1}}). 
\]
This, together with $x'_{a^{-1}}=\alpha_a^{-1}(x'_a)^*$, implies 
\[
\tau(\log x'_a)=\tau(\log x'_{a^{-1}})=-\tau(\log x'_a). 
\]
Thus $\tau(\log x'_a)=0$. 
Hence $\Delta_\tau(x_a)=\Delta_\tau(x'_a)=0$. 

Let us consider the other implication. 
Take $\ep>0$ arbitrarily. 
By applying Lemma \ref{CVonUHF1} to $\alpha:\Gamma\curvearrowright B$ 
with the empty set and $\ep/3$ in place of $F$ and $\ep$, 
we obtain $\delta>0$. 
Since the $\Z$-action generated by $\alpha_a$ has the Rohlin property, 
there exists a unitary $u\in B$ such that 
\[
\lVert x_a-u\alpha_a(u^*)\rVert<\min\{\delta/4,\ep/3\}. 
\]
Put $x'_g=u^*x_g\alpha_g(u)$ for all $g\in\Gamma$. 
From $x'_a\alpha_a(x'_b)=x'_b\alpha_b(x'_{a^{-1}})$, one obtains 
\[
\lVert\alpha_a(x'_b)-x'_b\rVert<\delta/2. 
\]
Besides, 
\begin{align*}
\tau(\log(x'_b\alpha_a(x'_b)^*))
&=\tau(\log(x'_a\alpha_a(x'_b)\alpha_b(x'_{a^{-1}})^*\alpha_a(x'_b)^*))\\
&=\tau(\log x'_a)+\tau(\log(\alpha_b(x'_{a^{-1}})^*))\\
&=2\tau(\log x'_a). 
\end{align*}
By $\Delta_\tau(x'_a)=\Delta_\tau(x_a)=0$, we have 
\[
r=\frac{1}{2\pi\sqrt{-1}}\tau(\log x'_a))\in\tau(K_0(B)). 
\]
It follows from \cite[Lemma 8.3]{MS} that 
there exists a unitary $v\in B$ such that 
\[
\lVert v-\alpha_a(v)\rVert<\min\{\delta/4,\ep/3\},\quad 
\tau(\log(v\alpha_a(v^*)))=2\pi\sqrt{-1}r. 
\]
Put $x''_g=v^*x'_g\alpha_g(v)$ for all $g\in\Gamma$. 
Then one has 
\[
\lVert\alpha_a(x''_b)-x''_b\rVert
=\lVert\alpha_a(v^*x'_b\alpha_b(v))-v^*x'_b\alpha_b(v)\rVert
<\delta/4+\delta/2+\delta/4=\delta
\]
and 
\begin{align*}
\tau(\log(x''_b\alpha_a(x''_b)^*))
&=\tau(\log(v^*x'_b\alpha_b(v)\alpha_a(\alpha_b(v^*))
\alpha_a(x'_b)^*\alpha_a(v)))\\
&=\tau(\log(x'_b\alpha_b(v)\alpha_a(\alpha_b(v^*))\alpha_a(x'_b)^*))
+\tau(\log(\alpha_a(v)v^*))\\
&=\tau(\log(v\alpha_{a^{-1}}(v^*)))+\tau(\log(\alpha_a(x'_b)^*x'_b))
+\tau(\log(\alpha_a(v)v^*))\\
&=\tau(\log(x'_b\alpha_a(x'_b)^*))-4\pi\sqrt{-1}r\\
&=0. 
\end{align*}
Therefore, by Lemma \ref{CVonUHF1}, we can find a unitary $w\in B$ such that 
\[
\lVert\alpha_a(w)-w\rVert<\ep/3
\quad\text{and}\quad \lVert x''_b-w\alpha_b(w^*)\rVert<\ep/3. 
\]
Consequently we get 
\[
\lVert x_a-uvw\alpha_a(w^*v^*u^*)\rVert
=\ep/3+\ep/3+\lVert x_a-u\alpha_a(u^*)\rVert<\ep
\]
and 
\[
\lVert x_b-uvw\alpha_b(w^*v^*u^*)\rVert
=\lVert x''_b-w\alpha_b(w^*)\rVert<\ep/3. 
\]
Since $\ep$ was arbitrary, (2) has been shown. 
\end{proof}

In the setting of the proposition above, 
the map $g\mapsto\Delta_\tau(x_g)$ is a homomorphism 
from $\Gamma$ to $\R/\tau(K_0(B))$. 
Therefore we always have $2\Delta_\tau(x_a)=0$ because of $bab^{-1}=a^{-1}$. 
Moreover, if $K_0(B)$ is 2-divisible 
(or equivalently, 
if $B$ absorbs the UHF algebra of type $2^\infty$ tensorially), 
then we always have $\Delta_\tau(x_a)=0$. 

\begin{theorem}\label{KleinUHF2}
Let $B$ be a UHF algebra. 
If $B$ absorbs the UHF algebra of type $2^\infty$ tensorially 
or $B$ does not contain a unital copy of $M_2$, then 
all strongly outer actions of $\Gamma$ on $B$ 
are strongly cocycle conjugate. 
\end{theorem}
\begin{proof}
By Theorem \ref{KleinUHF}, 
all strongly outer actions of $\Gamma$ on $B$ are cocycle conjugate. 
If $B$ absorbs the UHF algebra of type $2^\infty$ tensorially, 
then the conclusion follows from Proposition \ref{CVonUHF3} and 
the observation above. 
Suppose that $B$ does not contain a unital copy of $M_2$. 
Then $\{c\in\R/\tau(K_0(B))\mid 2c=0\}$ is isomorphic to $\Z/2\Z$ and 
$1/2+\tau(K_0(B))$ is the generator. 
Let $\alpha:\Gamma\curvearrowright B$ be a strongly outer action and 
let $(x_g)_g$ be an $\alpha$-cocycle in $B$ such that $\Delta_\tau(x_a)\neq0$. 
By the observation above, $\Delta_\tau(x_a)$ equals $1/2+\tau(K_0(B))$. 
For $g=a^nb^m\in\Gamma$, set $\zeta_g=(-1)^n\in\C\subset B$. 
It is easy to see that $(\zeta_gx_g)_g$ is an $\alpha$-cocycle satisfying 
$\Ad(\zeta_gx_g)=\Ad x_g$ and $\Delta_\tau(\zeta_ax_a)=0$. 
Hence the conclusion follows. 
\end{proof}

%%%%%%%%%%%%%%%%%%%%%%%%%%%%%%%%%%%%%%%%%%%%%%%%%%%%%%%%%%%%
\subsection{Actions of the Klein bottle group on the Jiang-Su algebra}

In this subsection we prove that 
all strongly outer cocycle actions of $\Gamma$ on $\mathcal{Z}$ 
are mutually cocycle conjugate (Theorem \ref{KleinZ}). 
We also show that 
all strongly outer actions of $\Gamma$ on $\mathcal{Z}$ 
are strongly cocycle conjugate to each other (Theorem \ref{KleinZ2}). 
The following is 
an analogue of Theorem \ref{KleinAF} for the Jiang-Su algebra. 

\begin{proposition}\label{preKleinZ}
Let $\phi\in\Aut(\mathcal{Z})$. 
For $i=1,2$, assume that 
$\psi_i\in\Aut(\mathcal{Z}\rtimes_\phi\Z)$ satisfies the following conditions. 
\begin{enumerate}
\item $\psi_i(\mathcal{Z})=\mathcal{Z}$ and 
$\psi_i(\lambda^\phi)\lambda^\phi\in\mathcal{Z}$. 
\item $\phi^n\circ(\psi_i|\mathcal{Z})^m\in\Aut(\mathcal{Z})$ is 
not weakly inner for each $(n,m)\in\Z^2\setminus\{(0,0)\}$. 
\end{enumerate}
Assume further that $\psi_1\circ\psi_2^{-1}$ is approximately inner. 
Then there exists an approximately inner automorphism 
$\mu\in\Aut_\T(\mathcal{Z}\rtimes_\phi\Z)$ and 
a unitary $v\in U(\mathcal{Z})$ such that 
\[
\mu\circ\psi_1\circ\mu^{-1}=\Ad v\circ\psi_2. 
\]
\end{proposition}
\begin{proof}
We denote the unique trace on $\mathcal{Z}$ by $\tau$. 
It is easy to see that 
$\psi_1\circ\psi_2^{-1}$ is in $\Aut_\T(\mathcal{Z}\rtimes_\phi\Z)$. 
By Proposition \ref{examples} (3), 
(the $\Z$-action generated by) $\phi$ is asymptotically representable. 
Then \cite[Theorem 4.8]{IM} implies that 
$\psi_1\circ\psi_2^{-1}$ is $\T$-approximately inner, that is, 
for any finite subset $F\subset\mathcal{Z}\rtimes_\phi\Z$ and $\ep>0$, 
there exists a unitary $x\in\mathcal{Z}$ such that 
$\lVert\psi_1(c)-x\psi_2(c)x^*\rVert<\ep$ for all $c\in F$, 
and we may also assume $\Delta_\tau(x)=0$ 
by replacing $x$ with its suitable scalar multiple. 
Moreover, for each $i=1,2$, the cocycle action of $\Gamma$ on $\mathcal{Z}$ 
generated by $\phi$ and $\psi_i|\mathcal{Z}$ is strongly outer. 
Hence, by virtue of Proposition \ref{CVonZ}, 
we have the following for each $i=1,2$: 
if a sequence $(x_n)_n$ of unitaries in $\mathcal{Z}$ is 
a central sequence in $\mathcal{Z}\rtimes_\phi\Z$ and $\Delta_\tau(x_n)=0$, 
then there exists a sequence $(y_n)_n$ of unitaries in $\mathcal{Z}$ 
such that $(y_n)_n$ is a central sequence in $\mathcal{Z}\rtimes_\phi\Z$ 
and satisfies $x_n-y_n\psi_i(y_n^*)\to0$ as $n\to\infty$. 
Then the Evans-Kishimoto intertwining argument shows the statement 
(see \cite[Theorem 5.1]{K98JFA} for example). 
\end{proof}

\begin{lemma}
Let $\phi$ be an automorphism of $\mathcal{Z}$ such that 
$\phi^n$ is not weakly inner for every $n\in\N$. 
Then an automorphism $\psi\in\Aut_\T(\mathcal{Z}\rtimes_\phi\Z)$ is 
approximately inner 
if and only if $\Delta_\tau(\psi(\lambda^\phi)(\lambda^\phi)^*)=0$, 
where $\tau$ denotes the tracial state on $\mathcal{Z}$. 
\end{lemma}
\begin{proof}
Let $A=\mathcal{Z}\rtimes_\phi\Z$. 
By Corollary \ref{so>Z}, $A$ is $\mathcal{Z}$-stable. 
It is well-known that for any UHF algebra $B$, 
$A\otimes B\cong(\mathcal{Z}\otimes B)\rtimes_{\phi\otimes\id}\Z$ is 
a unital simple AT algebra with real rank zero 
(see \cite[Theorem 1.3]{K95crelle}). 
Therefore one can apply \cite[Theorem 7.1]{M11JFA} to automorphisms of $A$. 

Let $\psi\in\Aut_\T(A)$. 
By the Pimsner-Voiculescu exact sequence, $K_0(A)\cong K_1(A)\cong\Z$. 
Clearly 
$K_0(\psi)([1])=[1]$ and $K_1(\psi)([\lambda^\phi])=[\lambda^\phi]$. 
It follows that $K_i(\psi)$ is the identity for each $i=0,1$, 
and hence $KL(\psi)=KL(\id)$. 
Let $\Theta_{\psi,\id}:K_1(A)\to\Aff(T(A))/\Ima D_A\cong\R/\Z$ be 
the homomorphism described in \cite[Section 3]{M11JFA}. 
Then, by definition, $\Theta_{\psi,\id}([\lambda^\phi])$ is 
equal to $\Delta_\tau(\psi(\lambda^\phi)(\lambda^\phi)^*)$, 
and so the conclusion follows from \cite[Theorem 7.1]{M11JFA}. 
\end{proof}

\begin{theorem}\label{KleinZ}
Let $(\alpha,u)$ and $(\beta,v)$ be strongly outer cocycle actions 
of $\Gamma$ on the Jiang-Su algebra $\mathcal{Z}$. 
Then $(\alpha,u)$ and $(\beta,v)$ are cocycle conjugate. 
\end{theorem}
\begin{proof}
Put $\check u=u(b,a)u(a^{-1},b)^*u(a^{-1},a)$ and 
$\check v=v(b,a)v(a^{-1},b)^*v(a^{-1},a)$. 
The automorphism $\alpha_b$ extends to 
$\tilde\alpha_b\in\Aut(\mathcal{Z}\rtimes_{\alpha_a}\Z)$ by letting 
$\tilde\alpha_b(\lambda^{\alpha_a})=\check u(\lambda^{\alpha_a})^*$. 
Likewise, we can extend $\beta_b$ 
to $\tilde\beta_b\in\Aut(\mathcal{Z}\rtimes_{\beta_a}\Z)$. 
By \cite[Theorem 1.3]{S10JFA}, we may assume that 
there exists $w\in U(\mathcal{Z})$ such that $\beta_a=\Ad w\circ\alpha_a$. 
Define an isomorphism 
$\theta:\mathcal{Z}\rtimes_{\beta_a}\Z\to\mathcal{Z}\rtimes_{\alpha_a}\Z$ 
by $\theta(x)=x$ for $x\in\mathcal{Z}$ and 
$\theta(\lambda^{\beta_a})=w\lambda^{\alpha_a}$. 
Then $\theta\circ\tilde\beta_b\circ\theta^{-1}\circ\tilde\alpha_b^{-1}$ is 
in $\Aut_\T(\mathcal{Z}\rtimes_{\alpha_a}\Z)$. 
Let $\zeta\in\C$ be such that $\lvert\zeta\rvert=1$ and 
\[
\Delta_\tau(\zeta^{-2}1)=\Delta_\tau(
(\theta\circ\tilde\beta_b\circ\theta^{-1}\circ\tilde\alpha_b^{-1})
(\lambda^{\alpha_a})(\lambda^{\alpha_a})^*). 
\]
Define $\psi\in\Aut_\T(\mathcal{Z}\rtimes_{\alpha_a}\Z)$ 
by $\psi(x)=x$ for $x\in\mathcal{Z}$ and 
$\psi(\lambda^{\alpha_a})=\zeta\lambda^{\alpha_a}$. 
Then 
\[
(\psi\circ\theta\circ\tilde\beta_b\circ\theta^{-1}\circ\psi^{-1}
\circ\tilde\alpha_b^{-1})(\lambda^{\alpha_a})
=\zeta^2(\theta\circ\tilde\beta_b\circ\theta^{-1}\circ\tilde\alpha_b^{-1})
(\lambda^{\alpha_a}), 
\]
and so 
\begin{align*}
&\Delta_\tau((\psi\circ\theta\circ\tilde\beta_b\circ\theta^{-1}\circ\psi^{-1}
\circ\tilde\alpha_b^{-1})(\lambda^{\alpha_a})(\lambda^{\alpha_a})^*)\\
&=\Delta_\tau(\zeta^21)+\Delta_\tau(
(\theta\circ\tilde\beta_b\circ\theta^{-1}\circ\tilde\alpha_b^{-1})
(\lambda^{\alpha_a})(\lambda^{\alpha_a})^*)\\
&=0. 
\end{align*}
It follows from the lemma above that 
$\psi\circ\theta\circ\tilde\beta_b\circ\theta^{-1}\circ\psi^{-1}
\circ\tilde\alpha_b^{-1}$ is approximately inner. 
Therefore, Proposition \ref{preKleinZ} applies to the automorphisms 
$\psi\circ\theta\circ\tilde\beta_b\circ\theta^{-1}\circ\psi^{-1}$ 
and $\alpha_b$, and 
yields an approximately inner automorphism 
$\mu\in\Aut_\T(\mathcal{Z}\rtimes_{\alpha_a}\Z)$ and 
a unitary $t\in U(\mathcal{Z})$ such that 
\[
\mu\circ\psi\circ\theta\circ\tilde\beta_b\circ\theta^{-1}\circ\psi^{-1}
\circ\mu^{-1}=\Ad t\circ\tilde\alpha_b. 
\]
In the same way as the proof of Theorem \ref{KleinAF}, 
we can conclude that $(\alpha,u)$ and $(\beta,v)$ are cocycle conjugate. 
\end{proof}

\begin{corollary}
Any strongly outer cocycle action of $\Gamma$ 
on the Jiang-Su algebra $\mathcal{Z}$ is approximately representable. 
\end{corollary}
\begin{proof}
This immediately follows from the theorem above and Proposition \ref{exist}. 
\end{proof}

Finally, we would like to prove that 
all strongly outer actions of $\Gamma$ on $\mathcal{Z}$ are 
mutually strongly cocycle conjugate (Theorem \ref{KleinZ2}). 

\begin{proposition}
Let $\alpha:\Gamma\curvearrowright\mathcal{Z}$ be a strongly outer action. 
For two $\alpha$-cocycles $(x_g)_g$ and $(y_g)_g$ in $\mathcal{Z}$, 
the following conditions are equivalent. 
\begin{enumerate}
\item $\Delta_\tau(x_g)=\Delta_\tau(y_g)$ for all $g\in\Gamma$, 
where $\tau$ is the unique tracial state on $\mathcal{Z}$. 
\item There exists a sequence $(w_n)_n$ of unitaries in $\mathcal{Z}$ 
such that 
$w_nx_g\alpha_g(w_n^*)\to y_g$ as $n\to\infty$ for every $g\in\Gamma$. 
\end{enumerate}
\end{proposition}
\begin{proof}
(2)$\Rightarrow$(1) is obvious. 
We prove the other implication. 

Let $\alpha^\otimes:\Gamma\curvearrowright\mathcal{Z}^\otimes$ be 
the two-sided infinite tensor product of 
$\alpha:\Gamma\curvearrowright\mathcal{Z}$, that is, 
\[
\mathcal{Z}^\otimes=\bigotimes_\Z\mathcal{Z},\quad 
\alpha^\otimes_g=\bigotimes_\Z\alpha_g\quad\forall g\in\Gamma. 
\]
First, we claim that the assertion holds for $\alpha^\otimes$. 
Let $(x_g)_g$ be an $\alpha^\otimes$-cocycle in $\mathcal{Z}^\otimes$. 
Let $\zeta_g\in\C\subset\mathcal{Z}^\otimes$ be the unitary 
satisfying $\Delta_\tau(x_g)=\Delta_\tau(\zeta_g)$. 
Then $(\zeta_g)_g$ is clearly an $\alpha^\otimes$-cocycle, and so 
it suffices to show that 
there exists a sequence $(w_n)_n$ of unitaries in $\mathcal{Z}$ 
such that 
$w_nx_g\alpha_g(w_n^*)\to\zeta_g$ as $n\to\infty$. 
Note that $g\mapsto\zeta_g$ is a homomorphism. 
In particular, $\zeta_a=1$ or $\zeta_a=-1$. 

Take $\ep>0$ arbitrarily. 
Apply Proposition \ref{CVonZ} to 
$\alpha^\otimes:\Gamma\curvearrowright\mathcal{Z}^\otimes$ 
with the empty set and $\ep/2$ in place of $F$ and $\ep>0$ 
to obtain a finite subset $G\subset\mathcal{Z}^\otimes$ and $\delta>0$. 
By \cite[Theorem 5.3]{S10JFA}, 
there exists a unitary $y\in\mathcal{Z}^\otimes$ such that 
\[
\lVert\zeta_a^{-1}x_a-y^*\alpha^\otimes_a(y)\rVert<\min\{\delta/2,\ep/2\}, 
\]
and hence 
\[
\lVert yx_a\alpha^\otimes_a(y^*)-\zeta_a\rVert<\min\{\delta/2,\ep/2\}. 
\]
From $x_a\alpha^\otimes_a(x_{a^{-1}})=1$, it is easy to see 
\[
\lVert yx_{a^{-1}}\alpha^\otimes_{a^{-1}}(y^*)-\zeta_a^{-1}\rVert
<\min\{\delta/2,\ep/2\}. 
\]
Therefore, 
defining an $\alpha^\otimes$-cocycle $(x'_g)_g$ 
by $x'_g=yx_g\alpha^\otimes_g(y^*)$ for $g\in\Gamma$, 
one gets 
\[
\lVert\alpha^\otimes_a(x'_b)-x'_b\rVert
=\lVert\zeta_a\alpha^\otimes_a(x'_b)-\zeta_a^{-1}x'_b\rVert
<\lVert x'_a\alpha^\otimes_a(x'_b)-x'_b\alpha^\otimes_b(x'_{a^{-1}})\rVert
+\delta=\delta. 
\]
Let $\sigma\in\Aut(\mathcal{Z}^\otimes)$ be the shift automorphism. 
Then we have 
$\sigma\circ\alpha^\otimes_g=\alpha^\otimes_g\circ\sigma$ for any $g\in\Gamma$ 
and $[\sigma^n(s),t]\to0$ as $n\to\infty$ for any $s,t\in\mathcal{Z}^\otimes$. 
Choose $n\in\N$ so that 
\[
\lVert[\sigma^n(x'_b),c]\rVert<\delta\quad\forall c\in G. 
\]
Evidently $\lVert\alpha^\otimes_a(\sigma^n(x'_b))-\sigma^n(x'_b)\rVert
=\lVert\sigma^n(\alpha^\otimes_a(x'_b))-\sigma^n(x'_b)\rVert<\delta$. 
It follows from Lemma \ref{CVonZ} that 
there exists $z\in U(\mathcal{Z}^\otimes)$ such that 
\[
\lVert\alpha^\otimes_a(z)-z\rVert<\ep/2\quad\text{and}\quad 
\lVert\zeta_b^{-1}\sigma^n(x'_b)-z^*\alpha^\otimes_b(z)\rVert<\ep/2. 
\]
Put $w=\sigma^{-n}(z)y$. 
Then we obtain 
\begin{align*}
\lVert wx_a\alpha^\otimes_a(w^*)-\zeta_a\rVert
&=\lVert\sigma^{-n}(z)x'_a\alpha^\otimes_a(\sigma^{-n}(z^*))-\zeta_a\rVert\\
&=\lVert\sigma^{-n}(z)x'_a\sigma^{-n}(\alpha^\otimes_a(z^*))-\zeta_a\rVert\\
&<\lVert\sigma^{-n}(z)\sigma^{-n}(\alpha^\otimes_a(z^*))-1\rVert+\ep/2\\
&<\ep
\end{align*}
and 
\begin{align*}
\lVert wx_b\alpha^\otimes_b(w^*)-\zeta_b\rVert
&=\lVert\sigma^{-n}(z)x'_b\alpha^\otimes_b(\sigma^{-n}(z^*))-\zeta_b\rVert\\
&=\lVert z\sigma^n(x'_b)\alpha^\otimes_b(z^*)-\zeta_b\rVert\\
&<\ep/2. 
\end{align*}
Since $\ep>0$ was arbitrary, the proof of the claim is completed. 

Now we consider $\alpha$-cocycles $(x_g)_g$ and $(y_g)_g$ in $\mathcal{Z}$ 
satisfying $\Delta_\tau(x_g)=\Delta_\tau(y_g)$ for all $g\in\Gamma$. 
By Theorem \ref{KleinZ}, $\alpha$ is cocycle conjugate to $\alpha^\otimes$. 
Thus, there exist an isomorphism $\theta:\mathcal{Z}\to\mathcal{Z}^\otimes$ 
and an $\alpha^\otimes$-cocycle $(u_g)_g$ in $\mathcal{Z}^\otimes$ such that 
\[
\theta\circ\alpha_g\circ\theta^{-1}=\Ad u_g\circ\alpha^\otimes_g
\quad\forall g\in \Gamma. 
\]
Then $(\theta(x_g)u_g)_g$ and $(\theta(y_g)u_g)_g$ are 
$\alpha^\otimes$-cocycles 
satisfying $\Delta_\tau(\theta(x_g)u_g)=\Delta_\tau(\theta(y_g)u_g)$. 
By the claim above, 
we can find a sequence $(w_n)_n$ of unitaries in $\mathcal{Z}^\otimes$ 
such that 
\[
w_n\theta(x_g)u_g\alpha^\otimes_g(w_n^*)\to\theta(y_g)u_g
\quad\forall g\in\Gamma, 
\]
and whence 
\begin{align*}
\theta^{-1}(w_n)x_g\alpha_g(\theta^{-1}(w_n^*))
&=\theta^{-1}(w_n)x_g\theta^{-1}(u_g\alpha^\otimes_g(w_n^*)u_g^*)\\
&=\theta^{-1}(w_n\theta(x_g)u_g\alpha^\otimes_g(w_n^*)u_g^*)\\
&\to y_g
\end{align*}
for all $g\in\Gamma$ as desired. 
\end{proof}

\begin{theorem}\label{KleinZ2}
Any two strongly outer actions $\alpha:\Gamma\curvearrowright\mathcal{Z}$ and 
$\beta:\Gamma\curvearrowright\mathcal{Z}$ are strongly cocycle conjugate. 
\end{theorem}
\begin{proof}
By Theorem \ref{KleinZ}, $\alpha$ and $\beta$ are cocycle conjugate, 
that is, there exist an isomorphism $\theta:\mathcal{Z}\to\mathcal{Z}$ 
and a $\beta$-cocycle $(x_g)_g$ in $\mathcal{Z}$ such that 
\[
\theta\circ\alpha_g\circ\theta^{-1}=\Ad x_g\circ\beta_g
\quad\forall g\in \Gamma. 
\]
As in the proof of the proposition above, 
let $\zeta_g\in\C\subset\mathcal{Z}$ be the unitary 
satisfying $\Delta_\tau(x_g)=\Delta_\tau(\zeta_g)$. 
Then $(\zeta_g^{-1}x_g)_g$ is also a $\beta$-cocycle. 
Hence, by replacing $x_g$ with $\zeta_g^{-1}x_g$, 
we may assume $\Delta_\tau(x_g)=0$. 
It follows from the proposition above that 
there exists $(w_n)_n$ in $U(\mathcal{Z})$ such that 
$w_n\beta_g(w_n^*)\to x_g$ as $n\to\infty$ for all $g\in\Gamma$, 
which implies that $\alpha$ is strongly cocycle conjugate to $\beta$. 
\end{proof}

\noindent
\textbf{Acknowledgement. }
The authors would like to thank 
Toshihiko Masuda, Narutaka Ozawa and Reiji Tomatsu
for several helpful conversations and valuable comments.

\end{document}